\documentclass[sn-mathphys-num]{sn-jnl}

\usepackage{graphicx}%
\usepackage{multirow}%
\usepackage{amsmath,amssymb,amsfonts}%
\usepackage{amsthm}%
\usepackage{mathrsfs}%
\usepackage[title]{appendix}%
\usepackage{xcolor}%
\usepackage{textcomp}%
\usepackage{manyfoot}%
\usepackage{booktabs}%
\usepackage{algorithm}%
\usepackage{algorithmicx}%
\usepackage{algpseudocode}%
\usepackage{listings}%
\usepackage{yhmath}%

\hypersetup{
  pdftitle={Approximating the signature of Brownian motion for high order SDE simulation},
  pdfauthor={James Foster}
}

\newtheorem{theorem}{Theorem}
\newtheorem{corollary}{Corollary}
\newtheorem{remark}{Remark}%

\newtheorem{definition}{Definition}%

\newcommand{\E}{\mathbb{E}}
\newcommand{\R}{\mathbb{R}}
\renewcommand{\P}{\mathbb{P}}
\newcommand{\m}{\hspace{0.25mm}}
\newcommand{\mm}{\hspace{5mm}}
\newcommand{\mmm}{\hspace{10mm}}
\newcommand{\mmmm}{\hspace{15mm}}

\newcommand{\var}{\mathrm{Var}}

\newcommand{\sgn}{\mathrm{sgn}}

\def\shuffle{\sqcup\mathchoice{\mkern-7mu}{\mkern-7mu}{\mkern-3.5mu}{\mkern-3.5mu}\sqcup\hspace{0.75mm}}
\def\textshuffle{\sqcup\mathchoice{\mkern-3mu}{\mkern-3mu}{\mkern-1mu}{\mkern-1mu}\sqcup}
\def\subshuffle{\hspace{0.25mm}\sqcup\hspace*{-0.55mm}\sqcup\hspace{0.25mm}}

\begin{document}

\title[Approximating the signature of Brownian motion for high order SDE simulation]{Approximating the signature of Brownian motion for high order SDE simulation}


\author{\fnm{James} \sur{Foster}}\email{jmf68@bath.ac.uk}

\affil{\orgdiv{Department of Mathematical Sciences}, \orgname{University of Bath}, \orgaddress{\country{UK}}}

\abstract{The signature is a collection of iterated integrals describing the ``shape'' of a path. It appears naturally in the Taylor expansions of controlled differential equations and, as a consequence, is arguably the central object within rough path theory. In this paper, we will consider the signature of Brownian motion with time, and present both new and recently developed approximations for some of its integrals.
Since these integrals (or equivalent L\'{e}vy areas) are nonlinear functions of the Brownian path, they are not Gaussian and known to be challenging to simulate.
To conclude the paper, we will present some applications of these approximations to the high order numerical simulation of stochastic differential equations (SDEs).}



\maketitle

\section{Introduction}\label{sect:introduction}

Since its development at the turn of the 20th century \cite{bachelier1900bm, einstein1905bm}, Brownian motion has seen widespread application for modelling real-world time-evolving random phenomena \cite{mazo2002bm}.
In particular, Brownian motion $W = \{W_t\}_{t\geq 0}$ is commonly used as the source of continuous-time random noise for stochastic differential equations (SDEs) of the form
\begin{align}\label{eq:intro_sde}
dy_t = f(y_t)\m dt + \sum_{i=1}^d g_i(y_t)\m dW_t^i\m,
\end{align}
where $W = (W_1\m, \cdots, W_d)$ is a $d$-dimensional Brownian motion and $f, g_i : \R^e \rightarrow \R^e$ denote vector fields on $\R^e$ that are  sufficiently regular so that a solution to (\ref{eq:intro_sde}) exists.
The SDE (\ref{eq:intro_sde}) can understood using either It\^{o} or Stratonovich stochastic integration and we refer the reader to \cite{oksendal2010sdes} for an introduction to stochastic calculus and SDE theory.\smallbreak

Given an SDE, it is natural to understand the well-posedness of the solution map:
\begin{align}\label{eq:ito_map}
\big(y_0\m, W\big) \mapsto \{y_t\}_{t\geq 0}\m.
\end{align}
For much of the 20th century, the continuity of (\ref{eq:ito_map}) remained a mystery when the underlying Brownian motion was multidimensional and $g = (g_1,\cdots, g_d)$ was generic.
However, this was eventually resolved in 1998 by Lyons' theory of rough paths \cite{lyons1998roughpaths} which showed that the aptly named ``It\^{o}-Lyons'' map (\ref{eq:ito_map}) does become continuous when the driving Brownian path is defined along with its (second) iterated integrals.\smallbreak

Rough path theory would then have a notable impact in stochastic analysis \cite{frizhairer2020roughpaths, frizvictoir2010roughpaths} as well as data science (where, the driving signal $X$ comes from sequential data \cite{kidger2020ncdes});
and this collection of iterated integrals would become known as the ``signature'' \cite{hambly2010roughpaths, morrill2021nrdes}.\medbreak

\begin{definition}[Signature of a path {\cite{hambly2010roughpaths, morrill2021nrdes}}]
Let $X:[0,T]\rightarrow \R^d$ be a continuous path. Then, for $N\geq 1$, the depth-$N$ signature of $X$ over the interval $[a,b]$ is defined as
\begin{align*}
S_{a,b}(X) = \Big(1, \big\{S_{a,b}^{\m i}(X)\big\}_{1\m\leq\m i\m\leq\m d}\m, \big\{S_{a,b}^{\m i\m,j}(X)\big\}_{1\m\leq\m i,\m j\m\leq\m d}\m, \cdots, \big\{S_{a,b}^{\m i_1\m,\m\cdots,i_N}(X)\big\}_{1\m\leq\m i_1\m,\m\cdots,\m i_N\m\leq\m d}\Big),
\end{align*}
where each term $\m S_{a,b}^{\m i_1\m,\cdots,i_n}(X)$ is given by
\begin{align}\label{eq:intro_iter_integrals}
S_{a,b}^{\m i_1\m,\cdots,i_n}(X) := \idotsint\displaylimits_{a\m\leq\m u_1\m\leq\m\cdots\m\leq\m u_n\m\leq\m b} dX_{u_1}^{i_1}\cdots dX_{u_n}^{i_n}\m.
\end{align}
\end{definition}
\begin{remark}
When $X$ has bounded variation (or length), then the iterated integrals (\ref{eq:intro_iter_integrals}) can be defined by Riemann-Stieltjes integration. However, this is no longer the case for Brownian motion and (\ref{eq:intro_iter_integrals}) is usually defined by either It\^{o} or Stratonovich integration.
\end{remark}\medbreak

The signature of Brownian motion is especially important in the study of numerical methods for SDEs since iterated integrals naturally appear within Taylor expansions (we give Theorem \ref{thm:intro_taylor} as such an example). As a consequence, numerical methods that use iterated integrals can achieve higher order rates of convergence. On the other hand, Clark and Cameron \cite{clark1980convergence} showed that methods which use only values of the Brownian path are limited to a strong convergence rate of just $O(\sqrt{h}\m)$, where $h$ is the step size.\medbreak

\begin{theorem}[Stochastic It\^{o}-Taylor expansion {\cite[Equation (2.10) and Lemma 2.2.]{milstein2004numerics}}]\label{thm:intro_taylor} Consider the It\^{o} SDE (\ref{eq:intro_sde}) where the vector fields $f, g_i : \R^e \rightarrow\R^e$ are assumed to be sufficiently smooth and with linear growth (see condition (2.17) in \cite[Lemma 2.2.]{milstein2004numerics}).
Then the solution $y$ of (\ref{eq:intro_sde}) can be expressed over the interval $[s,t]$, with $0\leq s\leq t$, as
\begin{align}\label{eq:intro_taylor}
y_t & = y_s + f(y_s)\m h + \sum_{i=1}^d g_i(y_s)W_{s,t}^i + \sum_{i,j=1}^d g_j^\prime(y_s) g_i(y_s)\int_s^t\int_s^u dW_v^i\, dW_u^j + R_{s,t}\,,
\end{align}
where $h := t - s$, $W_{s,t}^{i} := W_t^i - W_s^i\m$ and the remainder term $R_{s,t}$ is given by
\begin{align*}
R_{s,t} & := \int_s^t \big(f(y_u) - f(y_s)\big)\m du + \sum_{i,j=1}^d \int_s^t \int_s^u \big(\m g_j^\prime(y_v)g_i(y_v) - g_j^\prime(y_s)g_i(y_s)\big)\m dW_v^i\, dW_u^j \\
&\mm + \sum_{i=1}^d\int_s^t \int_s^u g_i^\prime(y_v)f(y_v)\, dv\, dW_u^i + \frac{1}{2}\sum_{i=1}^d \int_s^t\int_s^u \Delta g_i(y_v)\, dv\, dW_u^i\,,
\end{align*}
with $\,\E\big[\|R_{s,t}\|^2\big]^\frac{1}{2} \leq C\big(1 + \E\big[\|y_s\|^2\big]\big)^\frac{1}{2} h^\frac{3}{2}$ for some constant $C$ depending on $f$ and $g_i\m$.
\end{theorem}\bigbreak
\begin{remark}\label{rmk:commute}
When the vector fields $g := (g_1,\cdots, g_d)$ satisfy a commutativity condition,
\begin{align}\label{eq:intro_commute}
\hspace*{15mm} g_i^\prime(y) g_j(y) = g_j^\prime(y) g_i(y),\hspace{10mm}\forall y\in\R^e,
\end{align}
for $i,j\in\{1,\cdots, d\}$, it follows by It\^{o}'s lemma that the Taylor expansion (\ref{eq:intro_taylor}) becomes
\begin{align*}
y_t & = y_s + f(y_s)\m h + \sum_{i=1}^d g_i(y_s)W_{s,t}^i + \frac{1}{2}\sum_{i,j=1}^d g_j^\prime(y_s) g_i(y_s)\big(W_{s,t}^i W_{s,t}^j - \delta_{ij}h\big) + R_{s,t}\,,
\end{align*}
where $\delta_{ij}$ is the Kronecker delta. In this case, we can further expand $R_{s,t}$ in terms of other integrals in the signature of Brownian motion coupled with time $\overline{W}_t := \{(t, W_t)\}$.
In particular, the integrals $\big\{\int_{t_k}^{t_{k+1}} W_{t_k\m,u}\, du\big\}$ can be used alongside path increments $\{W_{t_k\m,t_{k+1}}\}$ by numerical methods to achieve faster rates of strong convergence \cite{foster2024splitting, tang2019sdes}.
\end{remark}\bigbreak

In this paper, we will consider the signature of ``space-time'' Brownian motion $\overline{W} := \{(t, W_t)\}_{0\m\leq\m t\m\leq\m T}\m,$ due to its importance within the numerical analysis of SDEs.
We shall present some new and recent approximations for the non-Gaussian integrals (\ref{eq:intro_integral1}) and (\ref{eq:intro_integral2}) which can then be used by SDE solvers to achieve high order convergence.
\noindent\begin{minipage}{.5\linewidth}
\begin{align}\label{eq:intro_integral1}
\int_s^t \int_s^u  dW_v^i\, dW_u^j\m, 
\end{align}
\end{minipage}%
\begin{minipage}{.5\linewidth}
\begin{align}\label{eq:intro_integral2}
\int_s^t W_{s,u}^i W_{s,u}^j\m du, 
\end{align}
\end{minipage}\bigbreak\noindent
where $i,j\in\{1,\cdots, d\}$ and $W_{s,u} := W_u - W_s\m$. As indicated in Remark \ref{rmk:commute}, the iterated integral (\ref{eq:intro_integral1}) is important for general SDEs whereas (\ref{eq:intro_integral2}) is more relevant to SDEs where the noise is scalar ($d=1$), additive ($g^\prime = 0$) or otherwise commutative (equation (\ref{eq:intro_commute})).
\medbreak

Unsurprisingly, these integrals and their approximation theory are well studied.
For example, due to It\^{o}'s lemma (or integration by parts in the Stratovonich setting), iterated integrals of Brownian motion have a well-known algebraic structure \cite{gaines1994integrals, li1994integrals}.
In particular, once the increment $W_{s,t}$ has been generated, the integral (\ref{eq:intro_integral1}) can be expressed in terms of a ``L\'{e}vy area'', which was introduced by Paul L\'{e}vy in 1940 \cite{levy1940area}.
Nevertheless, simulating such integrals is a challenging problem and an algorithm for exactly generating L\'{e}vy area (given a path increment) is only known for $d=2$ \cite{gaines1994area}.
Moreover, it was shown that any approximation using $N$ Gaussian random variables (obtained by linear maps on the Brownian path) converges slowly at rate $O(N^{-\frac{1}{2}})$ \cite{dickinson2007integrals}.
\medbreak
\begin{definition}[L\'{e}vy area] Over an interval $[s,t]$, we will define the (space-space) L\'{e}vy area of Brownian motion as a $d\times d$ matrix $A_{s,t}$ whose $(i,j)$-th entry is given by
\begin{align}\label{eq:intro_levy_area}
A_{s,t}^{i,j} := \frac{1}{2}\bigg(\int_s^t W_{s,u}^i\, dW_u^j - \int_s^t W_{s,u}^j\m dW_u^i\bigg),
\end{align}
where $W_{s,u} := W_u - W_s\m$.
\end{definition}\vspace*{-3.75mm}
\begin{figure}[ht]
\centering
\includegraphics[width=0.75\textwidth]{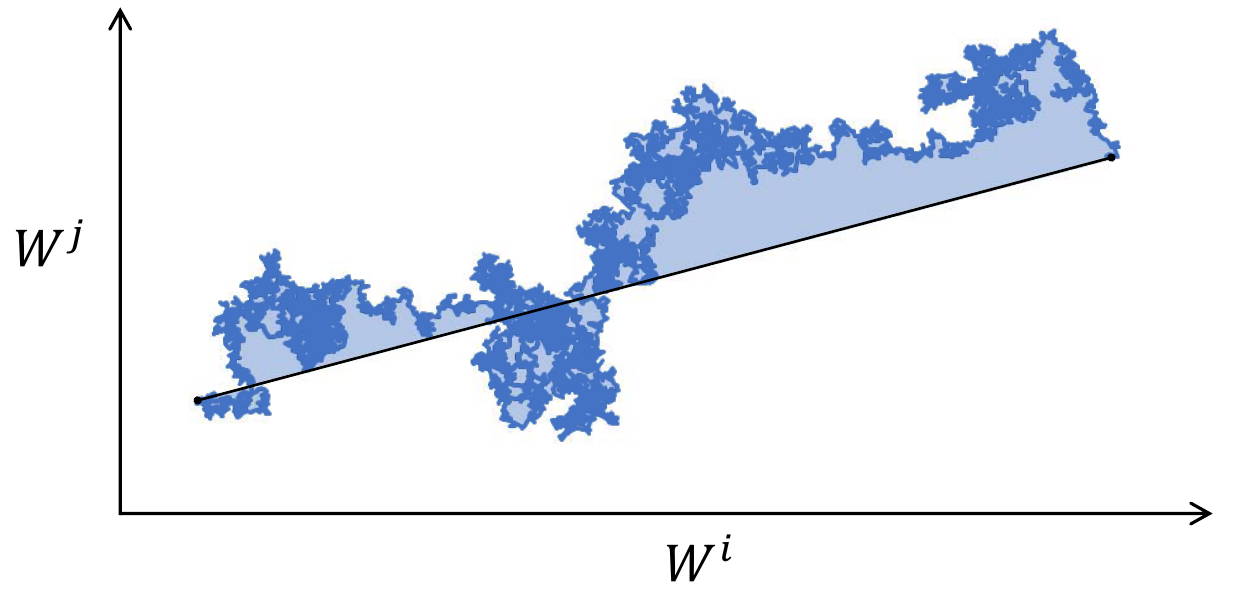}
\caption{The (space-space) L\'{e}vy area is the chordal area enclosed by independent Brownian motions (diagram taken from \cite{foster2020thesis}).}\label{fig:space_space_levy}
\end{figure}\vspace*{-2mm}

Despite this slow convergence, several strong approximations have been proposed \cite{foster2023levyarea, kloedenplaten1992book, milstein1994book, milstein1988book, kloeden1992platenwright, kuznetsov1997area1, kuznetsov1997area2} and we refer the reader to \cite[Table 2]{foster2023levyarea} for a summary of these approaches.
By including additional Gaussian random variables so that approximations have the same covariance as L\'{e}vy area and by constructing non-trivial probabilistic couplings, it is possible to achieve faster $O(N^{-1})$ convergence rates in the 2-Wasserstein metric.
Examples of L\'{e}vy area approximations that allow for this analysis include those of Wiktorsson \cite{wiktorsson2001area}, Ryd\'{e}n \& Wiktorsson \cite{ryden2001area}, Davie \cite{davie2014area, flint20216area}, Mrongowius \& R\"{o}\ss{}ler \cite{mrongowius2022area} and Foster \cite[Chapter 7]{foster2020thesis}. When $d=2$, Malham \& Wiese \cite{malham2014area} give a highly accurate L\'{e}vy area approximation based on an expansion involving logistic random variables. However, there have also been papers introducing high order methodologies for SDEs (including techniques for variance reduction) which avoid L\'{e}vy area generation \cite{cruzeiro2004noarea, giles2014anti, iguchi2025anti}.\medbreak

Recently, a machine learning (ML) approach was taken to train a neural network to generate samples that are close to Brownian L\'{e}vy area in a Wasserstein sense \cite{jelincic2025levygan}. While the ML model empirically outperformed the standard methods discussed above, it is a ``black box'' and thus unlikely to have a theoretical $O(N^{-1})$ convergence rate.
However, the experiments in \cite{jelincic2025levygan} also revealed that the weak approximation proposed in \cite[Definition 7.3.5]{foster2020thesis}, which aims to match the first four moments of L\'{e}vy area, gave similar 2-Wasserstein and maximum mean discrepancy errors as the ML model. 
Therefore, in section \ref{sect:space_space}, we will particularly focus on this moment-based approximation.\medbreak

When the noise vector fields $\{g_1\m,\cdots, g_d\}$ satisfy the commutativity condition (\ref{eq:intro_commute}), the  (space-space) L\'{e}vy area will no longer appear in the Taylor expansion of the SDE. In this case, it becomes more important to study approximations of the integral (\ref{eq:intro_integral2}).
Starting with the work of Clark \cite{clark1982efficient} as well as Newton \cite{newton1991efficient} and Castell \& Gaines \cite{castell1996efficient}, the concept of ``asymptotically efficient'' numerical methods for SDEs was developed.

Methods are asymptotically efficient if they use the following approximation of (\ref{eq:intro_integral2}),
\begin{align}\label{eq:asymptotically_efficient}
\E\bigg[ \int_s^t W_{s,u}^iW_{s,u}^j du\,\Big|\m W_{s,t}\bigg] = \frac{1}{3}\m h\m W_{s,t}^{i}W_{s,t}^{j} + \frac{1}{6}h^2\delta_{ij}\m.
\end{align}
Since the above is a conditional expectation, it follows that the right-hand side of (\ref{eq:asymptotically_efficient}) gives the least $L^2(\P)$ error among the $W_{s,t}\m$-measurable estimators of the integral (\ref{eq:intro_integral2}).
However, despite this notion of optimality, Tang \& Xiao \cite{tang2019sdes} developed an estimator for (\ref{eq:intro_integral2}) which is ``asymptotically optimal'' (i.e.~approaches an optimal rate as $N\rightarrow\infty$).
Foster \cite{foster2020thesis, foster2020poly} argued that asymptotically efficient methods could be improved using
\begin{align}\label{eq:asymptotically_efficient2}
\E\bigg[\int_s^t W_{s,u}^i W_{s,u}^j du\,\Big|\m W_{s,t}, \int_s^t W_{s,u} \m du\bigg] = \frac{1}{3}hW_{s,t}^i W_{s,t}^j & + \frac{1}{2}h\big(W_{s,t}^i H_{s,t}^j + H_{s,t}^i W_{s,t}^j\big)\nonumber\\[-2pt]
& + \frac{6}{5}hH_{s,t}^i H_{s,t}^j + \frac{1}{15}h^2 \delta_{ij}\m,
\end{align}
where $H_{s,t} := \frac{1}{h}\int_s^t W_{s,u} \m du - \frac{1}{2}W_{s,t}\sim\mathcal{N}(0,\frac{h}{12}I_d)$ is independent of $W_{s,t}\sim\mathcal{N}(0,h I_d)$.\vspace{1.5mm}
The estimator (\ref{eq:asymptotically_efficient2}) is particularly appealing as it is both optimal in an $L^2(\P)$ sense and uses $\int_s^t W_{s,u} \m du$, which is often needed by high order numerical methods for SDEs \cite{jelincic2024vbt}.
In section \ref{sect:space_space_time}, we will consider some additional random vectors related to the Brownian motion
and present the associated conditional expectations. These are more accurate than (\ref{eq:asymptotically_efficient2}) and have recently been applied to improve high order splitting methods \cite{foster2024splitting}.\medbreak

The paper is outlined as follows. In section \ref{sect:gaussian_integrals}, we define the integrals in the depth-$3$ Brownian signature which are normally distributed and can thus be generated exactly. These will then serve as input random vectors in subsequent numerical approximations.\medbreak

In section \ref{sect:space_space}, we will consider the problem of generating (space-space) L\'{e}vy area and present the moment-matching approximation proposed in the thesis \cite[Section 7.3]{foster2020thesis}.
Moreover, we shall introduce a random variable on $\{-1,1\}^4$ so that the fourth moment $\E\big[A_{s,t}^{ij}A_{s,t}^{jk}A_{s,t}^{jk}A_{s,t}^{kl}\big] = \frac{1}{48}h^4$  (distinct $i,j,k,l$) is better matched by the approximation.\medbreak

In section \ref{sect:space_space_time}, we will turn our attention to ``space-space-time'' L\'{e}vy area, which is equivalent to the integral (\ref{eq:intro_integral2}), and establish the optimal estimator (\ref{eq:asymptotically_efficient2}) from \cite{foster2020thesis, foster2020poly}. Furthermore, we introduce a random vector $n_{s,t}$ defined by $n_{s,t}^i := \mathrm{Sgn}(H_{s,u}^i - H_{u,t}^i)$ where $u=\frac{1}{2}(s+t)$ and derive a formula for the corresponding conditional expectation,
\begin{align}\label{eq:asymptotically_efficient3}
\E\bigg[ \int_s^t \big(W_{s,r}^i\big)^2 dr\,\Big|\m W_{s,t}, H_{s,t}, n_{s,t}\bigg] = \frac{1}{3}h\big(W_{s,t}^i\big)^2 & + hW_{s,t}^i H_{s,t}^i + \frac{6}{5}h\big(H_{s,t}^i\big)^2 + \frac{1}{15}h^2\nonumber\\[-2pt]
& - \frac{1}{4\sqrt{6\pi}}n_{s,t}^i h^\frac{3}{2} W_{s,t}^i\m.
\end{align}
The estimator (\ref{eq:asymptotically_efficient3}) has already been used to improve high order splitting methods \cite{foster2024splitting}.\medbreak

In section \ref{sect:examples}, we give several numerical examples from the literature to demonstrate the improvements in accuracy achievable using the proposed integral approximations. In particular, the SDE solvers that we obtain outperform standard numerical methods.
Finally, in section \ref{sect:conclusion}, we shall conclude the paper and suggest a future research topic.

\section{Gaussian iterated integrals of Brownian motion}\label{sect:gaussian_integrals}

In this section, we will consider the iterated integrals of Brownian motion with time in the depth-3 signature that are normally distributed. Over $[s,t]$, these integrals are
\begin{align*}
\int_s^t dW_u\m, &\int_s^t\int_s^u dW_v\, du, \int_s^t\int_s^u dv\, dW_u\m,\\ \int_s^t\int_s^u\int_s^v dW_r\, dv\, du\m, &\int_s^t\int_s^u\int_s^v dr\, dW_v\, du\m, \int_s^t\int_s^u\int_s^v dr\, dv\, dW_u\m.
\end{align*}
Since the above integrals have some interdependencies, they can be generated using three Gaussian random vectors -- which we shall now present in definitions \ref{def:path_increment}, \ref{def:space_time_levy} and \ref{def:space_time_time_levy}.
These quantities are especially useful for SDE solvers as they can be generated exactly.
\medbreak

Throughout this paper, $W = (W^1,\cdots, W^d)$ will denote a standard $d$-dimensional Brownian motion,  $I_d$ is the $d\times d$ identity matrix and we set $h := t - s$ with $0\leq s\leq t$.\medbreak

\begin{definition}[Path increment of Brownian motion]\label{def:path_increment}
Over an interval $[s,t]$, we define the path increment of $W$ as $W_{s,t} := W_t - W_s$. Then, by definition, $W_{s,t} \sim\mathcal{N}(0, h I_d)$.
\end{definition}\medbreak

\begin{definition}[Space-time L\'{e}vy area]\label{def:space_time_levy}
Over an interval $[s,t]$, we define the (rescaled) space-time L\'{e}vy area of Brownian motion as
\begin{align}\label{eq:space_time}
H_{s,t} := \frac{1}{h}\int_s^t \Big(W_{s,u} - \frac{u-s}{h}W_{s,t}\Big)\m du.
\end{align}
\end{definition}
\begin{remark}\label{rmk:h}
Since $H_{s,t}$ only depends on the Brownian bridge $\big\{W_{s,u} - \frac{u-s}{h}W_{s,t}\big\}_{u\in[s,t]}\m$,\vspace{0.5mm} it follows that $H_{s,t}$ is independent of $W_{s,t}$. In addition, we have $H_{s,t}\sim\mathcal{N}(0, \frac{1}{12}h I_d)$.\vspace*{1mm} This distribution can be shown using a polynomial expansion of Brownian motion \cite{foster2020poly}.
\end{remark}

\begin{figure}[ht]
\centering
\includegraphics[width=0.9\textwidth]{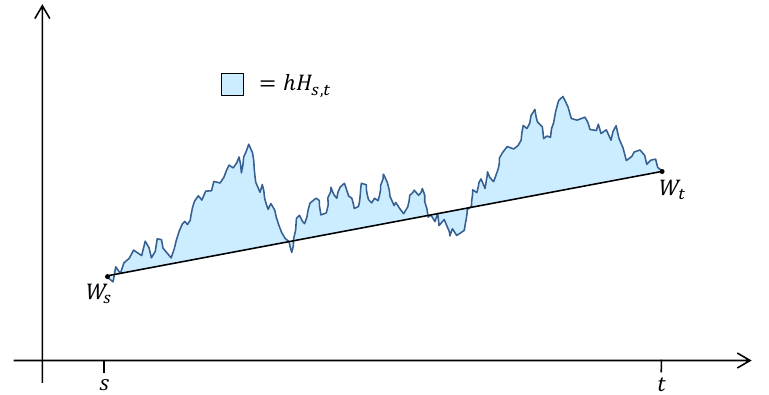}
\caption{$H_{s,t}$ is the area enclosed by a Brownian path and its linear interpolant (diagram from \cite{foster2020poly}).}\label{fig:space_time_levy}
\end{figure}

\begin{definition}[Space-time-time L\'{e}vy area]\label{def:space_time_time_levy}
Over an interval $[s,t]$, we define the (rescaled) space-time-time L\'{e}vy area of Brownian motion as
\begin{align}\label{eq:space_time_time}
K_{s,t} := \frac{1}{h^2}\int_s^t \Big(W_{s,u} - \frac{u-s}{h}W_{s,t}\Big)\bigg(\frac{1}{2}h - (u-s)\bigg)\m du.
\end{align}
\end{definition}
\begin{remark}\label{rmk:stt_levy}
Since $\frac{1}{2}h - (u-s)$ is an odd function around $s+\frac{1}{2}h$, it directly follows that
\begin{align*}
K_{s,t} = \frac{1}{h^2}\int_s^t \underbrace{\Big(W_{s,u} - \frac{u-s}{h}W_{s,t} - \frac{6(u-s)(t-u)}{h^2}H_{s,t}\Big)}_{=:\,\, Z_{s,u}\text{ and called the ``Brownian arch''}.}\bigg(\frac{1}{2}h - (u-s)\bigg)\m du.
\end{align*}
In \cite{foster2020poly}, the Brownian arch was introduced and shown to be independent of $\big(W_{s,t}\m , H_{s,t}\big)$.\vspace{0.5mm}
Hence $W_{s,t}\m, H_{s,t}$ and $K_{s,t}$ are jointly independent. Similar to Remark \ref{rmk:h}, we can use\vspace{0.5mm} the polynomial expansion of Brownian motion  to show that $\m K_{s,t}\sim\mathcal{N}(0, \frac{1}{720}h I_d)$ \cite{foster2020thesis}.\vspace*{-2mm}
\end{remark}
\begin{figure}[ht]
\centering
\includegraphics[width=0.9\textwidth]{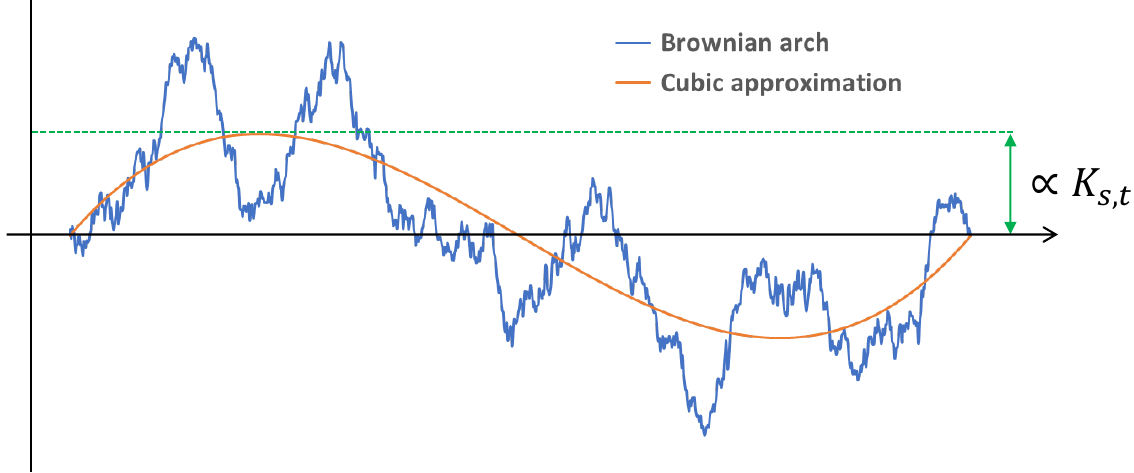}\vspace{1mm}
\caption{Space-time-time L\'{e}vy area corresponds to how asymmetric in time the Brownian motion is.}\label{fig:space_time_time_levy}
\end{figure}
As previously discussed in the introduction, the increment $W_{s,t}$ and space-time L\'{e}vy area $H_{s,t}$ can be used by high order solvers for SDEs satisfying the condition (\ref{eq:intro_commute}).
Since incorporating ``space-time-time'' L\'{e}vy area would not ordinarily improve these $O(h^\frac{3}{2})$ convergence rates,
we will consider $K_{s,t}$ with two specific applications in mind:
\begin{itemize}
\item $K_{s,t}$ is needed by the approximation of (space-space) L\'{e}vy area detailed in section \ref{sect:space_space}.\medbreak
\item Underdamped Langevin dynamics is fundamental model in statistical physics \cite{leimkuhler2015langevin} which has seen a number of recent applications in data science \cite{chen2014sghmc, cheng2018uld, nemeth2021mcmc, chada2024uld}. It is given by,
\begin{align}\label{eq:uld}
dx_t & = v_t\m dt,\\[1pt]
dv_t & = -\gamma v_t\m dt -\nabla f(x_t)\m dt + \sqrt{2\gamma}\, dW_t\m,\nonumber
\end{align}
where $x_t\m, v_t\in\R^d$, $\gamma > 0$, $f : \R^d\rightarrow \R$ and $W$ is a $d$-dimensional Brownian motion.
In \cite{foster2021uld, scott2025uld}, approximations were proposed for (\ref{eq:uld}) which uses $(W_{s,t}\m, H_{s,t}, K_{s,t})$ and achieves third order convergence (provided the gradient $\nabla f$ is sufficiently smooth).
\end{itemize}

We now show how the random vectors $W_{s,t}\m, H_{s,t}$ and $K_{s,t}$ can be used to determine the integrals in the Brownian signature that were given at the beginning of section \ref{sect:gaussian_integrals}.\medbreak

\begin{theorem}[Relationship between $(W_{s,t}\m, H_{s,t}\m, K_{s,t})$ and the Brownian signature]\label{thm:whk_relation}
Let $W_{s,t}\m, H_{s,t}\m,$ and $K_{s,t}$ be given by definitions \ref{def:path_increment}, \ref{def:space_time_levy} and \ref{def:space_time_time_levy}. Then for $0\leq s\leq t$, we have
\begin{align}
\int_s^t dW_u & = W_{s,t}\m,\label{eq:increment}\\
\int_s^t\int_s^u dW_v\, du &  = \frac{1}{2}h W_{s,t} + hH_{s,t}\m,\label{eq:st_area1}\\
 \int_s^t\int_s^u dv\, dW_u & = \frac{1}{2}h W_{s,t} - hH_{s,t}\m,\label{eq:st_area2}\\
 \int_s^t\int_s^u\int_s^v dW_r\, dv\, du & = \frac{1}{6}h^2 W_{s,t} + \frac{1}{2}h^2 H_{s,t} + h^2 K_{s,t}\m,\label{eq:stt_area1}\\
 \int_s^t\int_s^u\int_s^v dr\, dW_v\, du & = \frac{1}{6}h^2 W_{s,t} - 2h^2 K_{s,t}\m,\label{eq:stt_area2}\\
  \int_s^t\int_s^u\int_s^v dr\, dv\, dW_u & = \frac{1}{6}h^2 W_{s,t} - \frac{1}{2}h^2 H_{s,t} + h^2 K_{s,t}\m.\label{eq:stt_area3}
\end{align}
\end{theorem}
\begin{proof}
Equation (\ref{eq:increment}) follows immediately from the definition of It\^{o} integrals and $W_{s,t}\m$. From the definition of space-time L\'{e}vy area $H_{s,t}$, it is straightforward to show (\ref{eq:st_area1}) as
\begin{align*}
hH_{s,t} = \int_s^t \Big(W_{s,u} - \frac{u-s}{h}W_{s,t}\Big)\m du = \int_s^t\int_s^u dW_v\, du -  \frac{1}{2}h W_{s,t}\m.
\end{align*}
Equation (\ref{eq:st_area2}) can now be seen by applying the standard integration by parts formula,
\begin{align}\label{eq:int_by_parts}
\int_s^t\int_s^u dW_v\, du +  \int_s^t\int_s^u dv\, dW_u = hW_{s,t}\m.
\end{align}
To derive the remaining formulae, we first expand the integral in the definition of $K_{s,t}\m$.
\begin{align*}
h^2 K_{s,t} & =  \int_s^t \Big(W_{s,u} - \frac{u-s}{h}W_{s,t}\Big)\bigg(\frac{1}{2}h - (u-s)\bigg)\m du\\
& =  \frac{1}{2}h\int_s^t W_{s,u}\, du - \int_s^t (u-s)W_{s,u}\m du + \int_s^t \Big((u-s)^2 - \frac{1}{2}(u-s)\Big) W_{s,t}\, du\\
& = \frac{1}{4}h^2 W_{s,t} + \frac{1}{2}h H_{s,t} -  \int_s^t (u-s)W_{s,u}\m du + \frac{1}{12}h^2 W_{s,t}\m.
\end{align*}
From the above, we arrive at the following integral identity involving $W_{s,t}\m, H_{s,t}\m, K_{s,t}\m$.
\begin{align}\label{eq:important_integral}
\int_s^t (u-s)W_{s,u}\m du & = \frac{1}{3}h^2 W_{s,t} + \frac{1}{2}h H_{s,t} - h^2 K_{s,t}\m.
\end{align}
We will now express the remaining iterated integrals in terms of the integral (\ref{eq:important_integral}).
Using integration by parts, and integrating equation (\ref{eq:int_by_parts}) with respect to $t$, we obtain
\begin{align*}
\int_s^t W_{s,u}\, d((u-s)^2) + \int_s^t (u-s)^2 dW_u  & = h^2 W_{s,t}\m,\\
\int_s^t \bigg(\int_s^u W_{s,v}\, dv\bigg) du  + \int_s^t (u-s)\, d\bigg(\int_s^u W_{s,v}\, dv\bigg) & = h\int_s^t W_{s,u}\m du,\\
\int_s^t\int_s^u\int_s^v dW_r\, dv\, du + \int_s^t\int_s^u\int_s^v dr\, dW_v\, du  & =  \int_s^t (u-s)W_{s,u}\, du.
\end{align*}
Simplifying and rearranging these equations gives
\begin{align*}
\int_s^t\int_s^u\int_s^v dr\, dv\, dW_u  & = \frac{1}{2}h^2 W_{s,t} - \int_s^t (u-s)W_{s,u}\, du,\\
\int_s^t \int_s^u W_{s,v}\, dv\, du  & = h\int_s^t W_{s,u}\m du - \int_s^t (u-s)W_{s,u}\, du,\\
\int_s^t\int_s^u\int_s^v dr\, dW_v\, du  & =  \int_s^t (u-s)W_{s,u}\, du - \int_s^t\int_s^u\int_s^v dW_r\, dv\, du.
\end{align*}
Substituting (\ref{eq:important_integral}) into the above equations gives (\ref{eq:stt_area1}), (\ref{eq:stt_area2}) and (\ref{eq:stt_area3}) as required.
\end{proof}\medbreak
\begin{remark}
From equations (\ref{eq:increment}) to (\ref{eq:stt_area3}), we see that the vectors $W_{s,t}\m, H_{s,t}\m, K_{s,t}$ are simply rescaled coefficients in the log-signature of space-time Brownian motion $\overline{W}_t := \{(t, W_t)\}$. The definition of a path log-signature is given in \cite[Section 2.2.4]{lyons2007roughpaths}.
\end{remark}\medbreak
\begin{remark}
By It\^{o}'s isometry, we can compute the covariance of the below integrals,
\begin{align*}
\begin{pmatrix}
\int_s^t dW_u \\[2pt] \int_s^t\int_s^u dv\, dW_u \\[2pt] \m\int_s^t\int_s^u\int_s^v dr\, dv\, dW_u
\m\end{pmatrix} \sim \mathcal{N}\left(0, \begin{pmatrix}
h & \frac{1}{2} h^2 & \frac{1}{6} h^3 \\[4pt]
\m\frac{1}{2} h^2 & \frac{1}{3} h^3 & \frac{1}{8} h^4 \\[4pt]
\m\frac{1}{6} h^3 & \frac{1}{8} h^4 & \frac{1}{20} h^5
\end{pmatrix}\right).
\end{align*}
From the above with equations (\ref{eq:increment}), (\ref{eq:st_area1}) and (\ref{eq:stt_area3}), it is straightforward to show that
\begin{align*}
\begin{pmatrix} W_{s,t} \\[2pt] H_{s,t} \\[2pt] K_{s,t} \end{pmatrix} \sim \mathcal{N}\left(0, \begin{pmatrix}
h & 0 & 0 \\[1pt] 0 & \frac{1}{12}h & 0 \\[1pt] 0 & 0 & \frac{1}{720}h
\end{pmatrix}\right).
\end{align*}
\end{remark}

Since the iterated integrals considered in this section are all normally distributed, it is also possible to generate them for use by SDE solvers with adaptive step sizes.
This is detailed in \cite{foster2020thesis, jelincic2024vbt} and requires one to identify the conditional distributions,
\begin{align*}
W_{s,u} \,|\, W_{s,t}\m,\quad \big(W_{s,u}\m, H_{s,u}\big) \,|\, \big(W_{s,t}\m, H_{s,t}\big),\quad \big(W_{s,u}\m, H_{s,u}\m, K_{s,u}\big) \,|\, \big(W_{s,t}\m, H_{s,t}\m, K_{s,t}\big), 
\end{align*}
where $u\in[s,t]$. Whilst the distributions $W_{s,u} \,|\, W_{s,t}\m,$ and $\big(W_{s,u}\m, H_{s,u}\big) \,|\, \big(W_{s,t}\m, H_{s,t}\big)$\vspace{0.5mm} are well known, the third distribution was only recently computed explicitly \cite{foster2020thesis, jelincic2024vbt}.\medbreak

\section{L\'{e}vy area of multidimensional Brownian motion}\label{sect:space_space}

In this section, we shall present a moment-based weak approximation of L\'{e}vy area,
which was introduced in \cite{foster2020thesis} and subsequently tested in \cite{jelincic2025levygan}. At the end of the section, we will propose a new modification in the $d=4$ case that better matches moments.
Recall that the L\'{e}vy area of a $d$-dimensional Brownian motion over an interval $[s,t]$ is
\begin{align}\label{eq:levy_area}
A_{s,t} := \bigg\{\frac{1}{2}\bigg(\int_s^t W_{s,u}^i\, dW_u^j - \int_s^t W_{s,u}^j\m dW_u^i\bigg)\bigg\}_{1\m\leq\m i,\m j\m \leq\m d}.
\end{align}

We first note that the integral (\ref{eq:intro_integral1}) can be decomposed into a symmetric component (depending on the increment $W_{s,t}$) and an antisymmetric component (given by (\ref{eq:levy_area})). \medbreak
\begin{theorem}[Decomposition of iterated integral (\ref{eq:intro_integral1})]\label{thm:levy_area_relation}For $i,j\in\{1,\cdots,d\}$, we have
\begin{align}\label{eq:levy_area_relation}
\int_s^t W_{s,u}^i \circ dW_u^j = \frac{1}{2} W_{s,t}^i W_{s,t}^j + A_{s,t}^{ij}\m.
\end{align}
\end{theorem}
\begin{proof}
The result follows from the definition of $A_{s,t}$ and integration by part formula.\vspace{-1.5mm}
\begin{align*}
\int_s^t W_{s,u}^i \circ dW_u^j & = \frac{1}{2}\bigg(\int_s^t W_{s,u}^i \circ dW_u^j + \int_s^t W_{s,u}^j \circ dW_u^i \bigg)\\
&\mm + \frac{1}{2}\bigg(\int_s^t W_{s,u}^i \circ dW_u^j - \int_s^t W_{s,u}^j \circ dW_u^i \bigg) = \frac{1}{2} W_{s,t}^i W_{s,t}^j + A_{s,t}^{ij}\m,\\[-20pt]
\end{align*}
as the It\^{o} and Stratonovich integrals coincide when $W^i$ and $W^j$ are independent.
\end{proof}
Following \cite[Chapter 7]{foster2020thesis}, we will compute the first four moments of L\'{e}vy area. This will require the following decompositions of L\'{e}vy area involving $W_{s,t}\m, H_{s,t}\m, K_{s,t}\m$,
which we shall express more succinctly using the standard tensor product notation $\otimes\,$.\medbreak
\begin{definition}[Tensor product]	 
For vectors $X\in\R^n$ and $Y\in\R^m$, we define the matrix $X\otimes Y\in \R^{n\times m}\m$ with entries $\m(X\otimes Y)^{ij} = X^i\m Y^j\m$ for $\m 1\m\leq\m i\m\leq n\m$ and $\m 1\m\leq\m j\m\leq m$.
\end{definition}\medbreak
\begin{theorem}\label{thm:levy_decomp}
Using $W_{s,t}\m, H_{s,t}\m, K_{s,t}$, we can express the Brownian L\'{e}vy area (\ref{eq:levy_area}) as
\begin{align}
A_{s,t} & = H_{s,t}\otimes W_{s,t} - W_{s,t}\otimes H_{s,t} + b_{s,t}\m,\label{eq:levy_area_decomp_1}\\[3pt]
A_{s,t} & = H_{s,t}\otimes W_{s,t} - W_{s,t}\otimes H_{s,t} + 12\big(K_{s,t}\otimes H_{s,t} - H_{s,t}\otimes K_{s,t}\big) + a_{s,t}\m,\label{eq:levy_area_decomp_2}
\end{align}
where $b_{s,t}$ denotes the L\'{e}vy area of the Brownian bridge, $B_{s,u} := W_{s,u} - \frac{u-s}{h}W_{s,t}\m,$
\begin{align*}
b_{s,t}^{ij} := \int_s^t B_{s,u}^i \circ dB_u^j\m,
\end{align*}
and $a_{s,t}$ is the L\'{e}vy area of the Brownian arch, $Z_{s,u} := B_{s,u} - \frac{6(u-s)(t-u)}{h^2}H_{s,t}\m$,
\begin{align*}
a_{s,t}^{ij} := \int_s^t Z_{s,u}^i \circ dZ_u^j\m.
\end{align*}
\end{theorem}
\begin{proof}
By Brownian scaling, it is enough to prove the result on the interval $[0,1]$. Using the decomposition (\ref{eq:levy_area_relation}) and $W_t = tW_1 + B_t$, we have
\begin{align*}
A_{0,1}^{ij} & = \int_0^1 \big(tW_1^i + B_t^i\big)\circ d\big(tW_1^j + B_t^j\big) - \frac{1}{2}W_1^i W_1^j\\
& = W_1^i W_1^j\int_0^1 t\, dt +  W_1^j\int_0^1 B_t^i\, dt + W_1^i\int_0^1 t\, dB_t^j + \int_0^1 B_t^i \circ dB_t^j - \frac{1}{2}W_1^i W_1^j \\[3pt]
& = H_{0,1}^i W_1^j - W_1^i H_{0,1}^j + b_{0,1}^{ij}\m,
\end{align*}
which gives equation (\ref{eq:levy_area_decomp_1}) as required. Similarly, using $B_t = 6t(1-t)H_{0,1} + Z_t$, we have
\begin{align*}
b_{0,1}^{ij} = \int_0^1 B_t^i \circ dB_t^j & = \int_0^1 \big(6t(1-t)H_{0,1}^i + Z_t^i\big) \circ d\big(6t(1-t)H_{0,1}^j + Z_t^j\big)\\
& = H_{0,1}^i H_{0,1}^j \int_0^1 6t(1-t)\, d(6t(1-t)) + H_{0,1}^i\int_0^1 6t(1-t) \circ d Z_t^j\\
&\mmm + H_{0,1}^j\int_0^1 Z_t^i\, d\big(6t(1-t)\big) + \int_0^1 Z_t^i \circ dZ_t^j\\
& = H_{0,1}^i\int_0^1 6t(1-t) \circ d Z_t^j +  H_{0,1}^j\int_0^1 (6 - 12t)Z_t^i\, dt + a_{0,1}^{ij}\m.
\end{align*}
Note that, by Remark \ref{rmk:stt_levy}, we have $K_{0,1} = \int_0^1 \big(\frac{1}{2} - t\big)Z_t\, dt$. Using integration by parts then gives
$\int_0^1 6t(1-t) \circ d Z_t  = - \int_0^1 Z_t\, d(6t(1-t)) = -12\int_0^1 \Big(\frac{1}{2} - t\Big)Z_t\, dt = -12K_{0,1}\m$. The second result (\ref{eq:levy_area_decomp_2}) now directly follows from the above decomposition of $b_{0,1}^{ij}\m$.
\end{proof}

Before computing the moments of Brownian L\'{e}vy area, we note the following result. \medbreak
\begin{theorem}[Distribution of the Brownian bridge L\'{e}vy area]\label{thm:bb_distribution}
For each entry in $b_{s,t}\m$,
\begin{align}\label{eq:bb_distribution}
b_{s,t}^{ij} \sim \mathrm{Logistic}\Big(0, \frac{1}{\pi}\m h\Big).
\end{align}
In particular, this implies that the Brownian bridge L\'{e}vy area has the moments,
\begin{align}\label{eq:bb_moments}
\E\Big[\big(b_{s,t}^{ij}\big)^2\Big] & = \frac{1}{12}h^2,\mmm
\E\Big[\big(b_{s,t}^{ij}\big)^4\Big] = \frac{7}{240}h^4.
\end{align}
\end{theorem}
\begin{proof}
It was shown in \cite{levy1951area} that the joint density of $x = W_1^i$, $y = W_1^j$ and $z = A_{0,1}^{ij}$ is
\begin{align*}
p(x,y,z) = \frac{1}{2\pi^2}\int_0^\infty \frac{u}{\sinh(u)}\exp\Big(-\frac{(x^2 + y^2)u}{2\tanh(u)}\Big)\cos(z u) du.
\end{align*}
Setting $x = y = 0$ and simplifying the above integral gives the density of $\mathrm{Logistic}\big(0, \frac{1}{\pi}\big)$ (see \cite[Theorem 7.0.15]{foster2020thesis}). By the natural Brownian scaling, we arrive at equation (\ref{eq:bb_distribution}).
Since the logistic distribution has known moments, we can obtain (\ref{eq:bb_moments}) as required.
\end{proof}

\subsection{Conditional moments of two-dimensional L\'{e}vy area}\label{sect:moments1}

We shall now compute moments of Brownian L\'{e}vy area conditional on $W_{s,t}$ and $H_{s,t}\m$. This result was shown in the thesis \cite[Theorem 7.1.1]{foster2020thesis}, but has not yet been published.\medbreak

\begin{theorem}[Conditional moments of L\'{e}vy area]\label{thm:levy_area_moments1}
For $1\leq i, j\leq d$ with $i\neq j$, we have
\begin{align}
\E\big[A_{s,t}^{ij}\,|\, W_{s,t}\m, H_{s,t}\big] & = H_{s,t}^i W_{s,t}^j - H_{s,t}^j W_{s,t}^i\m,\label{eq:levy_area_cond_exp}\\[3pt]
\var\big(A_{s,t}^{ij} \,|\, W_{s,t}\m, H_{s,t}\big) & = \frac{1}{20} h^2 + \frac{1}{5}h\Big(\big(H_{s,t}^i\big)^2 + \big(H_{s,t}^j\big)^2\Big)\label{eq:levy_area_cond_var},\\[3pt]
\mathrm{Skew}\big(A_{s,t}^{ij} \,|\, W_{s,t}\m, H_{s,t}\big) & = 0\label{eq:levy_area_cond_skew},\\[-6pt]
\mathrm{Kurt}\big(A_{s,t}^{ij} \,|\, W_{s,t}\m, H_{s,t}\big) & = 3 + \frac{\frac{3}{1400}h^2 + \frac{3}{175}h\Big(\big(H_{s,t}^i\big)^2 + \big(H_{s,t}^j\big)^2\Big)}{\Big(\frac{1}{20} h^2 + \frac{1}{5}h\Big(\big(H_{s,t}^i\big)^2 + \big(H_{s,t}^j\big)^2\Big)\Big)^2}\,.\label{eq:levy_area_cond_kurt}
\end{align}
\end{theorem}
\begin{proof}
Since the Brownian arch $Z$ is independent of $W_{s,t}$ and $H_{s,t}$ (as shown in \cite{foster2020poly}), it follows that $a_{s,t}$ is independent of $W_{s,t}$ and $H_{s,t}\m$. Moreover, since $Z^i$ and $-Z^i$ have the same distribution, we also have that $a_{s,t}^{ij}$ and $-a_{s,t}^{ij}$ have the same distribution. Therefore, $\E\big[a_{s,t}^{ij}\,|\, W_{s,t}\m, H_{s,t}\big] = \E\big[a_{s,t}^{ij}\big] = 0$ and the conditional expectation of $b_{s,t}^{ij}$ is
\begin{align*}
\E\big[b_{s,t}^{ij}\,|\, W_{s,t}\m, H_{s,t}\big] & = \E\big[12\m K_{s,t}^i H_{s,t}^j - 12\m H_{s,t}^i K_{s,t}^j + a_{s,t}^{ij}\,|\, W_{s,t}\m, H_{s,t}\big]\\[2pt]
& = 12\m H_{s,t}^j\m\E[K_{s,t}^i] - 12\m H_{s,t}^i\m\E[K_{s,t}^j] + \E\big[a_{s,t}^{ij}\big] = 0\m,
\end{align*}
where we used Theorem \ref{thm:levy_decomp} in the first line. The second moment can be computed as
\begin{align*}
\E\Big[\big(b_{s,t}^{ij}\big)^2\,\big|\, W_{s,t}\m, H_{s,t}\Big] & = \E\Big[\big(12\m K_{s,t}^i H_{s,t}^j - 12\m H_{s,t}^i K_{s,t}^j + a_{s,t}^{ij}\big)^2 \,\big|\, W_{s,t}\m, H_{s,t}\Big]\\
& = 144\m\big(H_{s,t}^j\big)^2\m\E\Big[\big(K_{s,t}^i\big)^2\m\Big] + 144\m\big(H_{s,t}^i\big)^2\m\E\Big[\big(K_{s,t}^j\big)^2\m\Big]\\[1pt]
&\mmm + 24\m H_{s,t}^j\m\E\big[K_{s,t}^i a_{s,t}^{ij}\big] - 24\m H_{s,t}^i\m\E\big[K_{s,t}^j a_{s,t}^{ij}\big]\\[1pt]
&\mmm - 288\m H_{s,t}^i H_{s,t}^j\m\E\big[K_{s,t}^i K_{s,t}^j\big] + \E\Big[\big(a_{s,t}^{ij}\big)^2\m\Big].
\end{align*}
By reflecting the $j$-th coordinate of the Brownian arch, we see that $\E\big[K_{s,t}^i a_{s,t}^{ij}\big] = 0$. Since $K_{s,t}\sim\mathcal{N}\big(0, \frac{1}{720}h I_d\big)$, we have $\E\big[K_{s,t}^i K_{s,t}^j\big] = 0$ and the above expression becomes
\begin{align*}
\E\Big[\big(b_{s,t}^{ij}\big)^2\,\big|\, W_{s,t}\m, H_{s,t}\Big] & = \E\Big[\big(a_{s,t}^{ij}\big)^2\m\Big] + \frac{1}{5}h\Big(\big(H_{s,t}^i\big)^2 + \big(H_{s,t}^j\big)^2\Big).
\end{align*}
Since $b_{s,t}^{ij}$ has variance $\frac{1}{12}h^2$ (see Theorem \ref{thm:bb_distribution}), taking the expectation of the above gives
\begin{align*}
\E\Big[\big(b_{s,t}^{ij}\big)^2\Big] = \E\Big[\E\Big[\big(b_{s,t}^{ij}\big)^2\,\big|\, W_{s,t}\m, H_{s,t}\Big]\Big] = \E\Big[\big(a_{s,t}^{ij}\big)^2\m\Big] + \frac{1}{30}h^2 \implies \E\Big[\big(a_{s,t}^{ij}\big)^2\m\Big] = \frac{1}{20}h^2,
\end{align*}
which results in equation (\ref{eq:levy_area_cond_var}) as required.

To compute the conditional skewness of $A_{s,t}\m$, we note the following moments,
\begin{align*}
\E\big[\big(a_{s,t}^{ij}\big)^3\m\big] & = 0,\mm \big(\m\text{by the symmetry of }a_{s,t}^{ij}\big),\\
\E\big[K_{s,t}^i\big(a_{s,t}^{ij}\big)^2\m\big] & = 0,\mm \big(\m\text{by the symmetry of }Z^i\big),\\
\E\big[\big(K_{s,t}^i\big)^2 a_{s,t}^{ij}\m\big]  & = 0,\mm \big(\m\text{by the symmetry of }Z^j\big),\\
\E\big[K_{s,t}^i K_{s,t}^j a_{s,t}^{ij}\m\big] & = 0,\mm \big(\m\text{since }a_{s,t}^{ij} = -\m a_{s,t}^{ji}\big).
\end{align*}
Therefore, by the independence of $K_{s,t}$ and $(W_{s,t}\m, H_{s,t})$, it follows from the above that
\begin{align*}
\E\Big[\big(b_{s,t}^{ij}\big)^3\,\big|\, W_{s,t}\m, H_{s,t}\Big] & = \E\Big[\big(12\m K_{s,t}^i H_{s,t}^j - 12\m H_{s,t}^i K_{s,t}^j + a_{s,t}^{ij}\big)^3 \,\big|\, W_{s,t}\m, H_{s,t}\Big] = 0.
\end{align*}
By the polynomial expansion of Brownian motion (the main result of \cite{foster2020poly, habermann2021poly}), we have
\begin{align*}
W_t = tW_1 & + 6t(1-t)H_{0,1} + 60t(1-t)(1-2t)K_{0,1}\\
& + 210\m t(1-t)\big(5(2t-1)^2 - 1\big)M_{0,1} + G_t\m,
\end{align*}
for $t\in [0,1]$, where $M_{0,1} := \int_0^1 (W_t - t W_1)\big(\frac{1}{2}\big(t -\frac{1}{2}\big)^2 - \frac{1}{40}\big)\m dt \sim\mathcal{N}\big(0,\frac{1}{100800} I_d\big)$ is\vspace{0.5mm} independent of $W_1\m, H_{0,1}\m, K_{0,1}\m$ and $G = \{G_t\}$ is independent of  $W_1\m, H_{0,1}\m, K_{0,1}, M_{0,1}\m$.
Moreover, it follows from the orthogonality of the polynomials in the expansion that
\begin{align*}
\int_0^1 t^k G_t\m dt = 0,\mm\text{for }k\in\{0,1,2\}.
\end{align*}
Since $G_0 = G_1 = 0$, integration by parts implies that $\int_0^1 t^k d G_t = 0$ for $k\in\{0,1,2,3\}$.
Letting $p_2(t) := 60t(1-t)(1-2t)$ and $p_3(t) := 210\m t(1-t)\big(5(2t-1)^2 - 1\big)$, we have
\begin{align*}
a_{0,1}^{ij} & = \int_0^1 Z_u^i \circ dZ_u^j\\
& = \int_0^1 \big(\m p_2(t)K_{0,1}^i + p_3(t)M_{0,1}^i + G_t^i\m\big) \circ d\big(\m p_2(t)K_{0,1}^j + p_3(t)M_{0,1}^j + G_t^j\m\big)\\
& = K_{0,1}^i M_{0,1}^j\int_0^1 p_2(t)\,dp_3(t) + M_{0,1}^i K_{0,1}^j\int_0^1 p_3(t)\,dp_2(t)\\
&\mm + \int_0^1 \big(\m p_3(t)M_{0,1}^i + G_t^i\m\big) \circ d\big(\m p_3(t)M_{0,1}^j + G_t^j\big)\m,
\end{align*}
since $\int_0^1 p_k(t)\, dp_k(t) = 0$ for $k\in\{0,1\}$. Applying the usual Brownian scaling gives
\begin{align}\label{eq:arch_area_expand}
a_{s,t}^{ij} & = 720\big(M_{s,t}^i K_{s,t}^j - K_{s,t}^i M_{s,t}^j\big) +  c_{s,t}^{ij}\m,
\end{align}
where $M_{s,t}\sim\mathcal{N}\big(0, \frac{1}{100800}h I_d\big)$ is given by
\begin{align}\label{eq:m_def}
M_{s,t} := \frac{1}{h^3}\int_s^t \Big(W_{s,u} - \frac{u-s}{h}W_{s,t}\Big)\bigg(\frac{1}{2}\bigg((u-s)-\frac{1}{2}h\bigg)^2 - \frac{1}{40}h^2\bigg)\m du,
\end{align}
and $c_{s,t}$ is the L\'{e}vy area of the ``cubic Brownian bridge'', $\widetilde{Z}_{s,u} := Z_{s,u} - p_2\big(\frac{u-s}{h}\big)K_{s,t}\m,$
\begin{align}\label{eq:c_def}
c_{s,t}^{ij} := \int_s^t \widetilde{Z}_{s,u}^i \circ d \widetilde{Z}_u^j\m.
\end{align}
Since $c_{s,t}$ is independent of $K_{s,t}\m$, we have
\begin{align*}
\E\Big[\big(a_{s,t}^{ij}\big)^2\Big] & = \E\Big[\Big(720\big(M_{s,t}^i K_{s,t}^j - K_{s,t}^i M_{s,t}^j\big) +  c_{s,t}^{ij}\Big)^2\,\Big]\\[3pt]
& = 720^2\m\E\Big[\big(M_{s,t}^i K_{s,t}^j - K_{s,t}^i M_{s,t}^j\big)^2\Big] + \E\Big[\big(c_{s,t}^{ij}\big)^2\Big]\\
&\mmm + 1440\,\E\Big[\big(M_{s,t}^i K_{s,t}^j - K_{s,t}^i M_{s,t}^j\big)c_{s,t}^{ij}\Big]\\
& = \frac{1}{70}h^2 + \E\Big[\big(c_{s,t}^{ij}\big)^2\Big].
\end{align*}
Recall that $\E\Big[\big(a_{s,t}^{ij}\big)^2\Big] = \frac{1}{20}h^2$. Thus, from the above, it follows that $\E\Big[\big(c_{s,t}^{ij}\big)^2\Big] = \frac{1}{28}h^2$.\vspace{0.5mm}
This enables us to compute the following expectation, which we will use to derive (\ref{eq:levy_area_cond_kurt}),
\begin{align*}
\E\Big[\big(K_{s,t}^i\big)^2\big(a_{s,t}^{ij}\big)^2\Big] & = \E\Big[\big(K_{s,t}^i\big)^2\big(720\big(M_{s,t}^i K_{s,t}^j - K_{s,t}^i M_{s,t}^j\big) +  c_{s,t}^{ij}\big)^2\Big]\\
& = 720^2\m\E\Big[\big(K_{s,t}^i\big)^2\big(M_{s,t}^i K_{s,t}^j - K_{s,t}^i M_{s,t}^j\big)^2\Big] + \E\Big[\big(K_{s,t}^i\big)^2\big(c_{s,t}^{ij}\big)^2\Big]\\
&\mmm + 1440\,\E\Big[\big(K_{s,t}^i\big)^2\big(M_{s,t}^i K_{s,t}^j - K_{s,t}^i M_{s,t}^j\big)c_{s,t}^{ij}\Big]\\
& = 720^2\m\E\Big[\big(K_{s,t}^i\big)^4 \big(M_{s,t}^j\big)^2\Big] + 720^2\m\E\Big[\big(K_{s,t}^i\big)^2 \big(K_{s,t}^j\big)^2\big(M_{s,t}^i\big)^2\Big]\\
&\mmm + \E\Big[\big(K_{s,t}^i\big)^2\Big]\E\Big[\big(c_{s,t}^{ij}\big)^2\Big]\\
& = \frac{3}{100800}h^3 + \frac{1}{100800}h^3 + \frac{1}{720}h \times \frac{1}{28}h^2\\
& = \frac{1}{11200}h^3.
\end{align*}
In the above calculation, we used the independence of $K_{s,t}$ and the pair $(M_{s,t}\m, c_{s,t})$. By reflecting the $j$-th coordinate of the Brownian arch, it is straightforward to see that
\begin{align*}
\E\Big[K_{s,t}^i \big(a_{s,t}^{ij}\big)^3\Big] & = 0,\hspace{11.5mm} \E\Big[\big(K_{s,t}^i\big)^3 a_{s,t}^{ij}\Big] = 0,\\[3pt]
\E\Big[K_{s,t}^i K_{s,t}^j\big(a_{s,t}^{ij}\big)^2\Big] & = 0, \hspace{5mm} \E\Big[K_{s,t}^i\big(K_{s,t}^j\big)^2 a_{s,t}^{ij}\Big]  = 0.
\end{align*}

Expanding the fourth conditional moment of the Brownian bridge L\'{e}vy area gives
\begin{align*}
\E\Big[\big(b_{s,t}^{ij}\big)^4\,\big|\, W_{s,t}\m, H_{s,t}\Big] & = \E\Big[\big(12 K_{s,t}^i H_{s,t}^j - 12 H_{s,t}^i K_{s,t}^j + a_{s,t}^{ij}\big)^4 \,\big|\, W_{s,t}\m, H_{s,t}\Big] \\[3pt]
& = 12^4\big(H_{s,t}^j\big)^4\,\E\Big[\big(K_{s,t}^i\big)^4\Big] + 12^4\big(H_{s,t}^i\big)^4\,\E\Big[\big(K_{s,t}^j\big)^4\Big]\\
&\mmm + 6\times 12^4\big(H_{s,t}^i\big)^2\big(H_{s,t}^j\big)^2\, \E\Big[\big(K_{s,t}^i\big)^2 \big(K_{s,t}^j\big)^2\Big]\\
&\mmm + 6\times 144\m\big(H_{s,t}^j\big)^2\, \E\Big[\big(K_{s,t}^i\big)^2 \big(a_{s,t}^{ij}\big)^2\Big]\\
&\mmm + 6\times 144\m\big(H_{s,t}^i\big)^2\, \E\Big[\big(K_{s,t}^j\big)^2 \big(a_{s,t}^{ij}\big)^2\Big] + \E\Big[\big(a_{s,t}^{ij}\big)^4\Big]\\[3pt]
& = \frac{3}{25}h^2\Big(\big(H_{s,t}^i\big)^4 + \big(H_{s,t}^j\big)^4\Big) + \frac{6}{25}h^2\big(H_{s,t}^i\big)^2\big(H_{s,t}^j\big)^2\\
&\mmm + \frac{27}{350}h^3\Big(\big(H_{s,t}^i\big)^2 + \big(H_{s,t}^i\big)^2\Big) + \E\Big[\big(a_{s,t}^{ij}\big)^4\Big].
\end{align*}
Since $\E\Big[\big(b_{s,t}^{ij}\big)^4\,\Big] = \frac{7}{240}h^4$ by Theorem \ref{thm:bb_distribution}, taking an expectation of the above leads to
\begin{align*}
\E\Big[\big(b_{s,t}^{ij}\big)^4\,\Big] = \frac{1}{200}h^4 + \frac{1}{600}h^4 + \frac{9}{700}h^4 + \E\Big[\big(a_{s,t}^{ij}\big)^4\Big] \implies \E\Big[\big(a_{s,t}^{ij}\big)^4\Big] =  \frac{27}{2800}h^4\m.
\end{align*}
Therefore, the above expansion for the fourth conditional moment of $b_{s,t}^{ij}$ simplifies to
\begin{align*}
&\E\Big[\big(b_{s,t}^{ij}\big)^4\,\big|\, W_{s,t}\m, H_{s,t}\Big] \\[3pt]
&\mm = \frac{27}{2800}h^4 + \frac{27}{350}h^3\Big(\big(H_{s,t}^i\big)^2 + \big(H_{s,t}^i\big)^2\Big) + \frac{3}{25}h^2\Big(\big(H_{s,t}^i\big)^2 + \big(H_{s,t}^j\big)^2\Big)^2,
\end{align*}
and dividing by $\var\big(A_{s,t}^{ij} \,|\, W_{s,t}\m, H_{s,t}\big)^2 = \big(\frac{1}{20}h^2 + \frac{1}{5}h\big(\big(H_{s,t}^i\big)^2 + \big(H_{s,t}^j\big)^2\big)\big)^2$ gives (\ref{eq:levy_area_cond_kurt}).\hspace*{-1.15mm}
\end{proof}\vspace*{-4mm}
\begin{figure}[ht]
\centering
\includegraphics[width=\textwidth]{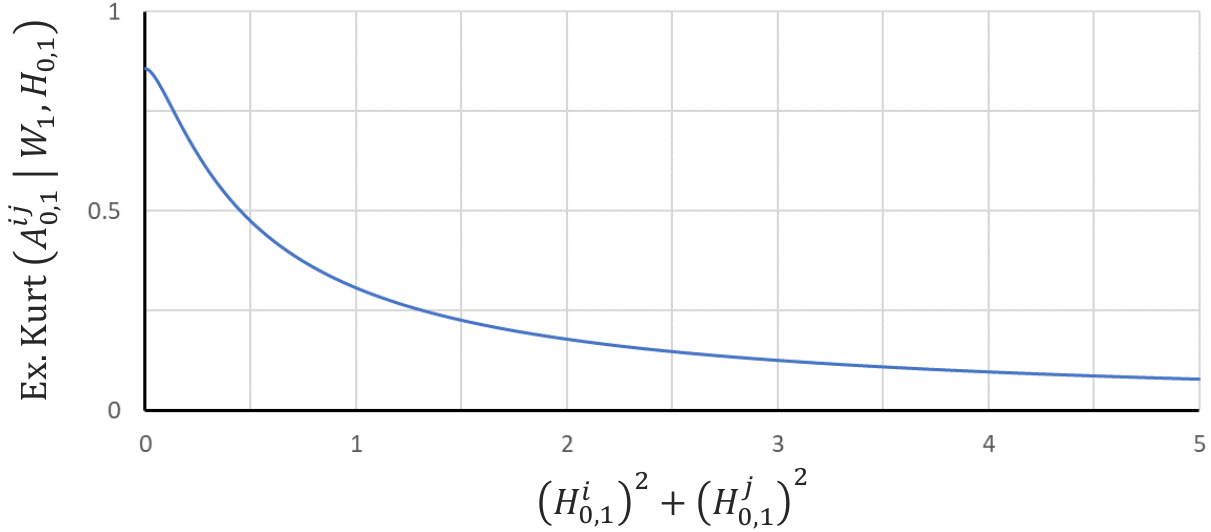}\vspace*{-2.5mm}
\caption{Given $(W_{s,t}\m, H_{s,t})$, Brownian L\'{e}vy area always has more kurtosis than a normal distribution.}\vspace*{-2.5mm}\label{fig:kurtosis}
\end{figure}

\subsection{Conditional moments of multi-dimensional L\'{e}vy area}\label{sect:moments2}

Continuing from the previous subsection, we now compute cross-moments when $d\geq 3$.\medbreak
\begin{theorem}[Additional moments of the Brownian bridge L\'{e}vy area]\label{thm:levy_area_moments2}
Let $b_{s,t}$ be the Brownian bridge L\'{e}vy area over $[s,t]$. Then for distinct $i,j,k,l\in\{1,\cdots, d\}$, we have
\begin{align}
\E\Big[\big(b_{s,t}^{ij}\big)^2\big(b_{s,t}^{\m jk}\big)^2\Big] & = \frac{7}{720}h^4,\label{eq:cross_moment1}\\[1pt]
\E\Big[H_{s,t}^i H_{s,t}^k\m b_{s,t}^{ij}b_{s,t}^{\m jk}\Big] & = -\frac{1}{720}h^4,\label{eq:cross_moment2}\\[1pt]
\E\Big[b_{s,t}^{ij}b_{s,t}^{\m jk}b_{s,t}^{kl}b_{s,t}^{li}\Big] & = \frac{1}{720}h^4.\label{eq:cross_moment3}
\end{align}
\end{theorem}
\begin{proof}
Let $B_{s,u} := W_{s,u} - \frac{u-s}{h}W_{s,t}$ be the Brownian bridge on $[s,t]$. Then the process
\begin{align}\label{eq:bridge_ik}
B^{i,k} := \frac{\sqrt{2}}{2}\big(B^i + B^k\big),
\end{align}
is also a Brownian bridge process with the same distribution as $B^i$ (and $B^k$). Hence
\begin{align*}
\int_s^t B_{s,u}^{i,k} \circ B_u^j & = \frac{\sqrt{2}}{2}\int_s^t B_{s,u}^i \circ B_u^j + \frac{\sqrt{2}}{2}\int_s^t B_{s,u}^k \circ B_u^j\\
& = \frac{\sqrt{2}}{2}\bigg(\int_s^t B_{s,u}^i \circ B_u^j - \int_s^t B_{s,u}^j \circ B_u^k\bigg) = \frac{\sqrt{2}}{2}\big(b_{s,t}^{ij} - b_{s,t}^{jk}\big),
\end{align*}
has the same distribution as $b_{s,t}^{ij}\m$. Therefore, by Theorem \ref{thm:bb_distribution}, it immediately follows that
\begin{align*}
\E\bigg[\bigg(\frac{\sqrt{2}}{2}\Big(b_{s,t}^{ij} - b_{s,t}^{\m jk}\Big)\bigg)^{\hspace{-0.25mm}4}\,\bigg] = \E\Big[\big(b_{s,t}^{ij}\big)^4\Big] = \frac{7}{240}h^4.
\end{align*}
Expanding the bracket on the left-hand side gives
\begin{align*}
\frac{1}{4}\E\Big[\big(b_{s,t}^{ij}\big)^4\Big] + \E\Big[\big(b_{s,t}^{ij}\big)\big(b_{s,t}^{\m jk}\big)^3\Big] + \frac{3}{2}\E\Big[\big(b_{s,t}^{ij}\big)^2\big(b_{s,t}^{\m jk}\big)^2\Big] + \E\Big[\big(b_{s,t}^{ij}\big)^3\big(b_{s,t}^{\m jk}\big)\Big] + \frac{1}{4}\E\Big[\big(b_{s,t}^{\m jk}\big)^4\Big]. 
\end{align*}
By the symmetry of $B^i\,$(or $B^k$), it follows that the second and fourth terms are zero.
Since the expectations of the first and last terms are known, we obtain equation (\ref{eq:cross_moment1}). To calculate the next cross-moment, we use the decompositions given by Theorem \ref{thm:levy_decomp}.
\begin{align*}
\E\Big[H_{s,t}^i H_{s,t}^k\m b_{s,t}^{ij}b_{s,t}^{\m jk}\Big] & = \E\Big[H_{s,t}^i H_{s,t}^k \big( 12\big(K_{s,t}^i H_{s,t}^j - H_{s,t}^i K_{s,t}^j\big) + a_{s,t}^{ij}\big)\\[-3pt]
&\hspace{21.3mm} \big(12\big(K_{s,t}^j H_{s,t}^k - H_{s,t}^j K_{s,t}^k\big) + a_{s,t}^{jk}\big)\Big].
\end{align*}
When the right-hand side is expanded, most of the resulting terms are zero due to the independence of $H_{s,t}$ and the pair $(K_{s,t}\m, a_{s,t})$. Evaluating the only non-zero term gives
\begin{align*}
\E\Big[H_{s,t}^i H_{s,t}^k\m b_{s,t}^{ij}b_{s,t}^{\m jk}\Big] & = - 144\m\E\Big[\big(H_{s,t}^i\big)^2\big(H_{s,t}^k\big)^2\big(K_{s,t}^j\big)^2\Big] = -\frac{1}{720}h^3,
\end{align*}
as required. By considering the Brownian bridge $B^{ik} = \frac{\sqrt{2}}{2}\big(B^i + B^k\big)$, we can see that
\begin{align*}
\bigg(\frac{\sqrt{2}}{2}\Big(b_{s,t}^{ij} + b_{s,t}^{kj}\Big), \frac{\sqrt{2}}{2}\Big(b_{s,t}^{li} + b_{s,t}^{lk}\Big)\bigg),
\end{align*}
has the same distribution as $\big(b_{s,t}^{ij}\m, b_{s,t}^{li}\big)$ and hence will have the same fourth moments,
\begin{align*}
\E\bigg[\bigg(\frac{\sqrt{2}}{2}\Big(b_{s,t}^{ij} + b_{s,t}^{kj}\Big) + \frac{\sqrt{2}}{2}\Big(b_{s,t}^{li} + b_{s,t}^{lk}\Big)\bigg)^4\,\bigg] & = \E\Big[\big(b_{s,t}^{ij} + b_{s,t}^{li}\big)^4\Big].
\end{align*}
Since $b_{s,t}^{ij} + b_{s,t}^{li}$ has the same distribution as $\sqrt{2}\m b_{s,t}^{ij}$, the right-hand side equals $\frac{7}{60}h^4$. By reflecting the $k$-th coordinate of the $B$ and using $\big(b_{s,t}^{kj}, b_{s,t}^{lk}\big) = - \big(b_{s,t}^{jk}, b_{s,t}^{kl}\big)$, we have
\begin{align}\label{eq:fourth_expand}
\E\bigg[\Big(b_{s,t}^{ij} + b_{s,t}^{jk} + b_{s,t}^{kl} + b_{s,t}^{li}\Big)^4\,\bigg] & = \frac{7}{15}h^4.
\end{align}
By reflecting the $i$-th or $l$-th corrdinate of $B$, we see the following moments are zero.
\begin{align*}
\E\Big[\big(b_{s,t}^{ij}\big)^3\m b_{s,t}^{\m jk}\Big] & = \E\Big[\big(b_{s,t}^{ij}\big)^3\m b_{s,t}^{kl}\Big]  = \E\Big[\big(b_{s,t}^{ij}\big)^3\m b_{s,t}^{li}\Big]  = 0. 
\end{align*}
Several other terms in the expansion of (\ref{eq:fourth_expand}) are computable using independence, e.g.
\begin{align*}
\E\Big[\big(b_{s,t}^{ij}\big)^2 \big(b_{s,t}^{kl}\big)^2\Big] & = \E\Big[\big(b_{s,t}^{ij}\big)^2\Big] \E\Big[\big(b_{s,t}^{kl}\big)^2\Big] = \frac{1}{144}h^4. 
\end{align*}
Therefore, expanding the left-hand side of (\ref{eq:fourth_expand}) and using moments (\ref{eq:bb_moments}) and (\ref{eq:cross_moment1}) gives,
\begin{align*}
&4\m\E\Big[\big(b_{s,t}^{ij}\big)^4\Big] + 12\m\E\Big[\big(b_{s,t}^{ij}\big)^2 \big(b_{s,t}^{kl}\big)^2\Big] + 24\m\E\Big[\big(b_{s,t}^{ij}\big)^2 \big(b_{s,t}^{jk}\big)^2\Big] + 24\m\E\Big[b_{s,t}^{ij}b_{s,t}^{\m jk}b_{s,t}^{kl}b_{s,t}^{li}\Big] = \frac{7}{15}h^4\\[3pt]
&\implies\, \E\Big[b_{s,t}^{ij}b_{s,t}^{\m jk}b_{s,t}^{kl}b_{s,t}^{li}\Big] = \frac{1}{24}\bigg(\frac{7}{15} - 4\times \frac{7}{240}- 12\times \frac{1}{144} - 24\times \frac{7}{720}\bigg)h^4 = \frac{1}{720}h^4,
\end{align*}
as required.
\end{proof}

Finally, we will consider the remaining moments of the Brownian bridge L\'{e}vy area. Using the symmetry of $B$, it will be straightforward to show these moments are zero. \medbreak

\begin{theorem}[Zero moments of $b_{s,t}$]\label{thm:levy_area_moments3}
Apart from the moments in Theorems \ref{thm:bb_distribution}, \ref{thm:levy_area_moments1}, \ref{thm:levy_area_moments2} and Remark \ref{rmk:stt_levy}\m, all other degree $k$ moments of the pair $(H_{s,t}\m, b_{s,t})$ with $k\leq 5$ will be zero.
\end{theorem}
\begin{proof}
We will first consider the following moment,
\begin{align}\label{eq:target_moment}
\E\Big[\big(H_{s,t}^{i_1}\cdots H_{s,t}^{i_m}\big)\big(b_{s,t}^{\m j_1\m, j_2}\cdots b_{s,t}^{\m j_{n-1}\m, j_n}\big)\Big],
\end{align}
where $I = \big(i_1\m,\cdots, i_m\big)$ and $J = \big(j_1\m,\cdots, j_n\big)$ are multi-indices with $i_k\m, j_k\in\{1,\cdots,d\}$. For each index $k\in\{1,\cdots,d\}$, we can count its number of appearances in $I$ and $J$ as
\begin{align*}
N_k := |\m p : i_p = k, i_p\in I\m| + |\m q : j_q = k, j_q\in J\m|.
\end{align*}
We will first consider the case when there exists $k\in\{1,\cdots, d\}$ such that $N_k$ is odd.
Letting $k$ be fixed, we define a new Brownian bridge $\overline{B}$ by flipping the $k$-th coordinate,
\begin{align*}
\overline{B}^{\m i} := \begin{cases} B^i, & \text{ if }i\neq k,\\ -B^i, & \text{ if }i = k.\end{cases}
\end{align*}
Similarly, we define the associated space-time L\'{e}vy area and space-space L\'{e}vy area,
\begin{align*}
\overline{H}_{s,t}^{\m i} :=  \begin{cases} H_{s,t}^i, & \text{ if }i\neq k,\\[2pt] -H_{s,t}^i, & \text{ if }i = k,\end{cases}\hspace{10mm}\overline{b}_{s,t}^{\m ij} :=  \begin{cases} b_{s,t}^{ij}, & \text{ if }k\notin\{i,j\}\text{ or }i = j,\\[2pt] -b_{s,t}^{ij}, & \text{ otherwise}.\end{cases}
\end{align*}
As $B$ and $\overline{B}$ have the same distribution, we see that the following moments are equal,
\begin{align*}
\E\Big[\big(H_{s,t}^{i_1}\cdots H_{s,t}^{i_m}\big)\big(b_{s,t}^{\m j_1\m, j_2}\cdots b_{s,t}^{\m j_{n-1}\m, j_n}\big)\Big] & = \E\Big[\big(\m\overline{H}_{s,t}^{\m i_1}\cdots \overline{H}_{s,t}^{\m i_m}\big)\big(\m\overline{b}_{s,t}^{\m j_1\m, j_2}\cdots \overline{b}_{s,t}^{\m j_{n-1}\m, j_n}\big)\Big]\\
& = (-1)^{N_k} \E\Big[\big(H_{s,t}^{i_1}\cdots H_{s,t}^{i_m}\big)\big(b_{s,t}^{\m j_1\m, j_2}\cdots b_{s,t}^{\m j_{n-1}\m, j_n}\big)\Big]. 
\end{align*}
Since we assume $N_k$ is odd, it immediately follows that the above expectation is zero.
Finally, we shall consider the case where $N_k$ is even for all indices $k\in\{1,\cdots, d\}$. Using that each coordinate of $B$ is independent and identically distributed, we have
\begin{align*}
\E\big[b_{s,t}^{ij}b_{s,t}^{jk}b_{s,t}^{ki}\big] & = \E\big[b_{s,t}^{kj}b_{s,t}^{\m ji}b_{s,t}^{ik}\big]\\[2pt]
& = -\E\big[b_{s,t}^{\m jk}b_{s,t}^{ij}b_{s,t}^{ki}\big],
\end{align*}
where the second lines follows by the antisymmetry of L\'{e}vy area, i.e.~$b_{s,t}^{ij} = - b_{s,t}^{\m ji}$.
Hence $\,\E\big[b_{s,t}^{ij}b_{s,t}^{jk}b_{s,t}^{ki}\big] = 0$. Using the conditional expectation in Theorem \ref{thm:levy_area_moments1}, we have
\begin{align*}
\E\big[H_{s,t}^i H_{s,t}^j b_{s,t}^{ij}\big] = \E\big[\E\big[H_{s,t}^i H_{s,t}^j b_{s,t}^{ij} \,|\, W_{s,t}\m, H_{s,t}\big]\big] & = 0.
\end{align*}
The remaining moments (with $N_k$ even) have been shown in Theorems \ref{thm:bb_distribution}, \ref{thm:levy_area_moments1}, \ref{thm:levy_area_moments2} or can simply be computed using independence of the coordinates in the Brownian bridge.
\end{proof}\medbreak
\begin{remark}
Using Theorems \ref{thm:bb_distribution}, \ref{thm:levy_area_moments2} and \ref{thm:levy_area_moments3}, it is straightforward to compute any moment (with degree at most five) of Brownian L\'{e}vy area conditional on the increment $W_{s,t}\m$. To the author's knowledge, such conditional moments of $A$ were unknown prior to \cite{foster2020thesis}.
\end{remark}

\subsection{Moment-matching approximation of Brownian L\'{e}vy area}

In this section, we present the weak approximation \cite[Definition 7.3.5]{foster2020thesis} which we will show matches the majority of conditional moments established in sections \ref{sect:moments1} and \ref{sect:moments2}.\medbreak

\begin{definition}[Weak approximation of Brownian L\'{e}vy area]\label{def:weak_levy_area}
Using the increment $W_{s,t}\m$, space-time L\'{e}vy area $H_{s,t}$ and space-time-time L\'{e}vy area $K_{s,t}$ on $[s,t]$, we define
\begin{align}\label{eq:weak_levy_area}
\widetilde{A}_{s,t}^{\m ij} := H_{s,t}^i W_{s,t}^j - W_{s,t}^i H_{s,t}^j + 12\big(K_{s,t}^i H_{s,t}^j - H_{s,t}^i K_{s,t}^j\big) + \widetilde{a}_{s,t}^{\m ij}\m,
\end{align}
where, conditional on $K_{s,t}\m$, the $d\times d$ random matrix $\m\widetilde{a}_{s,t}\m$ has entries given by
\begin{align*}
\widetilde{a}_{s,t}^{\m ij} := \begin{cases}\sigma_{s,t}^{ij}\, \xi_{s,t}^{ij}\m, & \text{ if }\,\, i < j,\\[3pt]
-\sigma_{s,t}^{ji}\, \xi_{s,t}^{\m ji}\m, & \text{ if }\,\, i > j,\\[3pt]
0, & \text{ if }\,\,i = j,
\end{cases}
\end{align*}
for $1\leq i,j\leq d$ with the independent random variables $\sigma_{s,t}^{ij}$ and $\xi_{s,t}^{ij}$ defined for $i < j$ as
\begin{align*}
\xi_{s,t}^{ij} &\sim \begin{cases} \mathrm{Uniform}\big[-\sqrt{3}\m, \sqrt{3}\big], & \text{with probability }\,p,\\[3pt] \mathrm{Rademacher}, & \text{with probability }\,1 - p,\end{cases}\\[6pt]
\sigma_{s,t}^{ij} & := \sqrt{\frac{3}{28}\big(C^{\m i} + c\big)\big(C^j + c\big)h^2 + \frac{1}{28}h\Big(\big(12 K_{s,t}^i\big)^2 + \big(12 K_{s,t}^j\big)^2\Big)},
\end{align*}
where $c, p$ are constants and $\{C^{\m i}\}$ are independent exponential random variables with
\begin{align*}
C^{\m i} & \sim \mathrm{Exponential}\Big(\frac{15}{8}\Big),\\[-2pt]
c & := \frac{\sqrt{3}}{3} - \frac{8}{15}\m,\\[3pt]
p & := \frac{21130}{25621}\m.
\end{align*}
\end{definition}

\begin{theorem}\label{thm:match_moments}
Let $\widetilde{A}_{s,t}$ be the weak approximation of $A_{s,t}$ given by Definition \ref{def:weak_levy_area}. Then
\begin{align}
\E\Big[\big(\widetilde{A}_{s,t}^{ij}\big)^{n_1}\,\big|\, W_{s,t}\m, H_{s,t}\m, K_{s,t}\Big] & = \E\Big[\big(A_{s,t}^{ij}\big)^{n_1}\,\big|\, W_{s,t}\m, H_{s,t}\m, K_{s,t}\Big],\label{eq:match_k_moments}\\
\E\Big[\big(\widetilde{A}_{s,t}^{ij}\big)^{n_2}\,\big|\, W_{s,t}\m, H_{s,t}\Big] & = \E\Big[\big(A_{s,t}^{ij}\big)^{n_2}\,\big|\, W_{s,t}\m, H_{s,t}\Big],\label{eq:match_h_moments}\\
\E\Big[\big(\widetilde{A}_{s,t}^{ij}\big)^{n_3}\big(\widetilde{A}_{s,t}^{jk}\big)^{n_4}\big(\widetilde{A}_{s,t}^{ki}\big)^{n_5}\,\big|\, W_{s,t}\Big] & = \E\Big[\big(A_{s,t}^{ij}\big)^{n_3}\big(A_{s,t}^{jk}\big)^{n_4}\big(A_{s,t}^{ki}\big)^{n_5}\,\big|\, W_{s,t}\Big],\label{eq:match_w_moments}
\end{align}
for distinct $i,j,k\in\{1,\cdots\hspace{-0.25mm},d\}$ where $n_m \geq 0$ with $n_1 \leq 3$, $n_2\leq 5$ and $n_3 + n_4 + n_5 \leq 5$.
\end{theorem}
\begin{proof}
Combining equations (\ref{eq:levy_area_decomp_2}) and (\ref{eq:arch_area_expand}), gives the following decomposition of $A_{s,t}\m$,
\begin{align}\label{eq:poly_expand}
A_{s,t}^{ij} = H_{s,t}^i W_{s,t}^j - W_{s,t}^i H_{s,t}^j & + 12\big(K_{s,t}^i H_{s,t}^j - H_{s,t}^i K_{s,t}^j\big)\\
& + 720\big(M_{s,t}^i K_{s,t}^j - K_{s,t}^i M_{s,t}^j\big) +  c_{s,t}^{ij}\m,\nonumber
\end{align}
where $M_{s,t} \sim\mathcal{N}\big(0, \frac{1}{100800}h I_d\big)$ and the L\'{e}vy area $c_{s,t}$ are defined by equations (\ref{eq:m_def})\vspace{0.5mm} and (\ref{eq:c_def}) respectively. In particular, $(M_{s,t}\m, c_{s,t})$ and $(W_{s,t}\m, H_{s,t}\m, K_{s,t})$ are independent. Since equation (\ref{eq:poly_expand}) can be further extended to a polynomial expansion of L\'{e}vy area \cite{foster2023levyarea},
by using the same techniques as in the proof of Theorem \ref{thm:levy_decomp}, we obtain the moments
\begin{align*}
\E\Big[A_{s,t}^{ij}\,\big|\, W_{s,t}\m, H_{s,t}\m, K_{s,t}\Big] & = H_{s,t}^i W_{s,t}^j - W_{s,t}^i H_{s,t}^j + 12\big(K_{s,t}^i H_{s,t}^j - H_{s,t}^i K_{s,t}^j\big),\\
\var\big(A_{s,t}^{ij} \,|\, W_{s,t}\m, H_{s,t}\m, K_{s,t}\big) & =\frac{1}{28}h^2 + \frac{36}{7} h\Big(\big(K_{s,t}^i\big)^2 + \big(K_{s,t}^j\big)^2\Big),\\[3pt]
\mathrm{Skew}\big(A_{s,t}^{ij} \,|\, W_{s,t}\m, H_{s,t}\m, K_{s,t}\big) & = 0,
\end{align*}
which are clearly matched by the weak approximation $\widetilde{A}_{s,t}\m$ (i.e.~equation (\ref{eq:match_k_moments}) holds).
Moreover, applying the tower law shows the corresponding expectations in equations (\ref{eq:match_k_moments}) and (\ref{eq:match_h_moments}) are also matched. We now consider the following fourth moment of $\widetilde{a}_{s,t}\m$,
\begin{align*}
\E\Big[\big(\,\widetilde{a}_{s,t}^{\m ij}\big)^4 \Big] & = \E\Big[\big(\sigma_{s,t}^{ij}\big)^4\Big]\E\Big[\big(\xi_{s,t}^{ij}\big)^4\Big]\\[3pt]
& = \bigg(\frac{9}{5}\cdot p + 1\cdot(1-p)\bigg)\\[-1pt]
&\mmm\E\bigg[\Big(\frac{3}{28}\big(C^{\m i} + c\big)\big(C^j + c\big)h^2 + \frac{1}{28}h\Big(\big(12 K_{s,t}^i\big)^2 + \big(12 K_{s,t}^j\big)^2\Big)\Big)^2\,\bigg]\\[3pt]
& = \frac{42525}{25621}\Bigg(\bigg(\frac{3}{28}\big(C^{\m i} + c\big)\big(C^j + c\big)h^2\bigg)^2 + \bigg(\frac{1}{28}h\Big(\big(12 K_{s,t}^i\big)^2 + \big(12 K_{s,t}^j\big)^2\Big)\bigg)^2\\
&\hspace{17.5mm} + 2\bigg(\frac{3}{28}\big(C^{\m i} + c\big)\big(C^j + c\big)h^2\bigg)\bigg(\frac{1}{28}h\Big(\big(12 K_{s,t}^i\big)^2 + \big(12 K_{s,t}^j\big)^2\Big)\bigg)\Bigg)\\
& = \frac{42525}{25621}\Bigg(\frac{9}{784}\bigg(\Big(\frac{8}{15}\Big)^2 + \frac{1}{3}\bigg)^2 + 2\times \frac{3}{28}\times\frac{1}{3}\times \frac{1}{28}\times\frac{2}{5} + \frac{1}{784}\times\frac{1}{25}\times 8\Bigg)h^4\\[3pt]
& = \frac{27}{2800}\m h^4.
\end{align*}
We also note that $\widetilde{a}_{s,t}$ is independent of $(W_{s,t}\m, H_{s,t})$ and correctly correlated with $K_{s,t}\m$.
\begin{align*}
\E\Big[\big(K_{s,t}^i\big)^n\big(\widetilde{a}_{s,t}^{\m ij}\big)^m\Big]
& = \E\Big[\E\Big[\big(K_{s,t}^i\big)^n\big(\widetilde{a}_{s,t}^{\m ij}\big)^m\,\big|\, K_{s,t}\Big]\Big]\\
& = \E\Big[\big(K_{s,t}^i\big)^n\m\E\Big[\big(\widetilde{a}_{s,t}^{\m ij}\big)^m\,\big|\, K_{s,t}\Big]\Big]\\
& = \E\Big[\big(K_{s,t}^i\big)^n\m\E\Big[\big(a_{s,t}^{ij}\big)^m\,\big|\, K_{s,t}\Big]\Big]\\
& = \E\Big[\E\Big[\big(K_{s,t}^i\big)^n\big(a_{s,t}^{ij}\big)^m\,\big|\, K_{s,t}\Big]\\
& = \E\Big[\big(K_{s,t}^i\big)^n\big(a_{s,t}^{\m ij}\big)^m\Big],
\end{align*}
for $n\geq 0$ and $m\in\{0,1,2,3\}$. Since these moments coincide for $\widetilde{a}_{s,t}^{\m ij}$ and $a_{s,t}^{\m ij}\m$, we have
\begin{align*}
&\E\Big[\big(H_{s,t}^i W_{s,t}^j - W_{s,t}^i H_{s,t}^j + 12\big(K_{s,t}^i H_{s,t}^j - H_{s,t}^i K_{s,t}^j\big) + \widetilde{a}_{s,t}^{\m ij}\big)^4\,\big|\, W_{s,t}\m, H_{s,t}\Big]\\
&\mm = \E\Big[\big(H_{s,t}^i W_{s,t}^j - W_{s,t}^i H_{s,t}^j + 12\big(K_{s,t}^i H_{s,t}^j - H_{s,t}^i K_{s,t}^j\big) + a_{s,t}^{\m ij}\big)^4\,\big|\, W_{s,t}\m, H_{s,t}\Big],
\end{align*}
which gives equation (\ref{eq:match_h_moments}) for $n_2 = 4$. The moments also match when $n_2 = 5$ since the symmetry of $(K_{s,t}\m,\widetilde{a}_{s,t})$ and $(K_{s,t}\m,a_{s,t})$ means that all of the odd moments are zero.
We now consider some cross-moments of the approximation which appear when $d\geq 3$.
\begin{align*}
\E\Big[\big(\widetilde{a}_{s,t}^{\m ij}\big)^2\big(\widetilde{a}_{s,t}^{\m jk}\big)^2\Big]
& = \E\Big[\Big(\frac{3}{28}\big(C^{\m i} + c\big)\big(C^j + c\big)h^2 + \frac{1}{28}h\Big(\big(12 K_{s,t}^i\big)^2 + \big(12 K_{s,t}^j\big)^2\Big)\Big)\\
&\hspace{12.5mm}\Big(\frac{3}{28}\big(C^{\m j} + c\big)\big(C^k + c\big)h^2 + \frac{1}{28}h\Big(\big(12 K_{s,t}^j\big)^2 + \big(12 K_{s,t}^k\big)^2\Big)\Big)\Big]\\[1pt]
& = \Big(\frac{3}{28}\Big)^2 \E\big[C^{\m i} + c\big]\E\Big[\big(C^{\m j} + c\big)^2\Big]\E\big[C^{\m k} + c\big] h^4\\
&\mm + \frac{6}{28^2}\m\E\big[\big(C^{\m i} + c\big)\big(C^j + c\big)\big]\E\Big[\big(12 K_{s,t}^j\big)^2 + \big(12 K_{s,t}^k\big)^2\Big] h^3\\
&\mm + \frac{1}{28^2}\m\E\Big[\Big(\big(12 K_{s,t}^i\big)^2 + \big(12 K_{s,t}^j\big)^2\Big)\Big(\big(12 K_{s,t}^j\big)^2 + \big(12 K_{s,t}^k\big)^2\Big)\Big] h^2\\[1pt]
& = \bigg(\frac{3}{28^2}\Big(\frac{1}{3} + \Big(\frac{8}{15}\Big)^2\m\Big) + \frac{2}{28^2}\times \frac{2}{5} + \frac{1}{28^2}\times\frac{6}{25}\bigg)h^4 = \frac{31}{8400}h^4\m.
\end{align*}
Letting $\widetilde{b}_{s,t}^{\m ij} := 12\big(K_{s,t}^i H_{s,t}^j - H_{s,t}^i K_{s,t}^j\big) + \widetilde{a}_{s,t}^{\m ij}$ be the approximate bridge area, we have
\begin{align*}
\E\Big[\big(\m\widetilde{b}_{s,t}^{\m ij}\big)^2\big(\m\widetilde{b}_{s,t}^{\m jk}\big)^2\Big] & = \E\Big[\big(12\big(K_{s,t}^i H_{s,t}^j - H_{s,t}^i K_{s,t}^j\big)\big)^2\big(12\big(K_{s,t}^j H_{s,t}^k - H_{s,t}^j K_{s,t}^k\big)\big)^2\Big]\\
&\mm + \E\Big[\big(\m\widetilde{a}_{s,t}^{\m ij}\big)^2\big(12\big(K_{s,t}^j H_{s,t}^k - H_{s,t}^j K_{s,t}^k\big)\big)^2\Big]\\
&\mm + \E\Big[\big(12\big(K_{s,t}^i H_{s,t}^j - H_{s,t}^i K_{s,t}^j\big)\big)^2\big(\m\widetilde{a}_{s,t}^{\m jk}\big)^2\Big] + \E\Big[\big(\m\widetilde{a}_{s,t}^{\m ij}\big)^2\big(\m\widetilde{a}_{s,t}^{\m jk}\big)^2\Big]\\[1pt]
& = \Big(\frac{1}{144}\times \frac{1}{25} \Big)\Big(1 + 3 + 3 + 1\Big)h^4 + 2\m\E\Big[\big(\m\widetilde{a}_{s,t}^{\m ij}\big)^2\big(12\big(K_{s,t}^j H_{s,t}^k\big)\big)^2\Big]\\
&\mm + 2\m\E\Big[\big(\m\widetilde{a}_{s,t}^{\m ij}\big)^2\Big]\E\Big[\big(12\big(H_{s,t}^j K_{s,t}^k\big)\big)^2\Big] + \frac{31}{8400}\m h^4\\[2pt]
& = \bigg(\frac{1}{450} + 24\bigg(\frac{1}{720}\times \frac{1}{28} + \frac{3}{28}\times \frac{144}{720^2} + \frac{1}{28}\times \frac{144}{720^2}\bigg)\\
&\hspace{13.5mm} + 2\m\Big(\frac{1}{28} + \frac{1}{28}\times\frac{2}{5}\Big)\times \frac{1}{5}\times\frac{1}{12} + \frac{31}{8400}\bigg) h^4 = \frac{7}{720}\m h^4,
\end{align*}
\begin{align*}
&\E\Big[H_{s,t}^i H_{s,t}^k\m \widetilde{b}_{s,t}^{\m ij}\widetilde{b}_{s,t}^{\m jk}\Big]\\
&\mm = \E\Big[H_{s,t}^i H_{s,t}^k\big(12\big(K_{s,t}^i H_{s,t}^j - H_{s,t}^i K_{s,t}^j\big) + \widetilde{a}_{s,t}^{\m ij}\big)\big(12\big(K_{s,t}^j H_{s,t}^k - H_{s,t}^j K_{s,t}^k\big) + \widetilde{a}_{s,t}^{\m jk}\big)\Big]\\[3pt]
&\mm = \E\Big[H_{s,t}^i H_{s,t}^k\big(12\big(K_{s,t}^i H_{s,t}^j - H_{s,t}^i K_{s,t}^j\big)\big)\big(12\big(K_{s,t}^j H_{s,t}^k - H_{s,t}^j K_{s,t}^k\big)\big)\Big]\\
&\mm = -144\m\E\Big[\big(H_{s,t}^i\big)^2 \big(H_{s,t}^k\big)^2 \big(K_{s,t}^j\big)^2\Big] = -\frac{1}{720}\m h^3.
\end{align*}
By the independence of $\xi_{s,t}\m$, it follows that the odd moments of $\big(W_{s,t}\m, H_{s,t}\m, K_{s,t}\m, \widetilde{a}_{s,t}\big)$ are zero -- and hence match the moments of $\big(W_{s,t}\m, H_{s,t}\m, K_{s,t}\m, a_{s,t}\big)$. Therefore, since $\widetilde{b}_{s,t}$ has the same fourth moments as $b_{s,t}\m$, the conditional expectations (\ref{eq:match_w_moments}) also match.
\end{proof}
\begin{remark}\label{rmk:not_perfect}The proposed weak approximation does not match all fourth moments\vspace{1mm} since $\m\E\big[b_{s,t}^{ij}b_{s,t}^{\m jk}b_{s,t}^{kl}b_{s,t}^{li}\big] = \frac{1}{720}h^4$ and $\m\E\big[\,\widetilde{b}_{s,t}^{ij}\m\widetilde{b}_{s,t}^{\m jk}\m\widetilde{b}_{s,t}^{kl}\m\widetilde{b}_{s,t}^{li}\big] = 2\times\frac{1}{144}\m h^2\times\frac{1}{25}\m h^2 = \frac{1}{1800}h^4$.
\end{remark}\medbreak

From the experiments in section \ref{sect:ss_examples}, we see that the weak approximation given by Definition \ref{def:weak_levy_area} can outperform standard and recent L\'{e}vy area approximations \cite{jelincic2025levygan, davie2014area, mrongowius2022area}.\medbreak

However, due to Remark \ref{rmk:not_perfect}, the approximation will not match the fourth moment $\E\big[A_{s,t}^{ij}A_{s,t}^{\m jk}A_{s,t}^{kl}A_{s,t}^{li}\,\big|\, W_{s,t}\,\big]$. In an effort to reduce this mismatch for the case when $d=4$, we consider replacing $\xi_{s,t}$ with a random matrix $\wideparen{\xi}_{s,t}$ such that $\E\big[\wideparen{\xi}_{s,t}^{\m ij}\wideparen{\xi}_{s,t}^{\m jk}\wideparen{\xi}_{s,t}^{\m kl}\wideparen{\xi}_{s,t}^{\, li}\big] > 0$.\medbreak

\begin{definition}[Subtly correlated random matrix]\label{def:subtle_random}
Let $Z^{1,2}$, $Z^{1,3}$, $Z^{1,4}$ and $Z^{2,3}$ denote independent Rademacher random variables and suppose $U := (U_1, U_2) \sim \mathrm{Uniform}\big(\{(1,1), (1,-1), (-1,-1)\}\big)$. We define random variables $Z^{2,4}$ and $Z^{3,4}$ as
\begin{align*}
Z^{2,4} & := Z^{2,3}Z^{1,3}Z^{1,4} U_1\m,\hspace{5mm}Z^{3,4} := Z^{2,3}Z^{1,2}Z^{1,4} U_2\m.
\end{align*}
Then, using $\xi_{s,t}$ from definition \ref{def:weak_levy_area}, we define a ``subtly correlated'' random matrix as
\begin{align*}
\wideparen{\xi}_{s,t}^{\m ij} := \begin{cases} |\m\xi_{s,t}^{ij}\m| Z^{ij}, & \text{if }i < j,\\[2pt]
-|\m\xi_{s,t}^{ij}\m| Z^{ij}, & \text{if }i > j,\\[1pt]
\,0, & \text{if }i = j.
\end{cases}
\end{align*}
\end{definition}

\begin{theorem}\label{thm:mismatch}The matrices $\xi_{s,t}$ and $\wideparen{\xi}_{s,t}$ have the same moments up to degree 5 except
\begin{align*}
\E\big[\m\wideparen{\xi}_{s,t}^{\m ij}\wideparen{\xi}_{s,t}^{\m jk}\wideparen{\xi}_{s,t}^{\m kl}\wideparen{\xi}_{s,t}^{\m li}\big] & = \frac{3841489152965701849}{20683572037203882288} + \frac{29202537105993615}{430907750775080881}\m\sqrt{3}\approx 0.30\,\,\,(\text{2.d.p})\m,
\end{align*}
for distinct $i,j,k,l\in\{1,2,3,4\}$ where we define $Z^{ji} := -Z^{ij}$ for $i > j$.
\end{theorem}
\begin{proof}
By the definitions of $Z^{2,4}$ and $Z^{3,4}$, the following fourth moments are non-zero,
\begin{align*}
\E\big[Z^{1,2}Z^{2,3}Z^{3,4}Z^{4,1}\big]
& = -\E\big[Z^{1,2}Z^{2,3}\big(Z^{2,3}Z^{1,2}Z^{1,4}U_2\big)Z^{1,4}\big] = -\E\big[U_2\big] = \frac{1}{3}\m,\\
\E\big[Z^{1,4}Z^{4,2}Z^{2,3}Z^{3,1}\big]
& = \E\big[Z^{1,4}\big(Z^{2,3}Z^{1,3}Z^{1,4}U_1\big)Z^{2,3}Z^{1,3}\big] = \E\big[U_1\big] = \frac{1}{3}\m,\\
\E\big[Z^{1,3}Z^{3,4}Z^{4,2}Z^{2,1}\big] & = \E\big[Z^{1,3} \big(Z^{2,3}Z^{1,2}Z^{1,4} U_2\big)\big(Z^{2,3}Z^{1,3}Z^{1,4}U_1\big)Z^{1,2}\big] = \E\big[U_1 U_2\big] = \frac{1}{3}\m.
\end{align*}
On the other hand, it can be verified by direct calculation that the remaining fourth moments of $Z$ will be zero. This can be shown as, even after simplifying the products, there will still be at least one independent Rademacher random variable remaining.
Similarly, it is straightforward to show that the other moments $\xi_{s,t}$ and $\wideparen{\xi}_{s,t}$ match.
Finally, by the independence of $\{\xi_{s,t}^{ij}\m, \xi_{s,t}^{jk}\m, \xi_{s,t}^{kl}\m, \xi_{s,t}^{li}\}$ and the value $p = \frac{21130}{25621}\m$, we have
\begin{align*}
\E\big[|\m\xi_{s,t}^{ij}||\m\xi_{s,t}^{jk}||\m\xi_{s,t}^{kl}||\m\xi_{s,t}^{li}|\big] & = \bigg(p + (1-p)\frac{\sqrt{3}}{2}\m\bigg)^4 = \bigg(\frac{21130}{25621} + \frac{4491}{51242}\,\sqrt{3}\bigg)^4\\[3pt]
& = \frac{3841489152965701849}{6894524012401294096} + \frac{87607611317980845}{430907750775080881}\,\sqrt{3}\m.
\end{align*}
The result now immediately follows by the independence of $\xi$ and $Z$.
\end{proof}\medbreak

Due to Theorem \ref{thm:mismatch}, we would expect that using $\wideparen{\xi}_{s,t}$ instead of $\xi_{s,t}$ in Definition \ref{def:weak_levy_area} (when $d=4$), should result in an improved weak approximation $\wideparen{A}_{s,t}$ as it still satisfies Theorem \ref{thm:match_moments} whilst better matching the non-trivial fourth moment $\E\big[A_{s,t}^{ij}A_{s,t}^{\m jk}A_{s,t}^{kl}A_{s,t}^{li}\,\big]$.\medbreak

\begin{definition}[Modified weak approximation of L\'{e}vy area]
Using the increment $W_{s,t}\m$, space-time L\'{e}vy area $H_{s,t}$ and space-time-time L\'{e}vy area $K_{s,t}$ on $[s,t]$, we define
\begin{align}\label{eq:weak_levy_area2}
\wideparen{A}_{s,t}^{\m ij} := H_{s,t}^i W_{s,t}^j - W_{s,t}^i H_{s,t}^j + 12\big(K_{s,t}^i H_{s,t}^j - H_{s,t}^i K_{s,t}^j\big) + \wideparen{a}_{s,t}^{\m ij}\m,
\end{align}
where, conditional on $K_{s,t}\m$, the $d\times d$ random matrix $\m\wideparen{a}_{s,t}\m$ has entries given by
\begin{align*}
\wideparen{a}_{s,t}^{\m ij} := \begin{cases}\sigma_{s,t}^{ij}\, \wideparen{\xi}_{s,t}^{\m ij}\m, & \text{ if }\,\, i > j,\\[3pt]
-\sigma_{s,t}^{ij}\, \wideparen{\xi}_{s,t}^{\m ji}\m, & \text{ if }\,\, i < j,\\[3pt]
0, & \text{ if }\,\,i = j,
\end{cases}
\end{align*}
for $1\leq i,j\leq d$ where $\sigma_{s,t}$ was given in Definition \ref{def:weak_levy_area} and $\wideparen{\xi}_{s,t}$ was given in Definition \ref{def:subtle_random}.
\end{definition}
\begin{remark}
By applying either Monte Carlo simulation or Gaussian quadrature to estimate $\E\big[\m\sigma_{0,1}^{\m ij}\m\sigma_{0,1}^{\m jk}\m\sigma_{0,1}^{\m kl}\m\sigma_{0,1}^{\m li}\big]$, we observe that the associated fourth moment of $\wideparen{a}_{s,t}$ is
\begin{align*}
& \,\,\E\big[\,\wideparen{a}_{s,t}^{\m ij}\m\wideparen{a}_{s,t}^{\m jk}\m\wideparen{a}_{s,t}^{\m kl}\m\wideparen{a}_{s,t}^{\m li}\big] \approx 0.00074\m h^4\,\,(\text{2.s.f}\m),\\[3pt]
\text{which would imply that } &\\[3pt]
&\,\,\,\,\E\big[\,\wideparen{b}_{s,t}^{\m ij}\m\wideparen{b}_{s,t}^{\m jk}\m\wideparen{b}_{s,t}^{\m kl}\m\wideparen{b}_{s,t}^{\m li}\big] \approx 0.0013\m h^4\,\,(\text{2.s.f}\m),\hspace*{40mm}
\end{align*}
where $\wideparen{b}_{s,t} := 12\big(K_{s,t}\otimes H_{s,t} - H_{s,t}\otimes K_{s,t}\big) + \wideparen{a}_{s,t}\m$ denotes the weak approximation of the Brownian bridge L\'{e}vy area. Note that this is substantially closer to the true value of\vspace{0.5mm} $\E\big[b_{s,t}^{ij}b_{s,t}^{\m jk}b_{s,t}^{kl}b_{s,t}^{li}\big] = \frac{1}{720}h^4 \approx 0.0014\m h^4$ than $\m\E\big[\,\widetilde{b}_{s,t}^{ij}\m\widetilde{b}_{s,t}^{\m jk}\m\widetilde{b}_{s,t}^{kl}\m\widetilde{b}_{s,t}^{li}\big] = \frac{1}{1800}h^4\approx 0.00056\m h^4$.\vspace{0.5mm} However, it is unclear how to extend the new approximation (\ref{eq:weak_levy_area2}) to higher dimensions.
\end{remark}

\section{Space-space-time L\'{e}vy area of Brownian motion}\label{sect:space_space_time}

In this section, we will consider approximations for the ``space-space-time'' L\'{e}vy area.
This will be relevant when designing high order numerical methods for SDEs satisfying the commutativity condition (\ref{eq:intro_commute}), which includes both additive and scalar noise types.
We note that the main results of the section (Theorems \ref{thm:sst_approx_whk} and \ref{thm:sst_cond_var}) are new, but already known when $d=1$ (see \cite{foster2020poly, foster2020thesis}). We first introduce some useful notation for integrals.

\subsection{The shuffle product for calculations with iterated integrals}\label{sect:shuffle}

\begin{definition}\label{def:shuffle_integral}
Let $\langle\mathcal{A}_d\rangle$ denote the set of words with letters from $\mathcal{A}_d := \{0,1,\cdots, d\}$. Let $\R\langle\mathcal{A}_d\rangle$ be the space of non-commutative polynomials in $\mathcal{A}_d$ with real coefficients. We can identify linear combinations of integrals with elements in $\R\langle\mathcal{A}_d\rangle$ by $I_e := 1$ and
\begin{align*}
i_1\cdots\m i_m & \longleftrightarrow I_{i_1\cdots\m i_m} :=
    \int_s^t \int_s^{r_1} \dots \int_s^{r_{m-1}} 
    \circ \, dW^{i_1}_{r_n} \circ dW^{i_2}_{r_{m-1}}\cdots \circ dW^{i_{m-1}}_{r_2} \circ dW^{i_m}_{r_1}\m,\\[3pt]
     \lambda u+\mu v & \longleftrightarrow I_{\lambda u + \mu v} := \lambda I_u + \mu I_v\m,
\end{align*}
for $m\geq 0$, $i_1\m, i_2\m,\cdots, i_m\in\mathcal{A}_d\m$, $u,v\in \mathcal{A}_d^\ast$ and $\lambda\m, \mu\in\R$, where we define $W_t^0 := t\m$.
\end{definition}\medbreak

Using the above notation, we can express integration by parts as a ``shuffle'' project.\medbreak

\begin{definition}\label{def:shuffle}
Suppose that $\mathcal{A}_d$ is a set containing $d$ letters and let $\R\langle\mathcal{A}_d\rangle$ be the corresponding space of non-commutative polynomials in $\mathcal{A}_d$ with real coefficients.
Then the shuffle product $\,\textshuffle :  \R\langle\mathcal{A}_d\rangle \times \R\langle\mathcal{A}_d\rangle \rightarrow \R\langle\mathcal{A}_d\rangle$ is the unique bilinear map such that
\begin{align*}
ua\shuffle\,vb & = (u\shuffle vb)\m a + (ua \shuffle v)\m b,\\[3pt]
u\shuffle e & = e \shuffle u = u,
\end{align*}
where $e$ denotes the empty letter.
\end{definition}\medbreak

This shuffle project will become useful for expanding products of iterated integrals.\medbreak

\begin{theorem}[Integration by parts formula]\label{thm:shuffle}
For all $\m u,v\in\R\langle\mathcal{A}_d\rangle$, we have that
\begin{align}\label{eq:shuffle_for_integrals}
I_{u}\cdot I_{v} = I_{u\subshuffle v}\m.
\end{align}
\end{theorem}
\begin{proof}
It is clear that the identity (\ref{eq:shuffle_for_integrals}) holds when $u=e$ or $v=e$ since $I_e=1$.
Suppose that (\ref{eq:shuffle_for_integrals}) holds for all words $u,v\in\mathcal{A}_d^\ast$ with a combined length less than $m$. Then for words $u,v\in\mathcal{A}_d^\ast\m$ and letters $a,b\in\mathcal{A}_d\m$ such that $|ua| + |vb| = m$, we have
\begin{align*}
\int_s^t I_u (r) \circ dW_r^{a}\int_s^t I_v(r)\circ dW_r^{b} & = \int_s^t\bigg(\int_s^{r_1} I_u(r_2)\circ dW_{r_2}^{a}\bigg)\circ d\bigg(\int_s^{r_1} I_v(r_2)\circ dW_{r_2}^{b}\bigg)\\
&\mm + \int_s^t\bigg(\int_s^{r_1} I_v(r_2)\circ dW_{r_2}^{b}\bigg)\circ d\bigg(\int_s^{r_1} I_u(r_2)\circ dW_{r_2}^{a}\bigg)\\
& = \int_s^t I_{ua}(r_1)\m I_v(r_1)\circ dW_{r_1}^{b} + \int_s^t I_{vb}(r_1)\m I_u(r_1)\circ dW_{r_1}^{a},
\end{align*}
where the second line uses integration by parts (which holds for Stratonovich integrals).
Since $ua$ and $vb$ have a combined length of $m$, applying the induction hypothesis gives
\begin{align*}
I_{ua}\cdot I_{vb} & = I_{(ua \subshuffle v)\m b} + I_{(u\subshuffle vb)\m a} = I_{(ua \subshuffle v)\m b + (u\subshuffle vb)\m a} = I_{ua\subshuffle vb}\m,
\end{align*}
which is an immediate consequence of the shuffle product's recursive definition.
\end{proof}

Using Theorem \ref{thm:shuffle}, it will be straightforward to rewrite products of integrals as linear combinations of (high order) integrals. In addition, it can be used to establish decompositions of iterated integrals into symmetric and antisymmetric components.\medbreak

\begin{theorem}[Decomposition of integrals into symmetric and antisymmetric parts]
\label{thm:symmetric_antisymmetric}
Let the Lie bracket $\m[\,\cdot\m,\m\cdot\m] : \R\langle\mathcal{A}_d\rangle\times \R\langle\mathcal{A}_d\rangle\rightarrow \R\langle\mathcal{A}_d\rangle$ be the unique bilinear map with
\begin{align}\label{eq:lie_bracket}
[u, v] = uv - vu\m,
\end{align}
for words $u,v\in \mathcal{A}_d^\ast\m$. Then, adopting the notation of Definition \ref{def:shuffle_integral} and Theorem \ref{thm:shuffle}, we have
\begin{align}
I_{ij} & = \frac{1}{2}I_i\cdot I_j + \frac{1}{2}I_{[i,j]}\m,\label{eq:ij_symmetric_antisymmetric}\\[3pt]
I_{ijk} & = \frac{1}{6}I_i\cdot I_j\cdot I_k 
    + \frac{1}{4} I_i\cdot I_{[j,k]} 
    + \frac{1}{4} I_{[i,j]}\cdot I_k 
    + \frac{1}{6} I_{[[i,j], k]} 
    + \frac{1}{6} I_{[i,[j,k]]} \m,\label{eq:ijk_symmetric_antisymmetric}
\end{align}
for $i,j,k\in\mathcal{A}_d\m$.
\end{theorem}
\begin{proof} The results follow by expanding the Lie brackets $[\,\cdot\m,\m\cdot\m]$ using (\ref{eq:lie_bracket}) and applying the shuffle product (which, by Theorem \ref{thm:shuffle}, is simply performing integration by parts).
\begin{align*}
&\hspace{-5mm}\frac{1}{6}I_i\cdot I_j\cdot I_k + \frac{1}{4} I_i\cdot I_{[j,k]} + \frac{1}{4} I_{[i,j]}\cdot I_k + \frac{1}{6} I_{[[i,j], k]} + \frac{1}{6} I_{[i,[j,k]]}\\
& = \frac{1}{6}(I_{ij} + I_{ji})\cdot I_k + \frac{1}{4}I_i\cdot I_{jk} - \frac{1}{4}I_i\cdot I_{kj} + \frac{1}{4} I_{ij}\cdot I_k -  \frac{1}{4} I_{ji}\cdot I_k\\
&\mm + \frac{1}{6} I_{[ij, k]} - \frac{1}{6} I_{[ji, k]} + \frac{1}{6} I_{[i,jk]} - \frac{1}{6} I_{[i,kj]}\\
& = \frac{1}{6}(I_{ijk} + I_{ikj} + I_{kij} + I_{jik} + I_{jki} + I_{kji})\\
&\mm + \frac{1}{4}(I_{ijk} + I_{jik} + I_{jki}) - \frac{1}{4}(I_{ikj} + I_{kij} + I_{kji})\\
&\mm + \frac{1}{4}(I_{ijk} + I_{ikj} + I_{kij}) - \frac{1}{4}(I_{jik} + I_{jki} + I_{kji})\\
&\mm + \frac{1}{6}(I_{ijk} - I_{jik} - I_{jik} + I_{kji}) + \frac{1}{6}(I_{ijk} - I_{jki} - I_{ikj} + I_{kji})\\
& = I_{ijk}\m,
\end{align*}
and just as for Theorem \ref{thm:levy_area_relation}, we have $I_{ij} = \frac{1}{2}(I_{ij} + I_{ji}) + \frac{1}{2}(I_{ij} - I_{ji}) = \frac{1}{2}I_i\cdot I_j + \frac{1}{2}I_{[i,j]}\m$.
\end{proof}

\subsection{Approximations of space-space-time L\'{e}vy area}

In practice, we are usually interested in approximating these third iterated integrals when $k = 0$ and the diffusion vector fields $\{g_i\}$ satisfy the commutativity condition:
\begin{align}\label{eq:commute_condition}
\hspace*{15mm} g_i^\prime(y) g_j(y) = g_j^\prime(y) g_i(y),\hspace{10mm}\forall y\in\R^e.
\end{align}
As discussed in Remark \ref{rmk:commute}, this simplifies the SDE's Taylor expansion and means that (space-space) L\'{e}vy area is not needed. However, this $(i,j)$-commutativity also extends to the third iterated integrals and therefore we only need to approximate the sums,
\begin{align*}
I_{ij0} + I_{ji0}\m,\quad I_{i0j} + I_{j0i}\m,\quad I_{0ij} + I_{0ji}\m.
\end{align*}
This leads us to the following definition for the ``space-space-time'' L\'{e}vy area of $W$.\medbreak
\begin{definition}[Space-space-time L\'{e}vy area]
Over an interval $[s,t]$, we will define the (space-space-time) L\'{e}vy area of Brownian motion as a $d\times d$ matrix $L_{s,t}$ with entries
\begin{align}
&\nonumber\\[-20pt]
L_{s,t}^{i,j} & := \frac{1}{12}\bigg(\int_s^t\int_s^u W_{s,v}^i \circ dW_v^j\, du + \int_s^t\int_s^u W_{s,v}^j \circ dW_v^i\, du\label{eq:sst_levy}\\[-1pt]
&\hspace{15mm} - 2\int_s^t\int_s^u W_{s,v}^i\, dv \circ dW_u^j - 2\int_s^t\int_s^u W_{s,v}^j\, dv \circ dW_u^i\nonumber\\
&\hspace{15mm} + \int_s^t\int_s^u (v-s)\, dW_v^i \circ dW_u^j + \int_s^t\int_s^u (v-s)\, dW_v^j \circ dW_u^i\bigg),\nonumber\\[-20pt] \nonumber
\end{align}
\end{definition}
\begin{remark}
Using the notation from section \ref{sect:shuffle}, we  have $L_{s,t}^{i,j} = \frac{1}{12}\big(I_{[i,[j,0]]} + I_{[j,[i,0]]}\big)$.
\end{remark}\medbreak
Similar to Theorem \ref{thm:whk_relation}, we can reconstruct iterated integrals using $W_{s,t}\m, H_{s,t}\m, L_{s,t}\m$.\medbreak
\begin{theorem}[Relationship between space-space-time L\'{e}vy area and other integrals]\label{thm:sst_levy_area}
\begin{align*}
&\\[-30pt]
\frac{1}{2}\big(I_{ij0} + I_{ji0}\big) & = \frac{1}{6}h W_{s,t}^{i}W_{s,t}^{j} 
    + \frac{1}{4}h\big(W_{s,t}^i H_{s,t}^j + H_{s,t}^i W_{s,t}^j\big)
    + L_{s,t}^{ij}\m,\\
\frac{1}{2}\big(I_{i0j} + I_{j0i}\big) & = \frac{1}{6}h W_{s,t}^{i}W_{s,t}^{j}
   - 2L_{s,t}^{ij}\m,\\
\frac{1}{2}\big(I_{0ij} + I_{0ji}\big) & = \frac{1}{6}h W_{s,t}^{i}W_{s,t}^{j}
	- \frac{1}{4}h\big(W_{s,t}^i H_{s,t}^j + H_{s,t}^i W_{s,t}^j\big)
    + L_{s,t}^{ij}\m.
\end{align*}
\end{theorem}
\begin{proof} By Theorem \ref{thm:symmetric_antisymmetric}, we have
\begin{align*}
&\\[-20pt]
I_{ij0} + I_{ji0} & =  \frac{1}{3}I_i I_j h 
    + \frac{1}{4} I_i\cdot I_{[j,0]} + \frac{1}{4} I_j\cdot I_{[i,0]}   
    + \frac{1}{6} I_{[i,[j,0]]} + \frac{1}{6} I_{[j,[i,0]]}\m,\\
I_{i0j} + I_{j0i} & =  \frac{1}{3}I_i I_j h
    + \frac{1}{3} I_{[[i,0], j]} 
    + \frac{1}{3} I_{[i,[0,j]]}\m,\\
I_{0ij} + I_{0ji} & = \frac{1}{3} I_i\cdot I_j h
    - \frac{1}{4} I_i\cdot I_{[j,0]} - \frac{1}{4} I_j\cdot I_{[i,0]}  
    + \frac{1}{6} I_{[[0,i], j]} + \frac{1}{6} I_{[[0,j], i]}\m.\\[-18pt]
\end{align*}
The result now follows as $h H_{s,t}^{i} = \int_s^t W_{s,u}\m du - \frac{1}{2}hW_{s,t}^i = \frac{1}{2}I_{[i,0]}$ and\\[6pt]
$\hspace*{7.5mm}L_{s,t}^{ij} = \frac{1}{12}\big(I_{[i,[j,0]]} + I_{[j,[i,0]]}\big) = -\frac{1}{12}\big(I_{[i,[0,j]]} + I_{[[i,0], j]}\big) = \frac{1}{12}\big(I_{[[0,j],i]} + I_{[[0,i],j]}\big).$
\end{proof}
\medbreak
\begin{corollary}\label{corr:simple_sst_levy}
Given $W_{s,t}\m, H_{s,t}\m,L_{s,t}$, we can recover the following Brownian integral,
\begin{align}\label{eq:simple_sst_levy}
\int_s^t W_{s,u}^i W_{s,u}^j\m du = \frac{1}{3}h W_{s,t}^{i}W_{s,t}^{j} 
    + \frac{1}{2}h\big(W_{s,t}^i H_{s,t}^j + H_{s,t}^i W_{s,t}^j\big)
    + 2L_{s,t}^{ij}\m.
\end{align}
\end{corollary}
\begin{proof}
Using integration by parts, we have $W_{s,u}^i W_{s,u}^j = \int_s^u W_{s,v}^i \circ dW_v^j + \int_s^u W_{s,v}^j \circ dW_v^i$. Thus, the left-hand side is equal to $I_{ij0} + I_{ji0}$ and the result follows by Theorem \ref{thm:sst_levy_area}.
\end{proof}

Using equation (\ref{eq:simple_sst_levy}), we shall derive our first approximation of the space-space-time L\'{e}vy area $L_{s,t}\m$, which will simply be its expectation conditional on $(W_{s,t}\m, H_{s,t}\m, K_{s,t})$.\medbreak
\begin{theorem}[Approximation of space-space-time L\'{e}vy area]\label{thm:sst_approx_whk} For $\m 0\leq s\leq t$, we have
\begin{align}
\E\big[L_{s,t}^{ij}\,\big|\, W_{s,t}\m, H_{s,t}\big] = \frac{3}{5}h H_{s,t}^i H_{s,t}^j & + \frac{1}{30}\m h^2\delta_{ij}\m,\label{eq:sst_wh_exp}\\[3pt]
\E\big[L_{s,t}^{ij}\,\big|\, W_{s,t}\m, H_{s,t}\m, K_{s,t}\big] =  \frac{3}{5}h H_{s,t}^i H_{s,t}^j & - \frac{1}{2}h\big(W_{s,t}^i K_{s,t}^j + K_{s,t}^i W_{s,t}^j\big)\label{eq:sst_whk_exp}\\[2pt]
& + \frac{60}{7}h K_{s,t}^i K_{s,t}^j + \frac{3}{140}\m h^2\delta_{ij},\nonumber
\end{align}
where $\delta_{ij}$ is the Kronecker delta. 
\end{theorem}
\begin{proof}
By the polynomial expansion of Brownian motion (see \cite{foster2020poly, habermann2021poly}), we have that
\begin{align}\label{eq:cubic_decomp}
W_t & = \widetilde{W}_t + \widetilde{Z}_t\m,
\end{align}
where $\widetilde{W} = \{\widetilde{W}_t\}_{t\in[0,1]}$ denotes the cubic polynomial approximation of $W$ given by
\begin{align}\label{eq:cubic}
\widetilde{W}_t & := tW_1 + 6t(1-t)H_{0,1} + 60t(1-t)(1-2t)K_{0,1}\m,
\end{align}
and $\widetilde{Z} = \{\widetilde{Z}_t\}_{t\in[0,1]}$ is a centred Gaussian process and independent of $W_1\m, H_{0,1}\m, K_{0,1}$.
With the above decomposition of Brownian motion, we can compute the expectation,
\begin{align*}
&\\[-22pt]
&\E\Bigg[\int_0^1 W_t^i\m W_t^j\m dt\,\Big|\, W_1\m, H_{0,1}\m, K_{0,1}\Bigg]\\[-5pt]
&\mm = \E\Bigg[\int_0^1 \big(\widetilde{W}_t^i + \widetilde{Z}_t^i\big)\big(\widetilde{W}_t^j + \widetilde{Z}_t^j\big) dt\,\Big|\, W_1\m, H_{0,1}\m, K_{0,1}\Bigg]\\
&\mm = \int_0^1\widetilde{W}_t^i\m\widetilde{W}_t^j dt +  \int_0^1 \E\big[\m \widetilde{Z}_t^i\m \widetilde{Z}_t^j\,\big|\, W_1\m, H_{0,1}\m, K_{0,1}\big] dt\\[-1pt]
&\mm\mm + \int_0^1 \E\big[\,\widetilde{W}_t^i\m \widetilde{Z}_t^j \,\big|\, W_1\m, H_{0,1}\m, K_{0,1}\big] dt + \int_0^1 \E\big[\,Z_t^i\m\widetilde{W}_t^j\,\big|\, W_1\m, H_{0,1}\m, K_{0,1}\big] dt\\
&\mm = \int_0^1\widetilde{W}_t^i\m\widetilde{W}_t^j dt + \int_0^1 \E\big[\m \widetilde{Z}_t^i\m \widetilde{Z}_t^j\big] dt + \int_0^1 \widetilde{W}_t^i\m\E\big[\m \widetilde{Z}_t^j\big] dt + \int_0^1 \widetilde{W}_t^j\m\E\big[\m \widetilde{Z}_t^i\big] dt\\
&\mm = \int_0^1\widetilde{W}_t^i\m\widetilde{W}_t^j dt + \int_0^1 \E\big[\m \widetilde{Z}_t^i\m \widetilde{Z}_t^j\big] dt.
\end{align*}
Since the first term is the integral of a polynomial, it can be explicitly computed as
\begin{align*}
&\int_0^1\widetilde{W}_t^i\m\widetilde{W}_t^j dt\\[-4pt]
&\mm = W_1\int_0^1 t(tW_1^j + 6t(1-t)H_{0,1}^j + 60t(1-t)(1-2t)K_{0,1}^j) dt\\
&\mm\mm + H_{0,1}\int_0^1 6t(1-t)(tW_1^j + 6t(1-t)H_{0,1}^j + 60t(1-t)(1-2t)K_{0,1}^j) dt\\
&\mm\mm + K_{0,1}\int_0^1 60t(1-t)(1-2t)(tW_1^j + 6t(1-t)H_{0,1}^j + 60t(1-t)(1-2t)K_{0,1}^j) dt\\
&\mm = \frac{1}{3}h W_{s,t}^{i}W_{s,t}^{j} + \frac{1}{2}h\big(W_{s,t}^i H_{s,t}^j + H_{s,t}^i W_{s,t}^j\big) + \frac{6}{5}h H_{s,t}^i H_{s,t}^j - h\big(W_{s,t}^i K_{s,t}^j + K_{s,t}^i W_{s,t}^j\big)\\[2pt]
&\hspace{27.6mm} + \frac{120}{7}h K_{s,t}^i K_{s,t}^j\m.
\end{align*}
Hence, $\int_0^1 \E\big[\m \widetilde{Z}_t^i\m \widetilde{Z}_t^j\big] dt = \E\big[\int_0^1 W_t^i\m W_t^j\m dt\big] - \E\big[\int_0^1\widetilde{W}_t^i\m\widetilde{W}_t^j dt\big] = \frac{1}{2}\delta_{ij} - \frac{16}{35}\delta_{ij} = \frac{3}{70}\delta_{ij}$.\vspace{1mm} The second result (\ref{eq:sst_whk_exp}) now follows by Corollary \ref{corr:simple_sst_levy} and the standard Brownian scaling. By taking the expectation of this conditional on $(W_{s,t}\m, H_{s,t})$, we then obtain (\ref{eq:sst_wh_exp}).
\end{proof}
\begin{remark}
By rearranging the terms in equation (\ref{eq:simple_sst_levy}), it is possible to show that
\begin{align*}
L_{s,t} & = \frac{1}{2}\int_s^t \bigg(W_{s,u} - \frac{u-s}{h}W_{s,t}\bigg)^{\otimes 2} du - \frac{1}{2}h\big(W_{s,t}\otimes K_{s,t} + K_{s,t}\otimes W_{s,t}\big).
\end{align*}
So by the independence of $(W_{s,t}\m, H_{s,t}\m, K_{s,t})$, taking the conditional expectation gives
\begin{align*}
\E\big[L_{s,t}\,\big|\, W_{s,t}\m, H_{s,t}\big] & = \frac{1}{2}\m\E\Bigg[\int_s^t\bigg(W_{s,u} - \frac{u-s}{h}W_{s,t}\bigg)^{\otimes 2} du \,\Big|\, W_{s,t}\m, H_{s,t}\Bigg].
\end{align*}
Therefore, in the typical setting where the SDE solver only requires $W_{s,t}$ or $(W_{s,t}\m, H_{s,t})$ we can interpret $L_{s,t}$ as measuring the discrepancy between the Brownian motion and its straight line approximation. However, this interpretation changes when $K_{s,t}$ is used.
\end{remark}\medbreak

Since the approximations given by the right-hand sides of equations (\ref{eq:sst_wh_exp}) and (\ref{eq:sst_whk_exp}) are conditional expectations of $L_{s,t}\m$, they will respectively produce the least $L^2(\mathbb{P})$ error among the $(W_{s,t}\m, H_{s,t})$-measurable and $(W_{s,t}\m, H_{s,t}\m, K_{s,t})$-measurable estimators.
Hence, for completeness, we will also compute explicit formulae for these $L^2(\mathbb{P})$ errors.\medbreak
\begin{theorem}[Mean-squared error of space-space-time L\'{e}vy area approximations]\label{thm:sst_cond_var}
\begin{align}
\var\big(\m L_{s,t}^{ij}\,\big|\, W_{s,t}\m, H_{s,t}\big) & = V_{s,t}^{ij} + \frac{1}{1440}h^3\Big(\big(W_{s,t}^i\big)^2 + \big(W_{s,t}^j\big)^2\Big)\label{eq:cond_var_wh}\\[3pt]
&\hspace{10mm} + \frac{1}{1400}h^3 \Big(\big(H_{s,t}^i\big)^2 + \big(H_{s,t}^j\big)^2\Big),\nonumber\\[3pt]
\var\big(\m L_{s,t}^{ij}\,\big|\, W_{s,t}\m, H_{s,t}\m, K_{s,t}\big) & =  \widetilde{V}_{s,t}^{ij} + \frac{1}{1400}h^3\Big(\big(H_{s,t}^i\big)^2 + \big(H_{s,t}^j\big)^2\Big)\label{eq:cond_var_whk}\\[3pt]
&\hspace{10mm} + \frac{1}{98}h^3 \Big(\big(K_{s,t}^i\big)^2 + \big(K_{s,t}^j\big)^2\Big),\nonumber
\end{align}
where
\begin{align*}
V_{s,t}^{ij} = \begin{cases} \frac{11}{25200}\m h^4, & \text{ if } i = j,\\ \frac{1}{900}\m h^4, & \text{ if } i\neq j,\end{cases} \hspace{10mm} \widetilde{V}_{s,t}^{ij} = \begin{cases} \frac{11}{88200}\m h^4, & \text{ if } i = j,\\ \frac{9}{19600}\m h^4, & \text{ if } i\neq j.\end{cases}
\end{align*}
\end{theorem}
\begin{proof}
As in previous proofs, we will use the stochastic integral given by equation (\ref{eq:simple_sst_levy}).
For simplicity, we will consider this integral on $[0,1]$ and use the below decomposition,
\begin{align*}
W_t & = \wideparen{W}_t + Z_t\m,
\end{align*}
where $\wideparen{W} = \{\wideparen{W}_t\}_{t\in[0,1]}$ is the quadratic polynomial approximation of $W$ given by
\begin{align*}
\wideparen{W}_t & := tW_1 + 6t(1-t)H_{0,1}\m,
\end{align*}
and $Z = \{Z_t\}_{t\in[0,1]}$ is a centred Gaussian process and independent of $W_1$ and $H_{0,1}$.
With this decomposition, it is now possible to expand the conditional variance of (\ref{eq:simple_sst_levy}),
\begin{align}\label{eq:variance_calculation}
&\var\bigg(\int_0^1 W_t^i\m W_t^j dt\, \Big|\, W_1\m, H_{0,1}\bigg)\\
& = \var\bigg(\int_0^1 \wideparen{W}_t^i\m \wideparen{W}_t^j dt + \int_0^1 \wideparen{W}_t^i\m Z_t^j dt + \int_0^1 Z_t^i\m \wideparen{W}_t^j dt + \int_0^1 Z_t^i\m Z_t^j\m dt\, \Big|\, W_1\m, H_{0,1}\bigg)\nonumber\\
& = \var\bigg(\int_0^1 \wideparen{W}_t^i\m \wideparen{W}_t^j dt\, \Big|\, W_1\m, H_{0,1}\bigg) + \var\bigg(\int_0^1 \wideparen{W}_t^i\m Z_t^j dt + \int_0^1 Z_t^i\m \wideparen{W}_t^j dt\, \Big|\, W_1\m, H_{0,1}\bigg)\nonumber\\
&\mm + \var\bigg(\int_0^1 Z_t^i\m Z_t^j\m dt\, \Big|\, W_1\m, H_{0,1}\bigg) + 2\m\E\Bigg[\int_0^1 \wideparen{W}_t^i\m Z_t^j dt\int_0^1 Z_t^i\m Z_t^j\m dt\, \Big|\, W_1\m, H_{0,1}\Bigg]\nonumber\\[-3pt]
&\mm  + 2\m\E\Bigg[\bigg(\int_0^1 \wideparen{W}_t^i\m Z_t^j dt + \int_0^1 Z_t^i\m \wideparen{W}_t^j dt\bigg)\bigg(\int_0^1 \wideparen{W}_t^i\m Z_t^j dt + \int_0^1 Z_t^i\m Z_t^j\m dt\bigg)\Big|\, W_1\m, H_{0,1}\Bigg]\nonumber\\
& = \var\bigg(\int_0^1 \wideparen{W}_t^i\m Z_t^j\m dt\, \Big|\, W_1\m, H_{0,1}\bigg) + \var\bigg(\int_0^1 \wideparen{W}_t^j\m Z_t^i\m dt\, \Big|\, W_1\m, H_{0,1}\bigg) + \var\bigg(\int_0^1 Z_t^i\m Z_t^j dt\bigg),\nonumber
\end{align}
where the last line follows as $\wideparen{W}$ is determined by $(W_1\m, H_{0,1})$ and $Z$ is an independent Gaussian process which is symmetric (that is, $Z$ and $-Z$ have the same distribution).
We will now assume $i\neq j$ and explicitly compute the different variances given above.
\begin{align*}
\var\bigg(\int_0^1 Z_t^i\m Z_t^j dt\bigg) & = \E\Bigg[\bigg(\int_0^1 Z_t^i\m Z_t^j dt\bigg)^2\m\Bigg] = \int_0^1\int_0^1 \E\Big[\big(Z_u^i\big)^2\big(Z_v^j\big)^2\Big]\, du\, dv.
\end{align*}
Therefore, since the covariance function of each Brownian arch process is given by
\begin{align*}
\E\big[Z_s^i Z_t^i\big] & = \min(s,t) - st - 3st(1-s)(1-t),
\end{align*}
see \cite[Definition 3.4]{foster2020poly}, it directly follows from the independence of $Z^i$ and $Z^j$ that
\begin{align*}
\var\bigg(\int_0^1 Z_t^i\m Z_t^j dt\bigg) & = \int_0^1\int_0^1 \E\Big[\big(Z_u^i\big)^2\Big]\E\Big[\big(Z_v^j\big)^2\Big]\, du\, dv = \Bigg(\int_0^1 \E\Big[\big(Z_u^i\big)^2\Big]\, du\Bigg)^2\\[-3pt]
& = \Bigg(\int_0^1 u - u^2 - 3u^2(1-u)^2\,du\Bigg)^2 = \bigg(\frac{1}{15}\bigg)^2 = \frac{1}{225}\m.\\[-20pt]
\end{align*}
Similarly, we can compute the conditional variances for the ``cross-integrals'' as
\begin{align*}
& \var\bigg(\int_0^1 \wideparen{W}_t^i\m Z_t^j\m dt\, \Big|\, W_1\m, H_{0,1}\bigg) =  \E\Bigg[\bigg(\int_0^1 \wideparen{W}_t^i\m Z_t^j\m dt\bigg)^2\, \Big|\, W_1\m, H_{0,1}\Bigg]\\[-3pt]
&\mm = \int_0^1\int_0^1 \wideparen{W}_u^i\m \wideparen{W}_v^i\,\E\big[Z_u^j\m Z_v^j\big]\m du\m dv\\
&\mm = \int_0^1\int_0^1 \wideparen{W}_u^i\m \wideparen{W}_v^i\,\big(\min(u,v) - uv - 3uv(1-u)(1-v)\big)\m du\m dv\\
&\mm = \int_0^1\int_0^1 \wideparen{W}_u^i\m \wideparen{W}_v^i\m\min(u,v) \m du\m dv - \bigg(\int_0^1 u\wideparen{W}_u^i\m du\bigg)^2 - 3\bigg(\int_0^1 u(1-u)\wideparen{W}_u^i\m du\bigg)^2\\
&\mm = \frac{2}{15}\big(W_1^i\big)^2  + \frac{13}{35}\big(H_{0,1}^i\big)^2 + \frac{13}{30}\m W_1^i H_{0,1}^i - \Big(\frac{1}{3}W_1^i + \frac{1}{2}H_{0,1}^i\Big)^2 - 3\Big(\frac{1}{12}W_1^i + \frac{1}{5}H_{0,1}^i\Big)^2\\[3pt]
&\mm = \frac{1}{720}\big(W_1^i\big)^2 + \frac{1}{700}\big(H_{0,1}^i\big)^2.
\end{align*}
When $i=j$, the conditional variance given by (\ref{eq:cond_var_wh}) was shown in \cite[Theorem 3.10]{foster2020poly}.
By the same arguments, we can compute the variance (\ref{eq:cond_var_whk}) with the cubic polynomial decomposition of Brownian motion, $W\hspace{-0.45mm} = \widetilde{W} + \widetilde{Z}$, previously used to prove Theorem \ref{thm:sst_approx_whk}.
\begin{align*}
&\var\bigg(\int_0^1 W_t^i\m W_t^j dt\, \Big|\, W_1\m, H_{0,1}\m, K_{0,1}\bigg)\\[-2pt]
&\mm = \var\bigg(\int_0^1 \widetilde{W}_t^i\m \widetilde{Z}_t^j\m dt\, \Big|\, W_1\m, H_{0,1}\m, K_{0,1}\bigg) + \var\bigg(\int_0^1 \widetilde{W}_t^j\m \widetilde{Z}_t^i\m dt\, \Big|\, W_1\m, H_{0,1}\m, K_{0,1}\bigg)\\[-2pt]
&\mm\mm + \var\bigg(\int_0^1 \widetilde{Z}_t^i\m\widetilde{Z}_t^j dt\bigg),\\[-20pt]
\end{align*}
where we used the independence and symmetry of $\widetilde{Z}$ (similar to the expansion (\ref{eq:variance_calculation})).
Similar to before, we note that the covariance function for the Gaussian process $\widetilde{Z}$ is
\begin{align*}
\E\big[\widetilde{Z}_s^i \widetilde{Z}_t^i\big] & = \min(s,t) - st - 3st(1-s)(1-t) - 5st(1-s)(1-t)(1-2s)(1-2t),
\end{align*}
which follows from the covariance function of $W$ and the equations (\ref{eq:cubic_decomp}) and (\ref{eq:cubic}).

Using the above covariance function for $\widetilde{Z}$, we can compute the remaining variances. 
\begin{align*}
\var\bigg(\int_0^1 \widetilde{Z}_t^i\m \widetilde{Z}_t^j dt\bigg) & = \int_0^1\int_0^1 \E\Big[\big(\widetilde{Z}_u^i\big)^2\Big]\E\Big[\big(\widetilde{Z}_v^j\big)^2\Big]\, du\, dv = \Bigg(\int_0^1 \E\Big[\big(\widetilde{Z}_u^i\big)^2\Big]\, du\Bigg)^2 \\
& = \Bigg(\int_0^1 \big(u - u^2 - 3u^2(1-u)^2 - 5u^2(1-u)^2(1-2u)^2\big)\, du\Bigg)^2\\
& = \bigg(\frac{3}{70}\bigg)^2 = \frac{9}{4900}\m,\\[3pt]
\var\Bigg(\int_0^1\big(\widetilde{Z}_t^i\big)^2\, dt\Bigg) & = \int_0^1\int_0^1 \E\Big[\big(\widetilde{Z}_u^i\big)^2\big(\widetilde{Z}_v^i\big)^2\Big]\, du\, dv - \Bigg(\E\Bigg[\int_0^1\big(\widetilde{Z}_t^i\big)^2\, dt\Bigg]\Bigg)^2\\
&\hspace{-30mm} = \int_0^1\int_0^1 \hspace{-0.5mm}\Big(\E\Big[\big(\widetilde{Z}_u^i\big)^2\Big]\E\Big[\big(\widetilde{Z}_v^i\big)^2\Big]\hspace{-0.25mm} + 2\m\E\big[\widetilde{Z}_u^i\widetilde{Z}_v^i\big]^2\Big)\m du\, dv - \int_0^1\int_0^1\E\Big[\big(\widetilde{Z}_u^i\big)^2\Big]\E\Big[\big(\widetilde{Z}_v^i\big)^2\Big]\m du\, dv\\[3pt]
&\hspace{-30mm} = 2\int_0^1\int_0^1 \E\big[\widetilde{Z}_u^i\widetilde{Z}_v^i\big]^2\, du\, dv = \frac{11}{22050}\m,\\[-30pt]
\end{align*}
\begin{align*}
& \var\bigg(\int_0^1 \widetilde{W}_t^i\m \widetilde{Z}_t^j\m dt\, \Big|\, W_1\m, H_{0,1}\m, K_{0,1}\bigg) =  \E\Bigg[\bigg(\int_0^1 \widetilde{W}_t^i\m \widetilde{Z}_t^j\m dt\bigg)^2\, \Big|\, W_1\m, H_{0,1}\m, K_{0,1}\Bigg]\\[-2pt]
&\mm = \int_0^1\int_0^1 \widetilde{W}_u^i\m \widetilde{W}_v^i\,\E\big[\widetilde{Z}_u^j\m \widetilde{Z}_v^j\big]\m du\m dv\\
&\mm = \int_0^1\int_0^1 \widetilde{W}_u^i\m \widetilde{W}_v^i\m\min(u,v) \m du\m dv\\
&\mm\mm - \bigg(\int_0^1 u\widetilde{W}_u^i\m du\bigg)^2\hspace{-0.5mm} - 3\bigg(\int_0^1 u(1-u)\widetilde{W}_u^i\m du\bigg)^2\hspace{-0.5mm} - 5\bigg(\int_0^1 u(1-u)(1-2u)\widetilde{W}_u^i\m du\bigg)^2\\
&\mm = \frac{2}{15}\big(W_1^i\big)^2 + \frac{13}{35}\big(H_{0,1}^i\big)^2 + \frac{10}{7}\big(K_{0,1}^i\big)^2 + \frac{13}{30}\m W_1^i H_{0,1}^i - H_{0,1}^i K_{0,1}^i - \frac{5}{7}W_1^i K_{0,1}^i\\
&\mm\mm - \Big(\frac{1}{3}W_1^i + \frac{1}{2}H_{0,1}^i - K_{0,1}^i\Big)^2 - 3\Big(\frac{1}{12}W_1^i + \frac{1}{5}H_{0,1}^i\Big)^2 - 5\Big(\frac{2}{7}K_{0,1}^i - \frac{1}{60} W_{0,1}^i\Big)^2\\[3pt]
&\mm = \frac{1}{700}\big(H_{0,1}^i\big)^2 + \frac{1}{49}\big(K_{0,1}^i\big)^2.
\end{align*}
The variances for $L_{s,t}^{ij}$ now follow by equation (\ref{eq:simple_sst_levy}) and the usual Brownian scaling.
\end{proof}\vspace{1mm}

\begin{remark}
Since the conditional variances (\ref{eq:cond_var_wh}) and (\ref{eq:cond_var_whk}) are simply the $L^2(\mathbb{P})$ errors of the approximations given in Theorem \ref{thm:sst_approx_whk}, they could also be used to estimate the $L^2(\mathbb{P})$ error produced by one step of a high order solver for a commutative noise SDE. These local error estimates would then be useful when adaptively reducing step sizes.
We refer the reader to \cite[Chapter 6]{foster2020thesis} for an example of such an adaptive methodology which exhibits high order convergence, even for a non-Lipschitz SDE (the CIR model).
However, the main downside of this approach is that it requires explicitly computing vector field derivatives -- which can be particularly difficult for high dimensional SDEs.
\end{remark}

\subsection{Approximation of $L_{s,t}$ for high order splitting methods}\label{sect:sst_splitting}

In this section, we will compute the mean and variance of $L_{s,t}$ conditional on $W_{s,t}\m, H_{s,t}$ and a Rademacher random vector -- which we expect is faster to generate than $K_{s,t}\m$.
This approximation was proposed in \cite{foster2024splitting} for deriving high order splitting methods.
Similar to the previous section, our main theorem for approximating space-space-time L\'{e}vy area (Theorem \ref{thm:sst_with_swing}) is new but already known when $d=1$ (see \cite[Theorem B.4]{foster2024splittingarxiv}).
We first define the Rademacher random vector $n_{s,t}$ that will be used instead of $K_{s,t}\m$.
\medbreak

\begin{definition}\label{def:st_swing} The space-time L\'{e}vy swing\footnote{The acronym ``\textbf{swing}'' stands for ``\textbf{s}ide \textbf{w}ith \textbf{in}tegral \textbf{g}reater''.} of Brownian motion over $[s,t]$ is given by
\begin{align}\label{eq:st_swing}
n_{s,t} := \sgn\big(H_{s,u} - H_{u,t}\big),
\end{align}
where $u:= \frac{1}{2}(s+t)$ is the interval's midpoint.
\end{definition}\vspace*{-2mm}
\begin{figure}[ht] \label{fig:levy_swing}
\centering
\includegraphics[width=0.85\textwidth]{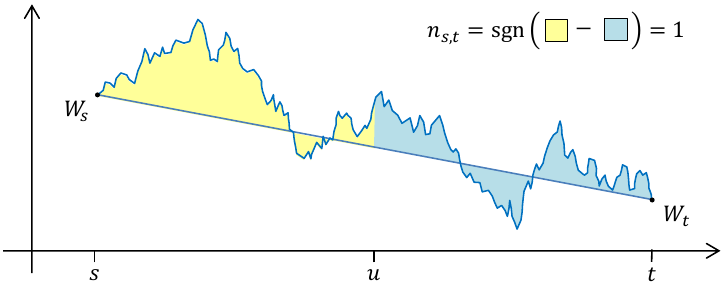}
\caption{Space-time L\'{e}vy swing gives the side where the path has greater space-time L\'{e}vy area (diagram taken from \cite{foster2020thesis}).}
\end{figure}
Just like $K_{s,t}\m$, it is straightforward to show that $n_{s,t}$ is independent of $\big(W_{s,t}, H_{s,t}\big)$.
However, we will also establish the independence of the Brownian arch at $u = \frac{1}{2}(s+t)$, since it will become helpful for proving the main result of this section (Theorem \ref{thm:sst_with_swing}).\medbreak
\begin{theorem}\label{thm:whzn_independence}
Let $u = \frac{1}{2}(s+t)$ be the midpoint of $[s,t]$. We define the random vectors,
\begin{align}
Z_{s,u} & := \frac{1}{8}\big(W_{s,u} - W_{u,t}\big) - \frac{3}{4}\big(H_{s,u} + H_{u,t}\big)\m,\label{eq:arch_midpoint}\\[3pt]
N_{s,t} & := H_{s,u} - H_{u,t}\m.\label{eq:gaussian_swing}
\end{align}
Then $W_{s,t}\m, H_{s,t}\m, Z_{s,u}\m, N_{s,t}$ are independent and
\begin{align*}
Z_{s,u}\sim\mathcal{N}\bigg(0, \frac{1}{16}h I_d\bigg),\hspace{7.5mm}N_{s,t}\sim\mathcal{N}\bigg(0, \frac{1}{12}h I_d\bigg).
\end{align*}
Moreover, we can write $(W_{s,u}\m, H_{s,u}\m, W_{u,t}\m, H_{u,t})$ in terms of $(W_{s,t}\m, H_{s,t}\m, Z_{s,u}\m, N_{s,t})$ as
\begin{align}
W_{s,u} & = \frac{1}{2}W_{s,t} + \frac{3}{2}H_{s,t} + Z_{s,u}\m,\mm\,\,\, W_{u,t} = \frac{1}{2}W_{s,t} - \frac{3}{2}H_{s,t} - Z_{s,u}\m,\label{eq:wsut_formulae}\\[3pt]
H_{s,u} & = \frac{1}{4}H_{s,t} - \frac{1}{2}Z_{s,u} + \frac{1}{2}N_{s,t}\m,\mm H_{u,t} = \frac{1}{4}H_{s,t} - \frac{1}{2}Z_{s,u} - \frac{1}{2}N_{s,t}\m.\label{eq:hsut_formulae}
\end{align}
\end{theorem}
\begin{proof}
Since $W_{s,t}\m, H_{s,t}\m, Z_{s,u}$ and $N_{s,t}$ can be expressed as linear functions of the same underlying Brownian motion $W$, it immediately follows that they are jointly normal. Thus, to show that they are independent, it suffices to show that they are uncorrelated.
\begin{align*}
\E\big[W_{s,t}\otimes Z_{s,u}\big] & = \E\big[W_{s,u}\otimes Z_{s,u}\big] +  \E\big[W_{u,t}\otimes Z_{s,u}\big] = \frac{1}{16}hI_d - \frac{1}{16}hI_d = 0,\\
\E\big[W_{s,t}\otimes N_{s,t}\big] & = \E\big[W_{s,u}\otimes (H_{s,u} - H_{u,t})\big] +  \E\big[W_{u,t}\otimes (H_{s,u} - H_{u,t})\big] = 0,\\
\E\big[Z_{s,u}\otimes N_{s,t}\big] & = \frac{3}{4}\m\E\big[(H_{s,u} + H_{u,t})\otimes (H_{s,u} - H_{u,t})\big] = 0.
\end{align*}
To show that $H_{s,t}$ is also uncorrelated with $Z_{s,u}$ and $N_{s,t}\m$, we note that
\begin{align*}
\int_s^t W_{s,r}\,dr & = \int_s^u W_{s,r}\,dr + \int_u^t W_{s,r}\,dr = \int_s^u W_{s,r}\,dr + \int_u^t W_{u,r}\,dr + \frac{1}{2}h W_{s,u}\m,
\end{align*}
and, by Theorem \ref{thm:whk_relation}, for $0\leq a\leq b$,
\begin{align*}
\int_a^b W_{a,r}\,dr & = \frac{1}{2}(b-a)W_{a,b} + (b-a)H_{a,b}\m.
\end{align*}
Therefore, from the above equations, it immediately follows that
\begin{align}
H_{s,t} & = \frac{1}{4}\big(W_{s,u} - W_{u,t}\big)+ \frac{1}{2}\big(H_{s,u} + H_{u,t}\big)\m,\label{eq:chen_for_h}
\end{align}
which is uncorrelated with $Z_{s,u}$ and $N_{s,t}\m$ as
\begin{align*}
\E\big[H_{s,t}\otimes Z_{s,u}\big] & = \frac{1}{32}\m\E\Big[\big(W_{s,u} - W_{u,t}\big)^{\otimes 2}\Big] - \frac{3}{8}\m\E\Big[\big(H_{s,u} + H_{u,t}\big)^{\otimes 2}\Big]\\[2pt]
& = \frac{1}{32}\Big(\frac{1}{2}h I_d + \frac{1}{2}h I_d\Big) - \frac{3}{8}\Big(\frac{1}{24}h I_d + \frac{1}{24}h I_d\Big) = 0,\\[3pt]
\E\big[H_{s,t}\otimes N_{s,t}\big] & = \frac{1}{2}\E\Big[\big(H_{s,u} + H_{u,t}\big)\otimes\big(H_{s,u} - H_{u,t}\big)\Big] = 0.
\end{align*}
Similarly, the variances of $Z_{s,u}$ and $N_{s,t}$ are both straightforward to calculate. Finally, equations (\ref{eq:wsut_formulae}) and (\ref{eq:hsut_formulae}) now follow by rearranging equations (\ref{eq:arch_midpoint}), (\ref{eq:gaussian_swing}) and (\ref{eq:chen_for_h}).
\end{proof}
Due to the independence of $N_{s,t} := H_{s,u} - H_{u,t}$, we obtain the following corollary:\medbreak
\begin{corollary}
The random vector $n_{s,t} =  \sgn\big(N_{s,t}\big)$ is independent of $(W_{s,t}\m, H_{s,t}\m, Z_{s,u})$.
\end{corollary}
\medbreak
We will now present the main result of the section, which is a direct extension of Theorems \ref{thm:sst_approx_whk} and \ref{thm:sst_cond_var}, to the case where we use $(W_{s,t}\m, H_{s,t}\m, n_{s,t})$ to approximate $L_{s,t}\m$.
\medbreak
\begin{theorem}[Approximation of space-space-time L\'{e}vy area using $(W_{s,t}\m, H_{s,t}\m, n_{s,t})$]\label{thm:sst_with_swing}
\begin{align}
\nonumber\\[-32pt]
&\E\big[L_{s,t}^{ij} \,|\, W_{s,t}\m, H_{s,t}\m, n_{s,t}\big]\label{eq:sst_exp_whn}\\
&\hspace{2.5mm} = \begin{cases}  \frac{3}{5}h\big(H_{s,t}^i\big)^2 - \frac{1}{8\sqrt{6\pi}}\m W_{s,t}^i n_{s,t}^i h^\frac{3}{2} + \frac{1}{30} h^2, &\hspace{11.75mm} \text{if }\,i = j,\\[5pt]
\frac{3}{5}h H_{s,t}^i H_{s,t}^j - \frac{1}{16\sqrt{6\pi}}\big(W_{s,t}^i n_{s,t}^j + n_{s,t}^i W_{s,t}^j\big)h^\frac{3}{2} + \frac{1}{40\pi}h^2\m n_{s,t}^i n_{s,t}^j\m, &\hspace{11.75mm} \text{if }\,i\neq j,\end{cases}\nonumber\\[4pt]
&\var\big(L_{s,t}^{ij} \,|\, W_{s,t}\m, H_{s,t}\m, n_{s,t}\big)\label{eq:sst_var_whn}\\[3pt]
&\hspace{2.5mm} = \begin{cases} \frac{11}{25200} h^4  + \big(\frac{1}{720} - \frac{1}{384\pi}\big)h^2\big(W_{s,t}^i\big)^2 + \frac{1}{700}h^3 \big(H_{s,t}^i\big)^2 & \hspace*{-5mm} \text{if }\,i = j,\\[3pt]
\,\,\,\, - \frac{1}{320\sqrt{6\pi}} W_{s,t}^i n_{s,t}^i h^\frac{7}{2}, & \\[5pt]
\big(\frac{1}{2880} - \frac{1}{1600\pi^2}\big) h^4 + \big(\frac{17}{46080} - \frac{1}{1536\pi}\big)h^3\big(\big(W_{s,t}^i\big)^2 + \big(W_{s,t}^j\big)^2\big) & \hspace*{-5mm}\text{if }\,i\neq j,\\[3pt]
\,\,\,\, + \frac{1}{1792}h^3\big(\big(H_{s,t}^i\big)^2 + \big(H_{s,t}^j\big)^2\m\big) - \frac{1}{640\sqrt{6\pi}}\big(1 - \frac{2}{\pi}\big)\big(W_{s,t}^i n_{s,t}^i + W_{s,t}^j n_{s,t}^j\big)h^\frac{7}{2}. &\end{cases}\nonumber
\end{align}
\end{theorem}
\begin{proof}
By Corollary \ref{corr:simple_sst_levy} we have that, for $0\leq a\leq b$,
\begin{align}
\hspace{-0.5mm}\int_a^b W_{a,r}^i W_{a,r}^j\m dr = \frac{1}{3}(b-a) W_{a,b}^{i}W_{a,b}^{j} + \frac{1}{2}(b-a)\big(W_{a,b}^i H_{a,b}^j + H_{a,b}^i W_{a,b}^j\big) + 2L_{a,b}^{ij}\m. \label{eq:time_integral_relation}
\end{align}
Letting $u := \frac{1}{2}(s+t)$ denote the midpoint of $s$ and $t$, we can express the integral (\ref{eq:time_integral_relation}) on the interval $[s,t]$ in terms of the same integral over the half-intervals $[s,u]$ and $[u,t]$.
\begin{align}
\int_s^t W_{s,r}^i W_{s,r}^j\m dr & = \int_s^u W_{s,r}^i W_{s,r}^j\m dr  + \int_u^t \big(W_{s,u}^i + W_{u,r}^i\big)\big(W_{s,u}^j + W_{u,r}^j\big)\m dr\nonumber\\
& =  \int_s^u W_{s,r}^i W_{s,r}^j\m dr + \int_u^t W_{u,r}^i W_{u,r}^j\m dr\label{eq:time_integral_relation2}\\
&\hspace{10mm} + \frac{1}{2}h W_{s,u}^i W_{s,u}^j + W_{s,u}^i\int_u^t W_{u,r}^j\m dr + W_{s,u}^j\int_u^t W_{u,r}^i\m dr.\nonumber
\end{align}
Substituting equations (\ref{eq:chen_for_h}) and (\ref{eq:time_integral_relation}) into the above and simplifying the terms produces
\begin{align}
L_{s,t}^{ij} & = L_{s,u}^{ij} + L_{u,t}^{ij} + \frac{1}{24}h \big(W_{s,u}^i - W_{u,t}^i\big)\big(W_{s,u}^j - W_{u,t}^j\big)\label{eq:chen_for_L} \\[2pt]
&\hspace{21.5mm} + \frac{1}{8}h\big(W_{s,u}^i H_{u,t}^j - H_{s,u}^i W_{u,t}^j\big) + \frac{1}{8}h\big(H_{u,t}^i W_{s,u}^j - W_{u,t}^i H_{s,u}^j\big).\nonumber
\end{align}
Reformulating equation (\ref{eq:chen_for_L}) in terms of $W_{s,t}\m, H_{s,t}\m, Z_{s,u}\m, N_{s,t}\m$ (see Theorem \ref{thm:whzn_independence}) gives
\begin{align}
L_{s,t}^{ij} = L_{s,u}^{ij} + L_{u,t}^{ij} & + \frac{9}{16}h H_{s,t}^i H_{s,t}^j - \frac{1}{12}h Z_{s,u}^i Z_{s,u}^j\label{eq:better_chen_for_L}\\[2pt]
& + \frac{1}{8}h\big(H_{s,t}^i Z_{s,u}^j + Z_{s,u}^i H_{s,t}^j\big) - \frac{1}{16} h\big(W_{s,t}^i N_{s,t}^j + N_{s,t}^i W_{s,t}^j\big).\nonumber
\end{align}
Using Theorems \ref{thm:sst_approx_whk} and \ref{thm:whzn_independence}, we will compute the following conditional expectation,
\begin{align}\label{eq:first_cond_exp}
&\E\big[L_{s,t}^{ij}\,\big|\, W_{s,u}\m, W_{u,t}\m, H_{s,u}\m, H_{u,t}\big]\nonumber\\[-2pt]
&\mm = \E\big[L_{s,u}^{ij}\,\big|\, W_{s,u}\m, H_{s,u}\big] + \E\big[L_{u,t}^{ij}\,\big|\, W_{u,t}\m, H_{u,t}\big] + \frac{9}{16}h H_{s,t}^i H_{s,t}^j\nonumber\\
& \mmm - \frac{1}{12}h Z_{s,u}^i Z_{s,u}^j + \frac{1}{8}h\big(H_{s,t}^i Z_{s,u}^j + Z_{s,u}^i H_{s,t}^j\big) - \frac{1}{16} h\big(W_{s,t}^i N_{s,t}^j + N_{s,t}^i W_{s,t}^j\big)\nonumber\\
&\mm = \frac{3}{10}h H_{s,u}^i H_{s,u}^j + \frac{1}{120}\m h^2\delta_{ij} + \frac{3}{10}h H_{u,t}^i H_{u,t}^j + \frac{1}{120}\m h^2\delta_{ij} + \frac{9}{16}h H_{s,t}^i H_{s,t}^j\nonumber\\
&\mmm - \frac{1}{12}h Z_{s,u}^i Z_{s,u}^j + \frac{1}{8}h\big(H_{s,t}^i Z_{s,u}^j + Z_{s,u}^i H_{s,t}^j\big) - \frac{1}{16} h\big(W_{s,t}^i N_{s,t}^j + N_{s,t}^i W_{s,t}^j\big)\nonumber\\
&\mm = \frac{3}{10}h \bigg(\frac{1}{4}H_{s,t}^i - \frac{1}{2}Z_{s,u}^i + \frac{1}{2}N_{s,t}^i\bigg)\bigg(\frac{1}{4}H_{s,t}^j - \frac{1}{2}Z_{s,u}^j + \frac{1}{2}N_{s,t}^j\bigg) + \frac{9}{16}h H_{s,t}^i H_{s,t}^j\nonumber\\
&\mmm + \frac{3}{10}h \bigg(\frac{1}{4}H_{s,t}^i - \frac{1}{2}Z_{s,u}^i - \frac{1}{2}N_{s,t}^i\bigg)\bigg(\frac{1}{4}H_{s,t}^j - \frac{1}{2}Z_{s,u}^j - \frac{1}{2}N_{s,t}^j\bigg) + \frac{1}{60}\m h^2\delta_{ij}\nonumber\\
&\mmm - \frac{1}{12}h Z_{s,u}^i Z_{s,u}^j + \frac{1}{8}h\big(H_{s,t}^i Z_{s,u}^j + Z_{s,u}^i H_{s,t}^j\big) - \frac{1}{16} h\big(W_{s,t}^i N_{s,t}^j + N_{s,t}^i W_{s,t}^j\big)\nonumber\\
&\mm = \frac{3}{5}h H_{s,t}^i H_{s,t}^j + \frac{1}{15}h Z_{s,u}^i Z_{s,u}^j + \frac{3}{20}h N_{s,t}^i N_{s,t}^j + \frac{1}{60}\m h^2\delta_{ij}\\[-1pt]
&\mmm + \frac{1}{20}h\big(H_{s,t}^i Z_{s,u}^j + Z_{s,u}^i H_{s,t}^j\big) - \frac{1}{16}h\big(W_{s,t}^i N_{s,t}^j + N_{s,t}^i W_{s,t}^j\big).\nonumber
\end{align}
Since $n_{s,t}^i := \sgn(N_{s,t}^i)$ and $N_{s,t}^i\sim\mathcal{N}\big(0, \frac{1}{12}h\big)$, it follows that $|N_{s,t}^i|$ has a half-normal distribution and is independent of $n_{s,t}^i\m$. Moreover, this implies that its moments are
\begin{align}
\nonumber\\[-20pt]
\E\big[N_{s,t}^i \,\big|\, n_{s,t}\big] & = \frac{1}{\sqrt{6\pi}}n_{s,t}^i h^\frac{1}{2}, &\hspace*{-2.5mm} \E\Big[\big(N_{s,t}^i\big)^3 \,\big|\, n_{s,t}\Big] & = \frac{1}{6\sqrt{6\pi}}n_{s,t}^i h^\frac{3}{2},\label{eq:half_normal_odd_moments}\\[3pt]
\E\Big[\big(N_{s,t}^i\big)^2 \,\big|\, n_{s,t}\Big] & = \frac{1}{12}h\m, &\hspace*{-2.5mm} \E\Big[\big(N_{s,t}^i\big)^4 \,\big|\, n_{s,t}\Big] & = \frac{1}{48}h^2.\label{eq:half_normal_even_moments}
\end{align}
Explicit formulae for the first four central moments of the half-normal distribution are given in \cite[Equation (16)]{elandt1961folded_normal}. By the independence of $N_{s,t}$ and $(W_{s,t}\m, H_{s,t}\m, Z_{s,u})$, taking the conditional expectation of equation (\ref{eq:better_chen_for_L}), and applying the tower law, gives
\begin{align*}
&\\[-20pt]
&\E\big[\m L_{s,t}^{ij}\,\big|\, W_{s,t}\m, H_{s,t}\m, n_{s,t}\big] = \E\Big[\E\big[\m L_{s,t}^{ij}\,\big|\, W_{s,t}\m, H_{s,t}\m, Z_{s,u}\m, N_{s,t}\big]\,\big|\, W_{s,t}\m, H_{s,t}\m, n_{s,t}\Big]\\
&\mm = \E\Big[\m\frac{3}{5}h H_{s,t}^i H_{s,t}^j + \frac{1}{15}h Z_{s,u}^i Z_{s,u}^j + \frac{3}{20}h N_{s,t}^i N_{s,t}^j + \frac{1}{60}\m h^2\delta_{ij}\\
&\hspace{15mm} + \frac{1}{20}h\big(H_{s,t}^i Z_{s,u}^j + Z_{s,u}^i H_{s,t}^j\big) - \frac{1}{16}h\big(W_{s,t}^i N_{s,t}^j + N_{s,t}^i W_{s,t}^j\big)\,\big|\, W_{s,t}\m, H_{s,t}\m, n_{s,t}\Big]\\
&\mm = \frac{3}{5}h H_{s,t}^i H_{s,t}^j + \frac{1}{15}h\m\E\big[Z_{s,u}^i Z_{s,u}^j\big] + \frac{3}{20}h\m\E\big[N_{s,t}^i N_{s,t}^j\,\big|\, n_{s,t}\big] + \frac{1}{60}h^2\delta_{ij}\\
&\mmm - \frac{1}{16\sqrt{6\pi}}\big(W_{s,t}^i n_{s,t}^j + n_{s,t}^i W_{s,t}^j\big)h^\frac{3}{2}\\
&\mm = \frac{3}{5}h H_{s,t}^i H_{s,t}^j + \frac{1}{30}h^2 \delta_{ij} - \frac{1}{16\sqrt{6\pi}}\big(W_{s,t}^i n_{s,t}^j + n_{s,t}^i W_{s,t}^j\big)h^\frac{3}{2} +  \frac{1-\delta_{ij}}{40\pi}h\m n_{s,t}^i n_{s,t}^j\m.
\end{align*}
Using the decomposition (\ref{eq:time_integral_relation2}) along with the independence of $W_{s,u}\m, H_{s,u}\m, W_{u,t}\m, H_{u,t}$ and the independence of the Brownian arch processes over $[s,u]$ and $[u,t]$, we see that
\begin{align*}
&\var\bigg(\int_s^t W_{s,r}^i W_{s,r}^j\m dr \,\Big|\, W_{s,u}\m, W_{u,t}\m, H_{s,u}\m, H_{u,t}\bigg)\\
&\mm = \var\bigg(\int_s^u W_{s,r}^i W_{s,r}^j\m dr + \frac{1}{2}h W_{s,u}^i W_{s,u}^j + W_{s,u}^i\int_u^t W_{u,r}^j\m dr + W_{s,u}^j\int_u^t W_{u,r}^i\m dr\\
&\hspace{20mm} + \int_u^t W_{u,r}^i W_{u,r}^j\m dr\,\Big|\, W_{s,u}\m, W_{u,t}\m, H_{s,u}\m, H_{u,t}\bigg)\\
&\mm = \var\bigg(\int_s^u W_{s,r}^i W_{s,r}^j\m dr \,\Big|\, W_{s,u}\m, H_{s,u}\bigg) + \var\bigg(\int_u^t W_{u,r}^i W_{u,r}^j\m dr \,\Big|\, W_{u,t}\m, H_{u,t}\bigg).
\end{align*}
Therefore, by Corollary \ref{corr:simple_sst_levy} and Theorem \ref{thm:sst_cond_var}, the conditional variance of $L_{s,t}$ is given by
\begin{align}\label{eq:working_var}
&\var\big(L_{s,t}^{ij}\,\big|\, W_{s,u}\m, W_{u,t}\m, H_{s,u}\m, H_{u,t}\big)\\[3pt]
&\mm = \var\big(L_{s,u}^{ij}\,\big|\, W_{s,u}\m, H_{s,u}\big) + \var\big(L_{u,t}^{ij}\,\big|\, W_{u,t}\m, H_{u,t}\big)\nonumber\\
&\mm = V_{s,u}^{ij} + V_{u,t}^{ij} + \frac{1}{11520}h^3\Big(\big(W_{s,u}^i\big)^2 + \big(W_{u,t}^i\big)^2 + \big(W_{s,u}^j\big)^2 + \big(W_{u,t}^j\big)^2\Big)\nonumber\\
&\hspace{25.75mm} + \frac{1}{11200}h^3 \Big(\big(H_{s,u}^i\big)^2 + \big(H_{u,t}^i\big)^2 + \big(H_{s,u}^j\big)^2 + \big(H_{u,t}^j\big)^2\Big),\nonumber
\end{align}
where $\, V_{s,u}^{ij} + V_{u,t}^{ij} = \begin{cases} \frac{11}{201600}\m h^4, & \text{ if } i = j\\[3pt]
\frac{1}{7200}\m h^4, & \text{ if } i\neq j\end{cases}$. Taking an expectation of (\ref{eq:working_var}) produces
\begin{align*}
&\E\big[\var\big(L_{s,t}^{ij}\,\big|\, W_{s,u}\m, W_{u,t}\m, H_{s,u}\m, H_{u,t}\big)\,\big|\,W_{s,t}\m, H_{s,t}\m, n_{s,t}\big] - \big(V_{s,u}^{ij} + V_{u,t}^{ij}\big)\\
& = \E\Bigg[\frac{1}{11520}h^3\bigg(\bigg(\frac{1}{2}W_{s,t}^i + \frac{3}{2}H_{s,t}^i + Z_{s,u}^i\bigg)^2 + \bigg(\frac{1}{2}W_{s,t}^i - \frac{3}{2}H_{s,t}^i - Z_{s,u}^i\bigg)^2\\[-3pt]
&\mmm + \bigg(\frac{1}{2}W_{s,t}^j + \frac{3}{2}H_{s,t}^j + Z_{s,u}^j\bigg)^2 + \bigg(\frac{1}{2}W_{s,t}^j - \frac{3}{2}H_{s,t}^j - Z_{s,u}^j\bigg)^2\,\bigg)\\[-3pt]
&\mmm + \frac{1}{11200}h^3 \bigg(\bigg(\frac{1}{4}H_{s,t}^i - \frac{1}{2}Z_{s,u}^i + \frac{1}{2}N_{s,t}^i\bigg)^2 + \bigg(\frac{1}{4}H_{s,t}^i - \frac{1}{2}Z_{s,u}^i - \frac{1}{2}N_{s,t}^i\bigg)^2\\[-3pt]
&\mmm + \bigg(\frac{1}{4}H_{s,t}^j - \frac{1}{2}Z_{s,u}^j + \frac{1}{2}N_{s,t}^j\bigg)^2 + \bigg(\frac{1}{4}H_{s,t}^j - \frac{1}{2}Z_{s,u}^j - \frac{1}{2}N_{s,t}^j\bigg)^2\,\bigg)\m\Big|\, W_{s,t}\m, H_{s,t}\m, n_{s,t}\Bigg]\\
& = \frac{1}{11520}h^3\bigg(\frac{1}{2}\big(W_{s,t}^i\big)^2 + \frac{9}{2}\big(H_{s,t}^i\big)^2 + \frac{1}{2}\big(W_{s,t}^j\big)^2 + \frac{9}{2}\big(H_{s,t}^j\big)^2 + \frac{1}{4}h\bigg)\\[-3pt]
&\mmm + \frac{1}{11200}h^3\bigg(\frac{7}{48}h + \frac{1}{8}\big(H_{s,t}^i\big)^2 + \frac{1}{8}\big(H_{s,t}^j\big)^2\bigg)\\
& = \frac{1}{28800}h^4 + \frac{1}{23040}h^3\Big(\big(W_{s,t}^i\big)^2  + \big(W_{s,t}^j\big)^2\Big) + \frac{9}{22400}h^3 \Big(\big(H_{s,t}^i\big)^2 + \big(H_{s,t}^j\big)^2\Big).
\end{align*}
Thus, the expectation of the variance (\ref{eq:working_var}) conditional on the triple $(W_{s,t}\m, H_{s,t}\m, n_{s,t})$ is\
\begin{align*}
&\E\big[\var\big(L_{s,t}^{ij}\,\big|\, W_{s,u}\m, W_{u,t}\m, H_{s,u}\m, H_{u,t}\big)\,\big|\,W_{s,t}\m, H_{s,t}\m, n_{s,t}\big]\\
&\mm = \,\widehat{V}_{s,t}^{ij} + \frac{1}{23040}h^3\Big(\big(W_{s,t}^i\big)^2  + \big(W_{s,t}^j\big)^2\Big) + \frac{9}{22400}h^3 \Big(\big(H_{s,t}^i\big)^2 + \big(H_{s,t}^j\big)^2\Big),
\end{align*}
where $\, \widehat{V}_{s,t}^{ij} := \begin{cases} \frac{1}{11200}\m h^4, & \text{ if } i = j\\[3pt]
\frac{1}{5760}\m h^4, & \text{ if } i\neq j\end{cases}$. To simplify calculations, we will first assume $i=j$.\medbreak\noindent
Using equation (\ref{eq:first_cond_exp}) along with the half-normal moments (\ref{eq:half_normal_odd_moments}) and (\ref{eq:half_normal_even_moments}), we have that
\begin{align*}
&\E\Big[\E\big[L_{s,t}^{ij}\,\big|\, W_{s,u}\m, W_{u,t}\m, H_{s,u}\m, H_{u,t}\big]^2\,\big|\, W_{s,t}\m, H_{s,t}\m, n_{s,t}\Big]\\
&\mm = \E\bigg[\bigg(\frac{3}{5}h\big(H_{s,t}^i\big)^2 + \frac{1}{15}h \big(Z_{s,u}^i\big)^2 + \frac{3}{20}h \big(N_{s,t}^i\big)^2 + \frac{1}{60}\m h^2\\[-3pt]
&\hspace{20mm} + \frac{1}{10}h H_{s,t}^i Z_{s,u}^i - \frac{1}{8}h\m W_{s,t}^i N_{s,t}^i\bigg)^2\,\Big|\, W_{s,t}\m, H_{s,t}\m, n_{s,t}\bigg]\\
&\mm = \frac{9}{25}h^2\big(H_{s,t}^i\big)^4 + \frac{2}{25}h^2\big(H_{s,t}^i\big)^2\m\E\Big[\big(Z_{s,u}^i\big)^2\Big] + \frac{9}{50}h^2\big(H_{s,t}^i\big)^2\m\E\Big[\big(N_{s,t}^i\big)^2\m\big|\m n_{s,t}\Big] \\
&\mmm + \frac{1}{50}h^3\big(H_{s,t}^i\big)^2 + \frac{3}{25}h^3\big(H_{s,t}^i\big)^3 \E\big[Z_{s,u}^i\big] - \frac{3}{20}h^2 W_{s,t}^i\big(H_{s,t}^i\big)^2\,\E\big[N_{s,t}^i\m\big|\m n_{s,t}\big]\\
&\mmm + \frac{1}{225}h^2\m\E\Big[\big(Z_{s,u}^i\big)^4\Big] + \frac{1}{50}h^2\m\E\Big[\big(Z_{s,u}^i\big)^2\Big]\E\Big[\big(N_{s,t}^i\big)^2\m\big|\m n_{s,t}\Big] + \frac{1}{450}h^3\m\E\Big[\big(Z_{s,u}^i\big)^2\Big] \\
&\mmm + \frac{1}{75}h^2 H_{s,t}^i\E\Big[\big(Z_{s,u}^i\big)^3\Big] - \frac{1}{60}h^2 W_{s,t}^i\m \E\Big[\big(Z_{s,u}^i\big)^2\Big]\E\big[N_{s,t}^i\m\big|\m n_{s,t}\big]\\
&\mmm + \frac{9}{400}\m h^2\E\Big[\big(N_{s,t}^i\big)^4\m\big|\m n_{s,t}\Big] + \frac{1}{200}h^3\m\E\Big[\big(N_{s,t}^i\big)^2\m\big|\m n_{s,t}\Big]\\
&\mmm + \frac{3}{100}h^2 H_{s,t}^i \m\E\big[Z_{s,u}^i\big]\E\Big[\big(N_{s,t}^i\big)^2\m\big|\m n_{s,t}\Big] - \frac{3}{80}h W_{s,t}^i\m \E\Big[\big(N_{s,t}^i\big)^3\m\big|\m n_{s,t}\Big] + \frac{1}{3600}h^4 \\
&\mmm + \frac{1}{300}h^3 H_{s,t}^i\m \E\big[Z_{s,u}^i\big] - \frac{1}{240} h^3\m W_{s,t}^i\m\E\big[N_{s,t}^i\m\big|\m n_{s,t}\big] + \frac{1}{100}h^2 \big(H_{s,t}^i\big)^2\m\E\Big[\big(Z_{s,u}^i\big)^2\Big]\\
&\mmm - \frac{1}{40} h^2\m W_{s,t}^i H_{s,t}^i\m \E\big[Z_{s,u}^i\big]\E\big[N_{s,t}^i\m\big|\m n_{s,t}\big] + \frac{1}{64}h^2\big(W_{s,t}^i\big)^2\m\E\Big[\big(N_{s,t}^i\big)^2\m\big|\m n_{s,t}\Big]\\
&\mm = \frac{7}{4800}h^4 + \frac{1}{768}h^2\big(W_{s,t}^i\big)^2 + \frac{13}{320}h^2\big(H_{s,t}^i\big)^2 + \frac{9}{25}h^2\big(H_{s,t}^i\big)^4\\
&\mmm - \frac{3}{20\sqrt{6\pi}}\m W_{s,t}^i\big(H_{s,t}^i\big)^2 n_{s,t}^i h^\frac{5}{2} - \frac{11}{960\sqrt{6\pi}}\m W_{s,t}^i\m n_{s,t}^i h^\frac{7}{2}\m.
\end{align*}
We can now compute the second moment of $L_{s,t}^{ij}$ conditional on $(W_{s,t}\m, H_{s,t}\m, n_{s,t})$ as
\begin{align*}
\E\Big[\big(L_{s,t}^{ij}\big)^2\,\big|\, W_{s,t}\m, H_{s,t}\m, n_{s,t}\Big]
& = \E\Big[\E\big[\big(L_{s,t}^{ij}\big)^2\,\big|\, W_{s,u}\m, W_{u,t}\m, H_{s,u}\m, H_{u,t}\big]\,\big|\, W_{s,t}\m, H_{s,t}\m, n_{s,t}\Big]\\[3pt]
& = \E\Big[\E\big[L_{s,t}^{ij}\,\big|\, W_{s,u}\m, W_{u,t}\m, H_{s,u}\m, H_{u,t}\big]^2\,\big|\, W_{s,t}\m, H_{s,t}\m, n_{s,t}\Big]\\
&\hspace{4mm} + \E\Big[\var\big(L_{s,t}^{ij}\,\big|\, W_{s,u}\m, W_{u,t}\m, H_{s,u}\m, H_{u,t}\big)\,\big|\, W_{s,t}\m, H_{s,t}\m, n_{s,t}\Big].
\end{align*}
Using the formulae for these two terms, we can obtain the following second moment.
\begin{align*}
\E\Big[\big(L_{s,t}^{ij}\big)^2\,\big|\, W_{s,t}\m, H_{s,t}\m, n_{s,t}\Big] & = \frac{7}{4800} h^4 + \frac{1}{768}h^2\big(W_{s,t}^i\big)^2 + \frac{13}{320}h^3 \big(H_{s,t}^i\big)^2 + \frac{9}{25}h^2\big(H_{s,t}^i\big)^4\\
&\mm - \frac{3}{20\sqrt{6\pi}}\m W_{s,t}^i\big(H_{s,t}^i\big)^2 n_{s,t}^i h^\frac{5}{2} - \frac{11}{960\sqrt{6\pi}} W_{s,t}^i n_{s,t}^i h^\frac{7}{2}\\
&\mm + \frac{1}{11200}h^4 + \frac{1}{11520}h^3 \big(W_{s,t}^i\big)^2 + \frac{9}{11200}h^3 \big(H_{s,t}^i\big)^2\\[3pt]
& = \frac{13}{8400} h^4 + \frac{1}{720}h^2\big(W_{s,t}^i\big)^2 + \frac{29}{700}h^3 \big(H_{s,t}^i\big)^2 + \frac{9}{25}h^2\big(H_{s,t}^i\big)^4 \\
&\mm  - \frac{3}{20\sqrt{6\pi}}\m W_{s,t}^i\big(H_{s,t}^i\big)^2 n_{s,t}^i h^\frac{5}{2} - \frac{11}{960\sqrt{6\pi}} W_{s,t}^i n_{s,t}^i h^\frac{7}{2}.
\end{align*}
From these moments, we can now compute the conditional variance (\ref{eq:sst_var_whn}) when $i=j$,
\begin{align*}
&\var\big(L_{s,t}^{ij}\,\big|\, W_{s,t}\m, H_{s,t}\m, n_{s,t}\big)\\
&\mm = \E\Big[\big(L_{s,t}^{ij}\big)^2\,\big|\, W_{s,t}\m, H_{s,t}\m, n_{s,t}\Big] - \Big(\E\big[\m L_{s,t}^{ij}\,\big|\, W_{s,t}\m, H_{s,t}\m, n_{s,t}\big]\Big)^2\\
&\mm = \frac{13}{8400} h^4 + \frac{1}{720}h^2\big(W_{s,t}^i\big)^2 + \frac{29}{700}h^3 \big(H_{s,t}^i\big)^2 + \frac{9}{25}h^2\big(H_{s,t}^i\big)^4 - \frac{11}{960\sqrt{6\pi}} W_{s,t}^i n_{s,t}^i h^\frac{7}{2}\\
&\mmm  - \frac{3}{20\sqrt{6\pi}}\m W_{s,t}^i\big(H_{s,t}^i\big)^2 n_{s,t}^i h^\frac{5}{2} - \Big(\frac{3}{5}h \big(H_{s,t}^i\big)^2 - \frac{1}{8\sqrt{6\pi}}W_{s,t}^i n_{s,t}^i h^\frac{3}{2} + \frac{1}{30}h^2\Big)^2\\
&\mm = \frac{11}{25200} h^4  + \bigg(\frac{1}{720} - \frac{1}{384\pi}\bigg)h^2\big(W_{s,t}^i\big)^2 + \frac{1}{700}h^3 \big(H_{s,t}^i\big)^2 - \frac{1}{320\sqrt{6\pi}} W_{s,t}^i n_{s,t}^i h^\frac{7}{2}.
\end{align*}
For the remainder for the proof, we will consider the more involved case where $i\neq j$. However, our strategy to computing the conditional variance (\ref{eq:sst_var_whn}) remains unchanged.
We first derive the conditional expectations for the second conditional moment of $L_{s,t}^{ij}\m$.
Here, we shall simplify the expansion using the independence and unbiasedness of $Z_{s,u}\m$.
\begin{align*}
&\E\Big[\E\big[L_{s,t}^{ij}\,\big|\, W_{s,u}\m, W_{u,t}\m, H_{s,u}\m, H_{u,t}\big]^2\,\big|\, W_{s,t}\m, H_{s,t}\m, N_{s,t}\Big]\\
&\mm= \E\bigg[\bigg(\frac{3}{5}h H_{s,t}^i H_{s,t}^j + \frac{1}{15}h Z_{s,u}^i Z_{s,u}^j + \frac{3}{20}h N_{s,t}^i N_{s,t}^j + \frac{1}{20}h\big(H_{s,t}^i Z_{s,u}^j + Z_{s,u}^i H_{s,t}^j\big)\\[-3pt]
&\hspace{20mm} - \frac{1}{16}h\big(W_{s,t}^i N_{s,t}^j + N_{s,t}^i W_{s,t}^j\big)\bigg)^2\,\Big|\, W_{s,t}\m, H_{s,t}\m, N_{s,t}\bigg]\\[-2pt]
&\mm = \E\bigg[\bigg(\frac{3}{5}h H_{s,t}^i H_{s,t}^j - \frac{1}{16}h\big(W_{s,t}^i N_{s,t}^j + N_{s,t}^i W_{s,t}^j\big) + \frac{3}{20}h N_{s,t}^i N_{s,t}^j \bigg)^2\Big|\, W_{s,t}\m, H_{s,t}\m, N_{s,t}\bigg]\\[-3pt]
&\hspace{20mm} + \frac{1}{400}h^2\E\Big[\big(H_{s,t}^i Z_{s,u}^j + Z_{s,u}^i H_{s,t}^j\big)^2\m\big|\m H_{s,t}\Big]  + \frac{1}{225}h^2\m\E\Big[\big(Z_{s,u}^i\big)^2\big(Z_{s,u}^j\big)^2\Big]\\[1pt]
&\mm = \frac{1}{57600}h^4 + \bigg(\frac{3}{5}h H_{s,t}^i H_{s,t}^j - \frac{1}{16}h\big(W_{s,t}^i N_{s,t}^j + N_{s,t}^i W_{s,t}^j\big) + \frac{3}{20}h N_{s,t}^i N_{s,t}^j \bigg)^2\\[-3pt]
&\hspace{20mm} + \frac{1}{6400}h^3 \Big(\big(H_{s,t}^i\big)^2 + \big(H_{s,t}^j\big)^2\Big).
\end{align*}
Hence, taking the expectation of the above conditional on $(W_{s,t}\m, H_{s,t}\m, n_{s,t})$, we obtain
\begin{align*}
&\E\Big[\E\big[L_{s,t}^{ij}\,\big|\, W_{s,u}\m, W_{u,t}\m, H_{s,u}\m, H_{u,t}\big]^2\,\big|\, W_{s,t}\m, H_{s,t}\m, n_{s,t}\Big]\\
&\mm = \E\Big[\E\Big[\E\big[L_{s,t}^{ij}\,\big|\, W_{s,u}\m, W_{u,t}\m, H_{s,u}\m, H_{u,t}\big]^2\,\big|\, W_{s,t}\m, H_{s,t}\m, N_{s,t}\Big]\,\Big|\, W_{s,t}\m, H_{s,t}\m, n_{s,t}\Big]\\[2pt]
&\mm = \frac{1}{57600}h^4 + \frac{1}{6400}h^3 \Big(\big(H_{s,t}^i\big)^2 + \big(H_{s,t}^j\big)^2\Big) + \frac{9}{25}h^2\big(H_{s,t}^i\big)^2\big(H_{s,t}^j\big)^2\\
&\mmm - \frac{3}{40}h^2\m W_{s,t}^i H_{s,t}^i H_{s,t}^j\m \E\big[N_{s,t}^j\m\big|\m n_{s,t}\big] - \frac{3}{40}h^2 W_{s,t}^j H_{s,t}^i H_{s,t}^j\m \E\big[N_{s,t}^i\m\big|\m n_{s,t}\big]\\
&\mmm + \frac{9}{50}h^2 H_{s,t}^i H_{s,t}^j\m  \E\big[N_{s,t}^i N_{s,t}^j\m\big|\m n_{s,t}\big] + \frac{1}{256}h^2 \big(W_{s,t}^i\big)^2\m\E\Big[\big(N_{s,t}^j\big)^2\m\big|\m n_{s,t}\Big]\\
&\mmm + \frac{1}{128}h^2\m W_{s,t}^i W_{s,t}^j\m\E\big[N_{s,t}^i N_{s,t}^j\m\big|\m n_{s,t}\big] - \frac{3}{160}h^2\m W_{s,t}^i\m\E\big[N_{s,t}^i\m\big|\m n_{s,t}\big]\E\Big[\big(N_{s,t}^j\big)^2\m\big|\m n_{s,t}\Big]\\
&\mmm + \frac{1}{256}h^2\m\big(W_{s,t}^j\big)^2\m\E\Big[\big(N_{s,t}^i\big)^2\m\big|\m n_{s,t}\Big] - \frac{3}{160}h^2\m W_{s,t}^j\m\E\Big[\big(N_{s,t}^i\big)^2\m\big|\m n_{s,t}\Big]\E\big[N_{s,t}^j\m\big|\m n_{s,t}\big]\\
&\mmm + \frac{9}{400}h^2\m\E\Big[\big(N_{s,t}^i\big)^2\m\big|\m n_{s,t}\Big]\E\Big[\big(N_{s,t}^j\big)^2\m\big|\m n_{s,t}\Big]\\[3pt]
&\mm = \frac{1}{5760}h^4 + \frac{1}{3072}h^2 \Big(\big(W_{s,t}^i\big)^2 + \big(W_{s,t}^j\big)^2\Big) + \frac{1}{6400}h^3 \Big(\big(H_{s,t}^i\big)^2 + \big(H_{s,t}^j\big)^2\Big)\\
&\mmm + \frac{9}{25}h^2\big(H_{s,t}^i\big)^2\big(H_{s,t}^j\big)^2 - \frac{3}{40\sqrt{6\pi}}H_{s,t}^i H_{s,t}^j \big(W_{s,t}^i \m n_{s,t}^j +  W_{s,t}^j \m n_{s,t}^i\big) h^\frac{5}{2}\\
&\mmm + \frac{1}{768\pi}h^3\m W_{s,t}^i W_{s,t}^j\m n_{s,t}^i n_{s,t}^j + \frac{3}{100\pi}h^3 H_{s,t}^i H_{s,t}^j n_{s,t}^i n_{s,t}^j\\[2pt]
&\mmm - \frac{1}{640\sqrt{6\pi}}\big(W_{s,t}^i\m n_{s,t}^i +  W_{s,t}^j\m n_{s,t}^j\big)h^\frac{7}{2}.
\end{align*}
Thus, the second moment of $L_{s,t}^{ij}$ conditional on the triple $(W_{s,t}\m, H_{s,t}\m, n_{s,t})$ is given by
\begin{align*}
&\E\Big[\big(L_{s,t}^{ij}\big)^2\,\big|\, W_{s,t}\m, H_{s,t}\m, n_{s,t}\Big]\\
&\mm = \E\Big[\E\big[\big(L_{s,t}^{ij}\big)^2\,\big|\, W_{s,u}\m, W_{u,t}\m, H_{s,u}\m, H_{u,t}\big]\,\big|\, W_{s,t}\m, H_{s,t}\m, n_{s,t}\Big]\\[3pt]
&\mm = \E\Big[\E\big[L_{s,t}^{ij}\,\big|\, W_{s,u}\m, W_{u,t}\m, H_{s,u}\m, H_{u,t}\big]^2\,\big|\, W_{s,t}\m, H_{s,t}\m, n_{s,t}\Big]\\
&\mmm + \E\Big[\var\big(L_{s,t}^{ij}\,\big|\, W_{s,u}\m, W_{u,t}\m, H_{s,u}\m, H_{u,t}\big)\,\big|\, W_{s,t}\m, H_{s,t}\m, n_{s,t}\Big]\\[1pt]
&\mm = \frac{1}{2880}h^4 + \frac{17}{46080}h^2 \Big(\big(W_{s,t}^i\big)^2 + \big(W_{s,t}^j\big)^2\Big) + \frac{1}{1792}h^3 \Big(\big(H_{s,t}^i\big)^2 + \big(H_{s,t}^j\big)^2\Big)\\
&\mmm + \frac{9}{25}h^2\big(H_{s,t}^i\big)^2\big(H_{s,t}^j\big)^2 - \frac{3}{40\sqrt{6\pi}}H_{s,t}^i H_{s,t}^j \big(W_{s,t}^i \m n_{s,t}^j +  W_{s,t}^j \m n_{s,t}^i\big) h^\frac{5}{2}\\
&\mmm + \frac{1}{768\pi}h^3\m W_{s,t}^i W_{s,t}^j\m n_{s,t}^i n_{s,t}^j + \frac{3}{100\pi}h^3 H_{s,t}^i H_{s,t}^j n_{s,t}^i n_{s,t}^j\\[2pt]
&\mmm - \frac{1}{640\sqrt{6\pi}}\big(W_{s,t}^i\m n_{s,t}^i +  W_{s,t}^j\m n_{s,t}^j\big)h^\frac{7}{2}.
\end{align*}
Finally, with these moments, we can compute the conditional variance (\ref{eq:sst_var_whn}) when $i\neq j$.
\begin{align*}
&\var\big(L_{s,t}^{ij}\,\big|\, W_{s,t}\m, H_{s,t}\m, n_{s,t}\big)\\
&\mm = \E\Big[\big(L_{s,t}^{ij}\big)^2\,\big|\, W_{s,t}\m, H_{s,t}\m, n_{s,t}\Big] - \Big(\E\big[\m L_{s,t}^{ij}\,\big|\, W_{s,t}\m, H_{s,t}\m, n_{s,t}\big]\Big)^2\\[1pt]
&\mm = \frac{1}{2880}h^4 + \frac{17}{46080}h^2 \Big(\big(W_{s,t}^i\big)^2 + \big(W_{s,t}^j\big)^2\Big) + \frac{1}{1792}h^3 \Big(\big(H_{s,t}^i\big)^2 + \big(H_{s,t}^j\big)^2\Big)\\
&\mmmm + \frac{9}{25}h^2\big(H_{s,t}^i\big)^2\big(H_{s,t}^j\big)^2 - \frac{3}{40\sqrt{6\pi}}H_{s,t}^i H_{s,t}^j \big(W_{s,t}^i \m n_{s,t}^j +  W_{s,t}^j \m n_{s,t}^i\big) h^\frac{5}{2}\\
&\mmmm + \frac{1}{768\pi}h^3\m W_{s,t}^i W_{s,t}^j\m n_{s,t}^i n_{s,t}^j + \frac{3}{100\pi}h^3 H_{s,t}^i H_{s,t}^j n_{s,t}^i n_{s,t}^j\\[2pt]
&\mmmm - \frac{1}{640\sqrt{6\pi}}\big(W_{s,t}^i\m n_{s,t}^i +  W_{s,t}^j\m n_{s,t}^j\big)h^\frac{7}{2}\\
&\mmmm - \bigg(\frac{3}{5}h H_{s,t}^i H_{s,t}^j - \frac{1}{16\sqrt{6\pi}}\big(W_{s,t}^i n_{s,t}^j + n_{s,t}^i W_{s,t}^j\big)h^\frac{3}{2} + \frac{1}{40\pi}h^2\m n_{s,t}^i n_{s,t}^j\bigg)^2\\[3pt]
&\mm = \bigg(\frac{1}{2880} - \frac{1}{1600\pi^2}\bigg) h^4 + \bigg(\frac{17}{46080} - \frac{1}{1536\pi}\bigg)h^3\Big(\big(W_{s,t}^i\big)^2 + \big(W_{s,t}^j\big)^2\Big)\\
&\mmmm  + \frac{1}{1792}h^3\Big(\big(H_{s,t}^i\big)^2 + \big(H_{s,t}^j\big)^2\Big) - \frac{1}{640\sqrt{6\pi}}\bigg(1 - \frac{2}{\pi}\bigg)\big(W_{s,t}^i n_{s,t}^i +W_{s,t}^j n_{s,t}^j\big)h^\frac{7}{2}.
\end{align*}
\end{proof}

To conclude this section, we present an estimator of $K_{s,t}$ which was used with the approximation (\ref{eq:sst_exp_whn}) of $L_{s,t}$ to design a high order splitting method \cite[Example 4.7]{foster2024splitting}.\medbreak
\begin{theorem}\label{thm:new_stt_estimator}
The space-time-time L\'{e}vy area $K_{s,t}$ can be approximated using $n_{s,t}$ as
\begin{align}
\E\big[K_{s,t} \,|\, n_{s,t}\big] & = \frac{1}{8\sqrt{6\pi}}\m n_{s,t} h^{\frac{1}{2}}\m,\hspace{1.25mm}\label{eq:K_mean}\\[3pt]
\var\big(K_{s,t}^i \,|\, n_{s,t}\big) & = \bigg(\frac{1}{720} - \frac{1}{384\pi}\bigg)\m h\m.\hspace{1.25mm}\label{eq:K_second_moment}
\end{align}
\end{theorem}
\begin{proof} We first consider the expectation of $K_{s,t}$ conditional on $W_{s,u}\m, W_{u,t}\m, H_{s,u}\m, H_{u,t}\m$,
which will be straightforward to compute from the definition of $K_{s,t}\m$, see equation (\ref{eq:space_time_time}). 
\begin{align*}
&h^2\m\E\Big[K_{s,t}\,\big|\, W_{s,u}\m, W_{u,t}\m, H_{s,u}\m, H_{u,t}\m\Big]\\
&\mm = \E\Bigg[\int_s^t \bigg(W_{s,r} - \frac{r-s}{h}\,W_{s,t}\bigg)\bigg(\frac{1}{2}h - (r-s)\bigg) dr\m\Big|\, W_{s,u}\m, W_{u,t}\m, H_{s,u}\m, H_{u,t}\m\Bigg]\\
&\mm = \frac{1}{2}h\int_s^t W_{s,r}\, dr - \E\Bigg[\int_s^t W_{s,r}(r-s)\m dr\m\Big|\, W_{s,u}\m, W_{u,t}\m, H_{s,u}\m, H_{u,t}\m\Bigg] + \frac{1}{12}h^2 W_{s,t}\\
&\mm =\frac{5}{12}h^2 W_{s,t} + \frac{1}{2}h^2 H_{s,t} - \int_s^t \E\Big[W_{s,r}\m\big|\, W_{s,u}\m, W_{u,t}\m, H_{s,u}\m, H_{u,t}\m\Big](r-s)\m dr.
\end{align*}
Using the independence of $\{W_{s,r}\}_{r\in[s,u]}$ and $\{W_{u,r}\}_{r\in[u,t]}$, along with the quadratic polynomial expansion of Brownian motion (see \cite[Theorem 2.4]{foster2020poly} with $n=2$), we have
\begin{align*}
&h^2\m\E\big[K_{s,t}\m|\, W_{s,u}\m, W_{u,t}\m, H_{s,u}\m, H_{u,t}\m\big]\\
& =\frac{1}{3}h^2 W_{s,t} + \frac{1}{2}h^2 H_{s,t} - \int_s^u \E\Big[W_{s,r}\m\big|\, W_{s,u}\m, H_{s,u}\m\Big](r-s)\m dr\\
&\mm - \int_u^t \Big(W_{s,u} + \E\Big[W_{u,r}\m\big|\, W_{u,t}\m, H_{u,t}\m\Big]\Big)(r-s)\m dr\\
&= \frac{1}{3}h^2 W_{s,t}+ \frac{1}{2}h^2 H_{s,t} - \int_s^u \bigg(\frac{2(r-s)}{h}W_{s,u} + \frac{24(r-s)(u-r)}{h^2}H_{s,u}\bigg)(r-s)\m dr\\
&\mm - \frac{3}{8}h^2 W_{s,u} - \int_u^t \bigg(\frac{2(r-u)}{h}W_{u,t} + \frac{24(r-u)(t-r)}{h^2}H_{u,t}\bigg)(r-s)\m dr\\
& = \frac{1}{3}h^2 W_{s,t}+ \frac{1}{2}h^2 H_{s,t} - \frac{1}{12}h^2 W_{s,u} - \frac{1}{8}h^2 H_{s,u} - \frac{3}{8}h^2 W_{s,u} - \frac{5}{24}h^2 W_{u,t} - \frac{3}{8}h^2  H_{u,t}\m.
\end{align*}
Expressing the above in terms of $W_{s,t}\m, H_{s,t}\m, Z_{s,u}\m, N_{s,t}$ (see Theorem \ref{thm:whzn_independence}), it follows that
\begin{align*}
&\E\big[K_{s,t}\m|\, W_{s,t}\m, H_{s,t}\m, Z_{s,u}\m, N_{s,t}\m\big]\\
&\mm = \frac{1}{3} W_{s,t}+ \frac{1}{2} H_{s,t} - \frac{11}{24} \bigg(\frac{1}{2}W_{s,t} + \frac{3}{2}H_{s,t} + Z_{s,u}\bigg) - \frac{1}{8}\bigg(\frac{1}{4}H_{s,t} - \frac{1}{2}Z_{s,u} + \frac{1}{2}N_{s,t}\bigg)\\
&\mmm  - \frac{5}{24}\bigg(\frac{1}{2}W_{s,t} - \frac{3}{2}H_{s,t} - Z_{s,u}\bigg) - \frac{3}{8}\bigg(\frac{1}{4}H_{s,t} - \frac{1}{2}Z_{s,u} - \frac{1}{2}N_{s,t}\bigg) = \frac{1}{8}h^2 N_{s,t}\m.
\end{align*}
It follows by the tower law and the conditional expectation of $N_{s,t}$ (equation (\ref{eq:half_normal_odd_moments})) that
\begin{align*}
\E\big[K_{s,t}\,\big|\, n_{s,t}\m\big] =  \E\big[\E\big[K_{s,t}\m|\, W_{s,t}\m, H_{s,t}\m, Z_{s,u}\m, N_{s,t}\m\big]\m|\, n_{s,t}\big] = \frac{1}{8}\E\big[N_{s,t}\m |\,n_{s,t}\big] = \frac{1}{8\sqrt{6\pi}}\m n_{s,t} h^{\frac{1}{2}}.
\end{align*}
Finally, we note that each $\big(K_{s,t}^i\big)^2$ will remain unchanged if $W$ is replaced by $-W$,\vspace{0.5mm} whereas $n_{s,t} = \sgn(H_{s,u} - H_{u,t})$ changes sign when the Brownian motion is ``flipped''.
Thus, by the symmetry of $W$, the random variables $\big(K_{s,t}^i\big)^2$ and $n_{s,t}$ are uncorrelated.
So using the law of total expectation, we can write $\E\big[\big(K_{s,t}^i\big)^2 n_{s,t}\big]$ and $\E\big[\big(K_{s,t}^i\big)^2\m\big]$ as
\begin{align*}
\underbrace{\E\Big[\big(K_{s,t}^i\big)^2 n_{s,t}\Big]}_{=\,0} & = \frac{1}{2}\m\E\Big[\big(K_{s,t}^i\big)^2\m|\m n_{s,t} = 1\m\Big] + \frac{1}{2}\m\E\big[-\big(K_{s,t}^i\big)^2\m|\m n_{s,t} = -1\big],\\[3pt]
\underbrace{\E\Big[\big(K_{s,t}^i\big)^2\Big]}_{=\,\frac{1}{720} h} & = \frac{1}{2}\m\E\Big[\big(K_{s,t}^i\big)^2\m|\m n_{s,t} = 1\m\Big] + \frac{1}{2}\m\E\Big[\big(K_{s,t}^i\big)^2\m|\m n_{s,t} = -1\m\Big].
\end{align*}
Solving this system gives $\E\big[\big(K_{s,t}^i\big)^2\m|\m n_{s,t} = 1\m\big] = \E\big[\big(K_{s,t}^i\big)^2\m|\m n_{s,t} = -1\m\big] = \frac{1}{720} h$ and  $\var\big(K_{s,t}^i \,|\, n_{s,t}\big) = \E\big[\big(K_{s,t}^i\big)^2\m|\m n_{s,t}\big] - \E\big[K_{s,t}\,\big|\, n_{s,t}\m\big]^2 = \frac{1}{720}h - \frac{1}{384\pi}h$, as required.
\end{proof}

\section{Numerical examples}\label{sect:examples}

In this section, we will present several numerical examples demonstrating the accuracy achievable when integrals in the Brownian signature are used for SDE simulation.
These numerical examples were taken from the literature \cite{foster2024splitting, jelincic2025levygan, foster2020poly, jelincic2024vbt, foster2021uld, foster2024adaptive, tubikanec2022igbm}, so will be briefly presented, and we refer the reader to these papers for further details.

\subsection{Simulating Langevin dynamics using Gaussian integrals}

The underdamped Langevin diffusion (ULD) is given the following system of SDEs,
\begin{align}
dx_t & = v_t\m dt,\hspace{7.5mm} dv_t = -\gamma v_t\m dt - \nabla f(x_t)\m dt + \sqrt{2\gamma}\m dW_t\m,\label{eq:uld2}
\end{align}
where $x, v\in\R^d$ represent the position and momentum of a particle, $f : \R^d\rightarrow \R$ is a scalar potential, $\gamma\m$ is a friction parameter and $W$ is a $d$-dimensional Brownian motion.\medbreak

Under mild conditions on $f$, the SDE (\ref{eq:uld2}) is known to admit an ergodic strong solution with a stationary distribution $\pi(x,v)\propto e^{-f(x)} e^{-\frac{1}{2}\|v\|^2}$  \cite[Proposition 6.1]{pavliotis2014spa}.
As a result, there has been great interest in the application of ULD as a Markov Chain Monte Carlo (MCMC) algorithm for high-dimensional sampling \cite{cheng2018uld, chen2014sghmc, chada2024uld, dalalyan2020uld, rioudurand2023malt, scott2025uld}.\medbreak

Due to the structure of (\ref{eq:uld2}), it is possible for SDE solvers to achieve a third order convergence rate by using $(W_{s,t}\m, H_{s,t}\m, K_{s,t})$. Currently, the following solver is the only to exhibit such fast convergence whilst only requiring evaluations of the gradient $\nabla f$,\medbreak

\begin{definition}[The SORT\footnote{\underline{S}hifted \underline{O}DE with \underline{R}unge-Kutta \underline{T}hree} method for Underdamped Langevin Dynamics {\cite{foster2024splitting, foster2021uld}}]\label{def:sort}
\begin{align*}
V_n^{(1)} & := V_n + \sqrt{2\gamma }\,(H_{t_n,\m t_{n+1}} + 6K_{t_n,\m t_{n+1}}),\\[3pt]
X_n^{(1)} & := X_n + \bigg(\frac{1-e^{-\frac{1}{2}\gamma h_n}}{\gamma}\bigg)V_n^{(1)} - \bigg(\frac{e^{-\frac{1}{2}\gamma h_n} + \frac{1}{2}\gamma h_n - 1}{\gamma^2}\bigg)\nabla f(X_n)\\
&\hspace{10.5mm} + \bigg(\frac{e^{-\frac{1}{2}\gamma h_n} + \frac{1}{2}\gamma h_n - 1}{\gamma^2 h_n}\bigg)\sqrt{2\gamma}\,\big(W_{t_n,\m t_{n+1}} - 12 K_{t_n,\m t_{n+1}}\big),\\[3pt]
X_{n+1} & := X_n + \bigg(\frac{1-e^{-\gamma h_n}}{\gamma}\bigg)V_n^{(1)} - \bigg(\frac{e^{-\gamma h_n} + \gamma h_n - 1}{\gamma^2}\bigg)\bigg(\frac{1}{3}\nabla f(X_n) + \frac{2}{3}\nabla f\big(X_n^{(1)}\big)\bigg)\\
&\hspace{10.5mm} + \bigg(\frac{e^{-\gamma h_n} + \gamma h_n - 1}{\gamma^2 h_n}\bigg)\sqrt{2\gamma}\,\big(W_{t_n,\m t_{n+1}} - 12 K_{t_n,\m t_{n+1}}\big),\\[3pt]
V_n^{(2)} & := e^{-\gamma h_n} V_n^{(1)} - \frac{1}{6}e^{-\gamma h_n} \nabla f(X_n) h_n - \frac{2}{3}e^{-\frac{1}{2}\gamma h_n} \nabla f\big(X_n^{(1)}\big) h_n - \frac{1}{6}\nabla f(X_{n+1})h_n\\
&\hspace{10.5mm} + \bigg(\frac{1 - e^{-\gamma h_n}}{\gamma h_n} \bigg)\sqrt{2\gamma}\,\big(W_{t_n,\m t_{n+1}} - 12 K_{t_n,\m t_{n+1}}\big),\\[3pt]
V_{n+1} & := V_n^{(2)} - \sqrt{2\gamma}\,(H_{t_n,\m t_{n+1}} - 6K_{t_n,\m t_{n+1}}),
\end{align*}
where $h_n := t_{n+1} - t_n > 0\m$ denotes the step size and $(X_n\m, V_n)$ approximates $(x_{t_n}\m, v_{t_n})$.
\end{definition}

\begin{remark}
Since $\nabla f(X_{n+1})$ can be used in the next step to compute $(X_{n+2}, V_{n+2})$, the SORT method only uses two additional evaluations of the gradient $\nabla f$ per step.
Hence, this is an example of the ``First Same As Last'' property in numerical analysis.
Moreover, as evaluating $\nabla f$ is typically much more computationally expensive than generating $d$-dimensional Gaussian random vectors in practice, we would expect that the SORT method is about twice as expensive per step as the Euler-Maruyama method.
\end{remark}\medbreak

Firstly, to demonstrate the third order strong convergence of the SORT method, we present an experiment in \cite{foster2021uld, foster2024splitting} where $f$ comes from a Bayesian logistic regression.
This involves German credit data from \cite{uci}, where each of the $m=1000$ individuals has $d=49$ features $x_i\in\R^d$ and a label $y_i\in\{-1, 1\}$ indicating if they are creditworthy.
The Bayesian logistic regression model states that $\mathbb{P}(Y_i = y_i | x_i) = (1+e^{-y_i x_i^{\top} \theta})^{-1}$
where $\theta\in\R^d$ are parameters coming from the target density $\pi(\theta)\propto \exp(-f(\theta))$ with
\begin{align*}
f(\theta) = \delta\|\theta\|^2 + \sum_{i=1}^{m} \log\big(1 + \exp\big(- y_i x_i^{\top} \theta\big)\big)\m.
\end{align*}
In the experiment, the regularisation parameter is $\delta = 0.05$, the friction coefficient is $\gamma = 2$ and the initial parameter configuration $\theta_0$ is sampled from a Gaussian prior as
\begin{align*}
\theta_0\sim\mathcal{N}\big(0, 10 I_d\big).
\end{align*}
We will use a fixed time horizon of $T=1000$ and compute the following error estimator:\medbreak
\begin{definition}[\textbf{Strong error estimator}] For $N\geq 1$, let $\big\{\theta_k^h\big\}_{0\m\leq k\m\leq\m N}$ denote a numerical solution of (\ref{eq:uld2}) computed over $[0,T]$ at times $t_k := kh$ with step size $h = \frac{T}{N}$. Let $\big\{\theta_k^{\frac{1}{2}h}\big\}_{0\m\leq k\m\leq\m 2N}$ be the approximation obtained by using a smaller step size of $\frac{1}{2}h$.
We generate $n$ samples of these numerical solutions and define an estimator at time $T$,
\begin{align}\label{eq:strong_estimator}
S_{N,n} := \sqrt{\,\frac{1}{n}\sum_{i=1}^n\big\|\theta_{N, i}^h - \theta_{2N, i}^{\frac{1}{2}h}\big\|^2\,}\,,
\end{align}
where each $\big(\theta_{N, i}^k\m, \theta_{2N, i}^{\frac{1}{2}h}\m\big)$ is computed from the same sample path of Brownian motion
and each initial value $\theta_{0, i}^{h\begin{matrix}\\[-12pt]\end{matrix}} = \theta_{ 0, i}^{\frac{1}{2}h}$ is sampled from the normal distribution $\mathcal{N}\big(0, 10\m I_d\big)$.
\end{definition}\medbreak
\begin{remark}
By the law of large numbers, the estimator (\ref{eq:strong_estimator}) converges as $n\rightarrow\infty$ to
\begin{align*}
S_{N} := \E\Big[\big\|\theta_{N}^h\begin{matrix}\\[-12pt]\end{matrix} - \theta_{2N}^{\frac{1}{2}h}\big\|^2\Big]^\frac{1}{2}\,,
\end{align*}
almost surely. Though for large $h$, $S_N$ may not be close to the L2 error $\m\E\big[\|\theta_N^h - \theta_{\m T}\|^2\big]^\frac{1}{2}\hspace*{-1mm}$.	
\end{remark}\medbreak

Using this estimator, we shall compare SORT against several prominent schemes. Whilst the literature is extensive, we chose the UBU splitting \cite{sanzserna2021ubu, chada2024uld}, randomized midpoint method \cite{shen2019midpoint}, OBABO splitting \cite{bussi2007obabo, rioudurand2023malt} and exponential Euler scheme \cite{cheng2018uld}.\medbreak

The results of the above numerical experiment are presented in Figure \ref{fig:ULD_convergence}.
and code for this experiment can be found at \href{https://github.com/james-m-foster/high-order-langevin}{github.com/james-m-foster/high-order-langevin}.\vspace{-2mm}

\begin{figure}[ht]
\centering
\includegraphics[width=\textwidth]{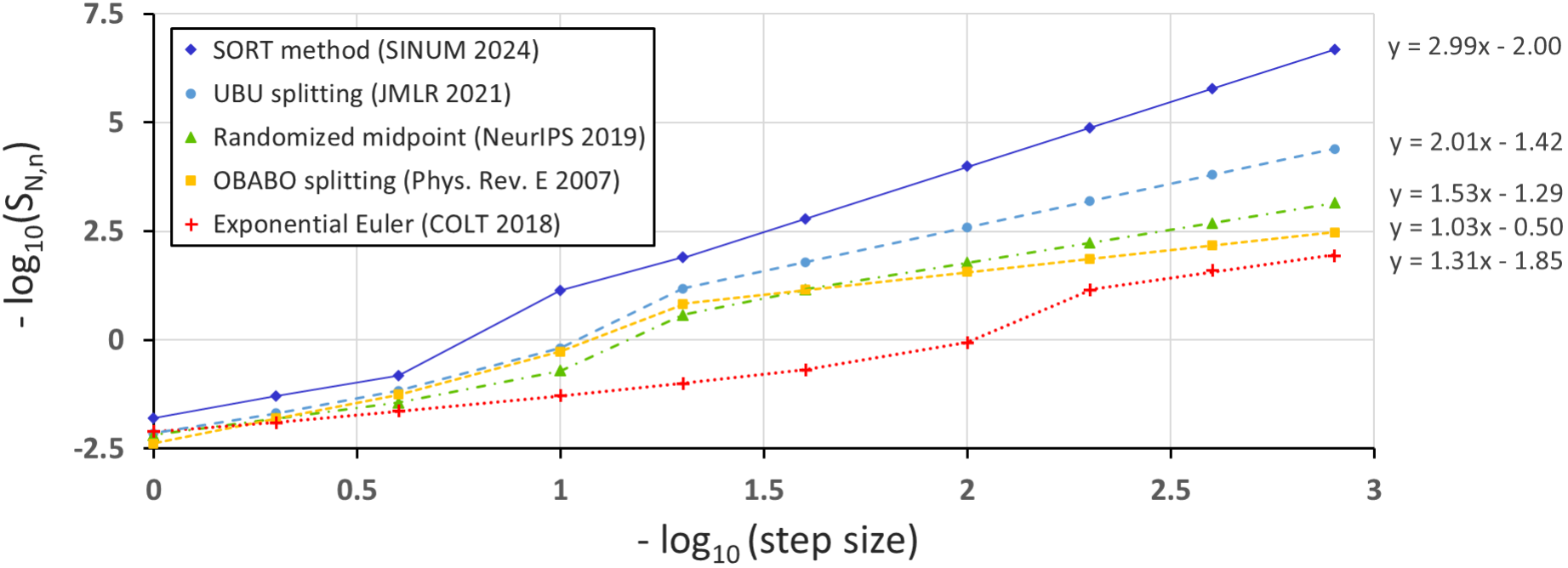}\vspace{-3mm}
\caption{L2 error estimated for (\ref{eq:uld2}) with $n = 1,000$ sample paths as a function of step size $h = \frac{T}{N}$.}\label{fig:ULD_convergence}
\end{figure}

From this graph, we see that SORT converges faster than its nearest competitor (the UBU splitting) and is an order of magnitude more accurate with a step size of $0.1$.
Moreover, its worth noting that this accelerated convergence was only made possible by generating the integrals in the log-signature of the Brownian motion $(W, H, K)_{t_k\m,\m t_{k+1}}$.\medbreak

To show the applicability of the SORT method to challenging sampling problems, \cite{jelincic2024vbt} considers the 10-dimensional ``Neal's funnel'' distribution \cite{neal2003funnel}, which is defined as
\begin{align}\label{eq:neal_funnel}
X & \sim\mathcal{N}(0,9),\mm Y \sim\mathcal{N}(0,e^X I_9),
\end{align}
where $(X,Y)\in\R\times\R^9$, and therefore corresponds to the scalar potential $f$ given by
\begin{align*}
f(x,y) = \frac{1}{2}x + \frac{1}{18}x^2 + \frac{1}{2}e^{-x}\|y\|^2.
\end{align*}
Neal's funnel is known to be a challenging distribution for MCMC algorithms as it has a narrow high density region when $X < 0$ and a wide low density region when $X > 0$. In particular, most MCMC samplers struggle to enter into the narrow ``funnel'' region.\medbreak

In \cite{jelincic2024vbt}, the following MCMC algorithms were tested on Neal's funnel distribution:
\begin{itemize}
\item The Euler-Maruyama and SORT methods with a constant step size of $h = \frac{1}{64}$.\vspace{0.5mm}
\item The SORT method with an adaptive step size. Here, steps are determined using the ``Proportional-Integral'' (PI) controller that is built into the ``Diffrax'' library \cite{kidger2022nde}. As recommended by \cite{ilie2015adaptive}, this PI controller used the parameters $K_P = 0.1, K_I = 0.3$.\vspace{0.5mm}
\item The No U-Turn Sampler (NUTS) \cite{hofttman2014nuts}, which is a state-of-the-art MCMC algorithm.
\end{itemize}\medbreak

For the above algorithms, samples were generated using 64 chains of 128 samples. The time between samples was $\Delta t = 2$ and an initial ``burn-in'' period of 16 iterations was used. When necessary, adaptive step sizes were shortened to produce the samples.
For the experiment's code, see \href{https://github.com/andyElking/Single-seed_BrownianMotion}{github.com/andyElking/Single-seed{\_\m}BrownianMotion}.

\begin{figure}[ht]
\centering
\includegraphics[width=\textwidth]{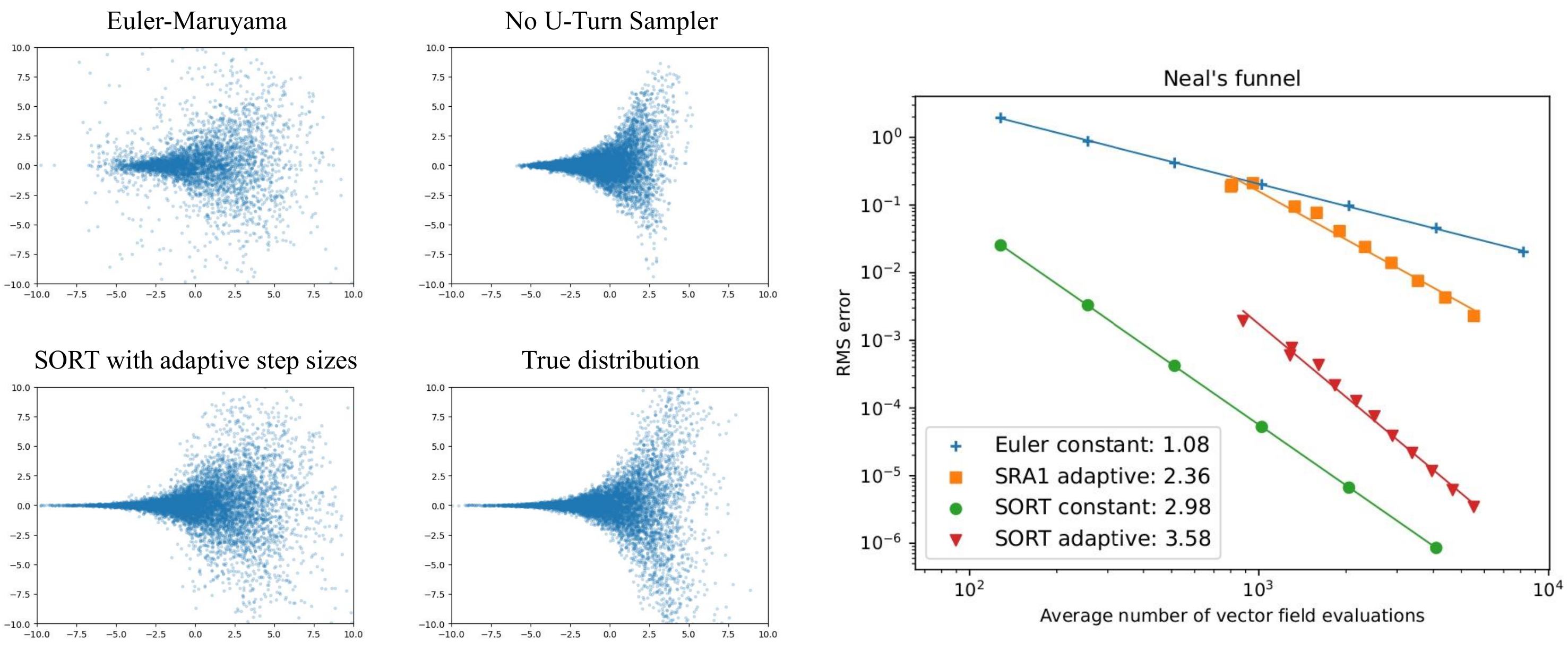}\vspace*{-2.5mm}
\caption{Left: Qualitative comparison of exact samples from Neal's funnel and approximate samples generated by different methods. Right: Graph showing estimated strong convergence rates for SORT, Euler-Maruyama and SRA1 -- which is a well-known high order stochastic Runge-Kutta method \cite{rossler2010srk}.}\label{fig:funnel_convergence}
\end{figure}\vspace*{-2.5mm}
To empirically compare the quality of samples produced by each MCMC algorithm, the authors of \cite{jelincic2024vbt} normalised the samples using the rescaling: $x^{\m \prime} = \frac{1}{3}x,\,\,\,y^{\m\prime} = e^{-\frac{1}{2}x} y$.
Since the true funnel distribution will produce rescaled samples $(x^{\m \prime}, y^{\m \prime})\sim\mathcal{N}(0, I_{10})$, the approximate samples were compared to $\mathcal{N}(0, I_{10})$ using the following error metrics:
\begin{itemize}
\item The maximum and average absolute error of the entries within the empirical mean (the correct mean is $0$) and covariance matrix (the correct covariance matrix is $I_{10}$).
\item The average Effective Sample Size (ESS) computed from the marginal distributions.
\item The average Kolmogorov-Smirnov (KS) test $p\m$-value computed from each marginal.
\end{itemize}

In addition, outliers with large magnitudes ($\|\cdot\| > 100$) were removed from the samples generated by the SDE solvers with constant step sizes. On the other hand, the adaptive step size controller prevented all such outliers by reducing its step size.\vspace*{-1mm}

\begin{table}[ht]
\centering
\begin{tabular}{cccccccccc}
\toprule
\multirow{2}{*}{Method} & \multirow{2}{*}{\begin{tabular}{c}$\nabla f$ evaluations \\ per sample \end{tabular}} & \multicolumn{2}{c}{Mean error} & \multicolumn{2}{c}{Cov error} & KS & ESS\\
 & & Max & Avg & Max & Avg & Avg & Avg\\
\midrule\\[-8pt]
NUTS  &  700.2 & 0.79 & 0.44 & 0.56 & 0.16 &  0.044 & 0.021\\[3pt]
\multirow{2}{*}{\begin{tabular}{c}Euler-Maruyama \\ (constant step size)\end{tabular}} & \multirow{2}{*}{137.0} & \multirow{2}{*}{\textbf{0.12}} & \multirow{2}{*}{0.069} & \multirow{2}{*}{150} & \multirow{2}{*}{4.9} & \multirow{2}{*}{0.0} & \multirow{2}{*}{\textbf{0.19}} \\ \\[3pt]
\multirow{2}{*}{\begin{tabular}{c}SORT\\ (constant step size) \end{tabular}}  &  \multirow{2}{*}{274.0} & \multirow{2}{*}{0.20} & \multirow{2}{*}{\textbf{0.034}} & \multirow{2}{*}{42} & \multirow{2}{*}{1.7} & \multirow{2}{*}{0.0} & \multirow{2}{*}{0.17}\\ \\[3pt]
\multirow{2}{*}{\begin{tabular}{c}SORT\\ (adaptive step size) \end{tabular}}  &  \multirow{2}{*}{\textbf{59.9}} & \multirow{2}{*}{0.20} & \multirow{2}{*}{0.047} & \multirow{2}{*}{\textbf{0.34}} & \multirow{2}{*}{\textbf{0.029}} & \multirow{2}{*}{\textbf{0.17}} & \multirow{2}{*}{0.13}\\ \\[1pt]
\bottomrule
\end{tabular}\medbreak
\caption{A quantitative comparison of the No U-Turn Sampler (NUTS) against Langevin-based MCMC algorithms (obtain from different discretizations of (\ref{eq:uld2})) for sampling from Neal's funnel. The results were averaged over 5 independent runs. We refer the reader to \cite{jelincic2024vbt} for further details.}
\label{tab:neals_funnel_results}
\end{table}\vspace*{-2.5mm}

From Table \ref{tab:neals_funnel_results} and Figure \ref{fig:funnel_convergence}, we see that the adaptive SORT method (which uses Gaussian integrals in the Brownian log-signature) gives the best overall performance. However, as an MCMC algorithm, it is worth noting that the SORT method should be viewed differently from NUTS, which is Metropolis-adjusted and thus has the correct stationary distribution. On the other hand, the SORT method can explore the parameter space faster as well as take smaller step sizes in the more difficult regions.
Moreover, the additional bias introduced by SORT comes from its discretization error, which quickly decreases at a third order rate (as is evident in both Figures \ref{fig:ULD_convergence} and \ref{fig:funnel_convergence}).\medbreak

In addition to these results, further experiments comparing a SORT-based method to NUTS \cite{hofttman2014nuts} for performing Bayesian logistic regression on real-world datasets can be found in \cite{jelincic2024vbt} and \cite{scott2025uld}. Whilst SORT-based Langevin MCMC is a topic of future work, we can already note that it gives a clear application of space-time-time L\'{e}vy area (\ref{eq:space_time_time}).

\subsection{Space-space L\'{e}vy area examples}\label{sect:ss_examples}

In this section, we present the experiments from \cite{jelincic2025levygan, foster2024adaptive} to demonstrate the L\'{e}vy area approximations considered in section \ref{sect:space_space}.
Firstly, in \cite{jelincic2025levygan}, several metrics were computed to compare the distributions of Brownian L\'{e}vy area and the weak approximation (\ref{eq:weak_levy_area}).

\subsubsection{L\'{e}vy area of a $d$-dimensional Brownian motion (with $d\leq 4$)}\label{sect:levy_experiments}

In the first experiment of \cite{jelincic2025levygan}, the authors compared four methods for approximately generating the L\'{e}vy area of Brownian motion over $[0,1]$ given the path increment $W_1\m$.
\begin{itemize}
\item The weak L\'{e}vy area approximation based on matching moments (Definition \ref{def:weak_levy_area}). \vspace{0.5mm}
\item The ``L\'{e}vyGAN'' model -- trained according to the methodology proposed in \cite{jelincic2025levygan}. \vspace{0.5mm}
\item The L\'{e}vy area approximation of Davie \cite{davie2014area}, which is based on matching moments. \vspace{0.5mm}
\item The truncated Fourier series approximation of L\'{e}vy area implemented in Julia \cite{kastner2023julia}.
\end{itemize}

As the ``true'' samples, $2^{20}$ independent tuples $(W_1^1\m, W_1^2\m, A_{0,1}^{12})$ of increments and finely discretized L\'{e}vy areas were generated using the Fourier series approximation.
Following this, $2^{20}$ independent samples were generated using each method ($d=2$) and the empirical $2$-Wasserstein error was computed between the two sets of samples.
The truncation level of the Fourier series was chosen so that it produced samples that gave a comparable $2$-Wasserstein error to the weak L\'{e}vy area and L\'{e}vyGAN samples. 
This generation was done using a GPU (NVIDIA Quadro RTX 8000), and we refer the reader to the appendix of \cite{jelincic2025levygan} for further details regarding this numerical experiment.
In particular, the associated code is available at \href{https://github.com/andyElking/LevyGAN}{github.com/andyElking/LevyGAN}.\vspace*{-1.5mm}
\begin{table}[ht]
    \begin{tabular}{@{}ccccc@{}}
\toprule
\multirow{3}{*}{\textbf{Test Metrics}\hspace*{-2mm}} & \multirow{3}{*}{\begin{tabular}{c}Weak L\'{e}vy area \\ approximation \\[1pt] (Definition \ref{def:weak_levy_area}) \end{tabular}} & \hspace{-2mm}\multirow{3}{*}{\begin{tabular}{c}\\[-6pt] L\'{e}vyGAN \\ \cite{jelincic2025levygan} \end{tabular}} & \hspace{-2mm}\multirow{3}{*}{\begin{tabular}{c}Davie's L\'{e}vy area \\ approximation \\[1pt] \cite{davie2014area} \end{tabular}} & \hspace{-3mm}\multirow{3}{*}{\begin{tabular}{c}Fourier series \\ approximation\\[1pt] \cite{mrongowius2022area, kastner2023julia}\end{tabular}} \\[2pt] \\ \\[-1pt] \midrule    
   \multirow{2}{*}{\begin{tabular}{c}Computational \\ time (s) \end{tabular}} & \multirow{2}{*}{$0.0071$} & \hspace{-2mm}\multirow{2}{*}{$0.019$} & \hspace{-3mm}\multirow{2}{*}{$0.002$} & \hspace{-2mm}\multirow{2}{*}{$3.1$} \\ \\[3pt]
   \multirow{3}{*}{\begin{tabular}{c} Marginal\\ $2$-Wasserstein\\ error ($10^{-2}$) \end{tabular}} & \multirow{3}{*}{$.254\pm .010$} & \hspace{-2mm}\multirow{3}{*}{$\mathbf{.246 \pm .013}$} & \hspace{-2mm}\multirow{3}{*}{$2.03\pm.013$} & \hspace{-3mm}\multirow{3}{*}{$.27 \pm 0.008$}\\ \\ \\
     \bottomrule
    \end{tabular}\smallbreak
    \caption{Computational cost and marginal $2$-Wasserstein error for L\'{e}vy area approximations.}\label{tab:levy_area_1}
\end{table}

From Table \ref{tab:levy_area_1}, we see the weak approximation (\ref{eq:weak_levy_area}) achieves similar accuracy as the L\'{e}vyGAN model (in under half the time) and outperforms other methods when $d=2$. However, in the $d=2$ case, there is an algorithm for exactly generating L\'{e}vy area \cite{gaines1994area}. Therefore, further numerical experiments were conducted in \cite{jelincic2025levygan} for testing the joint distribution of the approximate L\'{e}vy areas in multidimensional cases where $d > 2$. 
These test metrics consist of Maximum Mean Discrepancy (MMD) distances with two different kernels (polynomial and Gaussian), an Empirical Unitary Characteristic Function Distance (EUCFD) and computing the largest error in the fourth moments. As before, we refer the reader to \cite{jelincic2025levygan} and its appendix for further experimental details.\smallbreak
\begin{table}[ht]\vspace*{-2mm}
\begin{tabular}{@{}ccccc@{}}
\toprule
\hspace{2mm}\multirow{3}{*}{\textbf{Dim}} & \multirow{3}{*}{\textbf{Test Metrics}} & \multirow{3}{*}{\begin{tabular}{c}Weak L\'{e}vy area \\ approximation \\[1pt] (Definition \ref{def:weak_levy_area}) \end{tabular}} & \multirow{3}{*}{\begin{tabular}{c}\\[-6pt] L\'{e}vyGAN \\ \cite{jelincic2025levygan} \end{tabular}} & \hspace{-2mm}\multirow{3}{*}{\begin{tabular}{c}Davie's L\'{e}vy area \\ approximation \\[1pt] \cite{davie2014area} \end{tabular}} \\[2pt] \\ \\ \midrule
\hspace{2mm}\multirow{4}{*}{2}
 & Max fourth moment error  & $\mathbf{.002\pm .002}$ & $.004\pm .002$ & \hspace{-2mm}$.042\pm.001$  \\[2pt]
 & Polynomial MMD ($10^{-5}$) & $.654 \pm .131$ & $\mathbf{.341 \pm .070}$ & \hspace{-2mm}$.646\pm.188$ \\[2pt]
  & Gaussian MMD ($10^{-6}$)& $\mathbf{1.44 \pm .128}$  & $1.47\pm.125$ & \hspace{-2mm}$34.6\pm.683$ \\
  [2pt]
  & EUCFD ($10^{-2}$) & $1.92 \pm .113$ & $\mathbf{1.52\pm.213}$ & \hspace{-2mm}$10.1\pm.851$ \\
 \midrule
\hspace{2mm}\multirow{4}{*}{3} 
 & Max fourth moment error & $\mathbf{.004\pm.002}$ & $\mathbf{.004 \pm .002}$ & \hspace{-2mm}$.043\pm.001$ \\[2pt]
 & Polynomial MMD ($10^{-5}$) & $2.30 \pm .732$ & $\mathbf{2.18 \pm .568}$ & \hspace{-2mm}$2.26 \pm .773$ \\[2pt]
  & Gaussian MMD ($10^{-6}$) & $\mathbf{1.84 \pm .001}$ & $1.87 \pm .002$ & \hspace{-2mm}$16.3\pm.001$ \\
  [2pt]
  & EUCFD ($10^{-2}$) & $2.03 \pm .034$ & $\mathbf{1.88\pm.063}$ & \hspace{-2mm}$18.5\pm1.11$ \\
 \midrule
\hspace{2mm}\multirow{4}{*}{4} 
 & Max fourth moment error & $.006 \pm .002$ & $\mathbf{.004 \pm .000}$ & \hspace{-2mm}$.043\pm.002$ \\[2pt]
 & Polynomial MMD ($10^{-5}$) & $4.65 \pm 1.31$ & $\mathbf{4.04 \pm .436}$ & \hspace{-2mm}$5.62 \pm .808$ \\[2pt]
  & Gaussian MMD ($10^{-6}$) & $\mathbf{1.90\pm.001}$ & $\mathbf{1.90 \pm .001}$ & \hspace{-2mm}$263\pm .003$ \\
  [2pt]
  & EUCFD ($10^{-2}$)& $2.03 \pm .036$ & $\mathbf{1.92\pm.026}$ & \hspace{-2mm}$17.5\pm.483$ \\
 \bottomrule
\end{tabular}\medbreak
\caption{Fourth moment, MMD and characteristic function based test metrics across different methods for approximating the L\'{e}vy area of a $d$-dimensional Brownian motion with $d\leq 4$.}\label{tab:levy_area_2}
\end{table}\vspace*{-2mm}
We see from Table \ref{tab:levy_area_2} that the proposed weak approximation is comparable to the L\'{e}vyGAN model but significantly more accurate than Davie's approximation in the fourth moment, Gaussian MMD and EUCFD metrics. However, unlike L\'{e}vyGAN, the proposed weak approximation (\ref{eq:weak_levy_area}) matches several moments of L\'{e}vy area exactly (Theorem \ref{thm:match_moments}) and therefore may be more appealing for the weak simulation of SDEs.
Moreover, since the L\'{e}vyGAN model is about twice as expensive to evaluate as the weak approximation, it would have been reasonable to use a ``two-step'' approximation,
\begin{align}\label{eq:chen_error}
\overset{\,\approx}{A}_{0,1} := \widetilde{A}_{0,\m\frac{1}{2}} + \frac{1}{2}\big(W_{\frac{1}{2}}\otimes W_{\frac{1}{2}, 1} - W_{\frac{1}{2}, 1}\otimes W_1\big) + \widetilde{A}_{\frac{1}{2}, 1}\m,
\end{align}
which, by Brownian scaling, we would expect to have $\frac{\sqrt{2}}{2}\approx 70\%$ of the error as $\widetilde{A}_{0,1}\m$.\medbreak

Although the proposed weak approximation and L\'{e}vyGAN are similar in accuracy, due to the difference in evaluation speed, we believe the former performs slightly better (at least when $d\leq 4$). In any case, it is clear that both the proposed moment-matching approximation (\ref{eq:weak_levy_area}) and the machine learning based approach \cite{jelincic2025levygan} are significantly more efficient than the traditional Fourier series approximation of Brownian L\'{e}vy area.

\subsubsection{Heston stochastic volatility model}

Following the experiments summarised in section \ref{sect:levy_experiments}, the authors of \cite{jelincic2025levygan} used the moment-based and L\'{e}vyGAN approximations to simulate the (log-)Heston model \cite{heston1993model}.
The Heston model is a popular stochastic volatility model in mathematical finance due to the semi-analytic formula it gives for European call options \cite{crisostomo2015heston}. It is given by
\begin{align}\label{eq:heston}
dS_t & = r S_t\m dt + \sqrt{V_t}\m S_t\m dW_t^1,\\
dV_t & = \kappa (\theta - \nu_t)\m dt + \sigma\sqrt{V_t}\m dW_t^2,\nonumber
\end{align}
where $r\in\R$ is the discount rate and $\kappa, \theta, \sigma > 0$ satisfy the Feller condition $2\kappa\theta \geq \sigma^2$. In the experiment, we will compute the value of a call option with strike price $K > 0$,
\begin{align}\label{eq:option_price}
\E\big[\varphi(S_T)\big] \mm\text{where}\mm \varphi(S) := e^{-rT}(S - K)^{+},
\end{align}
using Multilevel Monte Carlo (MLMC) simulation, introduced by Giles in \cite{giles2008mlmc}, for variance reduction. The key idea of MLMC is to write $\E\big[\varphi(S_T)\big]$ as a telescoping sum,
\begin{align}\label{eq:mlmc}
\E\big[\varphi(S_T^L)\big] = \E\big[\varphi(S_T^{\m 0})\big] + \sum_{\ell = 1}^{L}  \E\big[\varphi(S_T^{\m\ell}) - \varphi(S_T^{\m\ell - 1})\big],
\end{align}
where $S^{\m\ell}$ denotes a numerical solution of the SDE whose accuracy increases with $\ell$.
Typically, $S^{\m\ell}$ will be a discretization obtained using a step size proportional to $2^{-\ell}$.
The advantage of MLMC is that, when estimating $\E\big[\varphi(S_T^{\m\ell}) - \varphi(S_T^{\m\ell - 1})\big]$ at each ``level'', one can simulate the numerical solutions $(S_T^{\m\ell}\m, S_T^{\m\ell - 1})$ using the same underlying noise. This reduces the variance of the Monte Carlo estimator for $\E\big[\varphi(S_T^{\m\ell}) - \varphi(S_T^{\m\ell - 1})\big]$ and thus results in an MLMC estimator that requires fewer finely discretized sample paths.
We refer the reader to \cite{giles2015mlmc} for a detailed account of MLMC's theory and applications.\medbreak

In the experiment, we shall estimate the call option (\ref{eq:option_price}) using MLMC with different SDE solvers -- including a high order method that will use L\'{e}vy area approximations. 
Since the Heston model produces an efficient semi-analytic formula for computing call option prices, it is possible to directly estimate weak errors for each numerical scheme.
We summarise the methods below, but would refer the reader to \cite{jelincic2025levygan} for full details.
\begin{itemize}
\item Multilevel Monte Carlo with the no-area Milstein method. Due to the slow $O(\sqrt{h}\m)$ strong convergence of the scheme, we would only expect the variance of the $\ell$-th level Monte Carlo estimator for $\E\big[\varphi(S_T^{\m\ell}) - \varphi(S_T^{\m\ell - 1})\big]$ to be $O(h)$, where $h\propto 2^{-\ell}$.\medbreak
\item Antithetic Multilevel Monte Carlo with the no-area Milstein's method \cite{giles2014anti} and Strang splitting. By coupling $S_T^{\m\ell - 1}$ with $S_T^{\m\ell}$ and an antithetic version of $S_T^{\m\ell}$, we expect the variance of the $\ell$-th level estimator for $\E\big[\varphi(S_T^{\m\ell}) - \varphi(S_T^{\m\ell - 1})\big]$ to be $O(h^2)$.\medbreak
\item Multilevel Monte Carlo with the ``Strang-Log-ODE'' method introduced in \cite{foster2024splittingarxiv}. When used with true Brownian L\'{e}vy area, we expect this method will achieve first order strong convergence (thus giving $O(h^2)$ MLMC variance reduction) and second order weak convergence. However, we will instead use the approximate L\'{e}vy areas:
\begin{itemize}
\item The weak L\'{e}vy approximation given by Definition \ref{def:weak_levy_area} (``Strang-F'' in Figure \ref{fig:heston_convergence})\vspace{0.5mm}
\item The deep-learning-based L\'{e}vyGAN model from \cite{jelincic2025levygan} (``Strang-Net'' in Figure \ref{fig:heston_convergence})\vspace{0.5mm}
\item Talay's weak approximation of $A_{s,t}^{ij}$ using $\bar{A}_{s,t}\in\{\pm \frac{1}{2}h\}$ (``Strang-T'' in Figure \ref{fig:heston_convergence})\vspace{0.25mm}
\item No L\'{e}vy area, i.e.~we replace $A_{s,t}^{ij}$ by $\E\big[A_{s,t}^{ij}\m|\m W_{s,t}\big] = 0$ (``Strang-NA'' in Figure \ref{fig:heston_convergence})
\end{itemize}
\end{itemize}

However, for all these approximations, the MLMC estimator that we use introduces a fixed bias due to the difference in distribution between $\widetilde{A}$ and $\overset{\,\approx}{A}$ (see equation (\ref{eq:chen_error})).\medbreak

Similar to \cite{algerbi2016mlmc, iguchi2025anti}, we expect that these approximate L\'{e}vy areas could be used with antithetic MLMC to achieve both high order weak convergence and variance reduction.
This would likely involve both $S_T^{\m\ell - 1}$ with $S_T^{\m\ell}$ having ``antithetic twins'' obtained by reflecting the sign of their respective L\'{e}vy area approximations (i.e.~four simulations based on the same Brownian path). However, we leave this as a topic of future work.\medbreak

In the experiment, we will simulate the Heston model (\ref{eq:heston}) using the methods summarised above and estimate the option price (\ref{eq:option_price}) with the following parameters:\medbreak
$\displaystyle r = 0.1,\hspace{3mm} K = 20,\hspace{3mm} \kappa = 2,\hspace{3mm} \theta = 0.1,\hspace{3mm} \sigma = 0.5,\hspace{3mm} S_0 = 20,\hspace{3mm} V_0 = 0.4, \hspace{3mm} T = 1.$\medbreak

The number of samples taken at each level are determined via the approach of \cite{giles2008mlmc}. 
In Figure \ref{fig:heston_convergence}, we report the bias of the MLMC estimators (as a function of the level $L$) as well as the variance of the $\ell$-th level Monte Carlo estimator for $\E\big[\varphi(S_T^{\m\ell}) - \varphi(S_T^{\m\ell - 1})\big]$.\vspace{-1.5mm}
\begin{figure}[H]
\centering
\includegraphics[width=\textwidth]{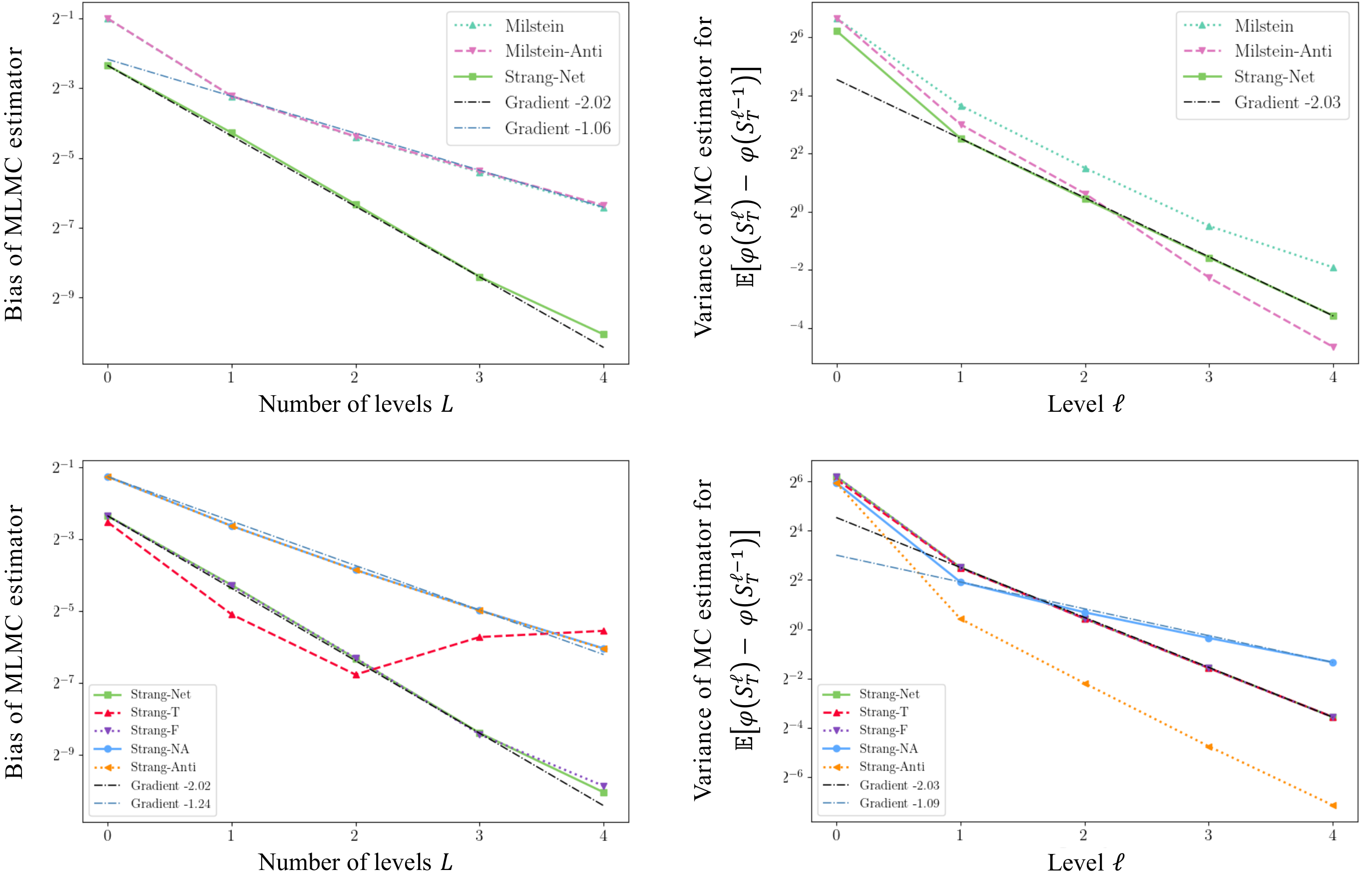}\vspace*{-1mm}
\caption{Top left: Bias of MLMC estimator for the call option (\ref{eq:option_price}) with the no-area Milstein, antithetic no-area Milstein and the Strang-Log-ODE method (using the L\'{e}vyGAN approximation of L\'{e}vy area).
Bottom left: Bias of MLMC estimator with the Strang-Log-ODE method using the approximations of L\'{e}vy area summarised above -- including the weak approximation (\ref{eq:weak_levy_area}), L\'{e}vyGAN model and no area.
Right: Variances of the Monte Carlo estimators at each level $\ell$ for estimating $\E\big[\varphi(S_T^{\m\ell}) - \varphi(S_T^{\m\ell - 1})\big]$.}\label{fig:heston_convergence}
\end{figure}

In Figure \ref{fig:heston_convergence}, we see that only the Strang-Log-ODE method (with approximate L\'{e}vy areas) can achieve both second order weak convergence and variance reduction.
Though, as previously noted, the MLMC estimator that we use also introduces a fixed bias that will tend to zero as the error in the L\'{e}vy area approximation tends to zero.
In the case of Talay's weak approximation, which simply replaces each $A_{s,t}^{ij}$ with a Rademacher random variable, this limits the accuracy of the MLMC estimator to $2^{-7}$.
However, this is not the case for the proposed weak approximation of L\'{e}vy area (\ref{eq:weak_levy_area}) and L\'{e}vyGAN model, which introduce a much smaller bias of order $2^{-11}$ at each level.
Whilst this additional bias could be a potential problem in applications, it is always possible to reduce this bias by ``Chen-combining'' L\'{e}vy area according to (\ref{eq:chen_error}), see \cite{jelincic2025levygan}.\medbreak

Finally, to conclude this section, we will compare the computational times required for a given accuracy by the Milstein and L\'{e}vyGAN-based Strang log-ODE methods.

\begin{table}[ht]
    \centering
    \begin{tabular}{@{}cccccccc@{}}
    \toprule
    RMSE & $0.1$  & $0.0441$ & $0.0129$  & $0.0086$ &$ 0.0057$& $0.0038$ &$ 0.0025$ \\ \midrule
   Milstein (s) &$ \textbf{0.0097}$& $0.0256$& $0.376$&$ 1.03$&$ 2.86$& $8.63$ & $23.6$ \\[2pt]
   L\'{e}vyGAN (s) & $0.0102$& $\textbf{0.0128}$& $\textbf{0.142}$& $\textbf{0.311}$&$ \textbf{0.806}$& $\textbf{2.25}$& $\textbf{5.83}$\\
     \bottomrule
    \end{tabular}\medbreak
    \caption{The average computational time (taken across $25$ runs) for the Milstein and Strang-Log-ODE methods to reach a given root mean-squared error (RMSE). In each run, the methodology of \cite[Section 5]{giles2008mlmc} is used to determine the number of levels in the MLMC estimator and the number of samples generated at each level. All of the random variables were generated using PyTorch on a GPU, whereas the numerical methods were implemented using NumPy on a CPU.}\label{tab:heston_times}
\end{table}\vspace{-4mm}

In Table \ref{tab:heston_times}, we see that the Strang-Log-ODE method with the L\'{e}vyGAN model achieves root mean-squared errors smaller than $0.1$ significantly faster than Milstein.
Therefore, based on the computational times in Table \ref{tab:levy_area_1} and the results in Figure \ref{fig:heston_convergence}, we would also expect the proposed weak approximation (\ref{eq:weak_levy_area}) to outperform Milstein.\medbreak

As a topic of future work, we would like to incorporate the new weak L\'{e}vy area approximation into an antithetic Multilevel Monte Carlo framework similar to \cite{algerbi2016mlmc, iguchi2025anti}.
This would then have the advantages of second order weak convergence and variance reduction without introducing a fixed bias related to the L\'{e}vy area approximation.
With this goal, we could also consider MLMC with the following simulations per level:\vspace{-1mm}
\begin{table}[ht]
\centering
\begin{tabular}{ccccc}
\toprule
Simulation &   \multicolumn{2}{c}{Approximate L\'{e}vy areas} \\
\midrule\\[-8pt]
Fine (\hspace{0.25mm}step size $ = \frac{1}{2}h$)  & \hspace{5mm} $X_1\otimes W_{s,u} - W_{s,u}\otimes X_1$ & $X_2\otimes W_{u,t} - W_{u,t}\otimes X_2$ \\[5pt]
Antithetic twin  & \hspace{5mm} $X_2\otimes W_{s,u} - W_{s,u}\otimes X_2$ & $X_1\otimes W_{u,t} - W_{u,t}\otimes X_1$ \\[2pt]
\midrule\\[-8pt]
Coarse (\hspace{0.25mm}step size $ = h$) &   \multicolumn{2}{c}{$\big(\frac{1}{4}\big(W_{s,u} - W_{u,t}\big) + X_1\big)\otimes W_{s,t} - W_{s,t}\otimes\big(\frac{1}{4}\big(W_{s,u} - W_{u,t}\big) + X_1\big)$} \\[5pt]
Antithetic twin & \multicolumn{2}{c}{$\big(\frac{1}{4}\big(W_{s,u} - W_{u,t}\big) + X_2\big)\otimes W_{s,t} - W_{s,t}\otimes\big(\frac{1}{4}\big(W_{s,u} - W_{u,t}\big) + X_2\big)$}\\[2pt]
\bottomrule
\end{tabular}\medbreak
\caption{Brief summary of a new antithetic MLMC estimator where all four simulations use the same Brownian path but different L\'{e}vy area approximations. Here, $X_1\m, X_2\sim\mathcal{N}\big(0,\frac{1}{16}hI_d\big)$ are independent and give approximate L\'{e}vy areas with the correct mean and covariance matrix.\\ In the above, $u := \frac{1}{2}(s+t)$ denotes the midpoint of the interval $[s,t]$ (which has length $h = t - s$).}
\label{tab:future_mlmc}
\end{table}

\subsubsection{SABR stochastic volatility model}

In this section, we will consider the SABR\footnote{Stochastic \textbf{A}lpha-\textbf{B}eta-\textbf{R}ho} model -- which is another popular stochastic volatility model used in mathematical finance \cite{hagan2002SABR, cai2017SABR, leitao2017SABR, cui2018SABR}. It is given by the following SDE,
\begin{align}\label{eq:SABR}
dS_t & = \sqrt{1-\rho^2}\m \sigma_t (S_t)^\beta  dW_t^1 + \rho\m \sigma_t (S_t)^\beta  dW_t^2,\\[1pt]
d\sigma_t & = \alpha\m\sigma_t\m dW_t^2,\nonumber
\end{align}
where $\alpha,\beta,\rho \geq 0$.
For simplicity, the authors of \cite{foster2024adaptive} set $\alpha = 1$ and $\beta = \rho = 0$, giving
\begin{align}
dS_t & = \sigma_t dW_t^1,\label{eq:SABR_ver2}\\[1pt]
d\sigma_t & = \sigma_t\m dW_t^2.\nonumber
\end{align}
In \cite{foster2024adaptive}, the simplified SABR model given by (\ref{eq:SABR_ver2}) is converted into Stratonovich form and then simulated using a variety of numerical methods and step size methodologies.
These are the Euler-Maruyama method and the following Stratonovich SDE solvers:\medbreak
\begin{definition}[Heun's method]
Consider the following Stratonovich SDE on $[0,T]$,\vspace{-2mm}
\begin{align}\label{eq:strat_sde}
dy_t = f(y_t)\m dt + \sum_{i=1}^d g_i(y_t) \circ dW_t^i\m.\\[-18pt]\nonumber
\end{align}
Then for a given number of steps $N$, we define a numerical solution $\{Y_n\}_{0\m\leq\m n\m\leq\m N}$ by\vspace{-2mm}
\begin{align*}
Y_{n+1} := Y_n + \frac{1}{2}\big(f(Y_n)+f(Z_{n+1})\big)h + \frac{1}{2}\sum_{i=1}^d \big(g_i(Y_n) + g_i(Z_{n+1})\big) W_{t_n,\m t_{n+1}}^i\m,\\[-20pt]\nonumber
\end{align*}
where $\,Z_{n+1} := Y_n + f(Y_n)h + \sum_{i=1}^d g_i(Y_n)  W_{t_n,\m t_{n+1}}^i\m$, $t_n := nh$, $h := \frac{T}{N}$ and $Y_0 := y_0\m$.
\end{definition}\vspace*{3mm}
\begin{definition}[Splitting Path Runge-Kutta method (SPaRK)] For a number of steps $N\geq 1$, we define a numerical solution $\{Y_n\}_{0\m\leq\m n\m\leq\m N}$ for the Stratonovich SDE (\ref{eq:strat_sde}) by\vspace{-2mm} 
\begin{align*}
Y_{n+1} := Y_n & + \bigg(\frac{3-\sqrt{3}}{6}\big(f(Y_n)+f(Z_{n+1})\big) + \frac{\sqrt{3}}{3} f(Y_{n+\frac{1}{2}})\bigg)h + \frac{\sqrt{3}}{3}\sum_{i=1}^d g_i(Y_{n+\frac{1}{2}}) W_{s,t}^i\\
& + \sum_{i=1}^d g_i(Y_n) \bigg(\frac{3-\sqrt{3}}{6} W_{s,t}^i + H_{s,t}^i\bigg) + \sum_{i=1}^d g_i(Z_{n+1})\bigg(\frac{3-\sqrt{3}}{6} W_{s,t}^i - H_{s,t}^i\bigg),
\end{align*}
where $t_n := nh$, $h := \frac{T}{N}$, $Y_0 := y_0\m$ and\vspace{-2mm}
\begin{align*}
Y_{n+\frac{1}{2}} & := Y_n + \frac{1}{2}f(Y_n)h + \sum_{i=1}^d g_i(Y_n) \bigg(\frac{1}{2}W_{s,t}^i + \sqrt{3} H_{s,t}^i\bigg),\\
Z_{n+1} & := Y_n + f(Y_{n+\frac{1}{2}}) h + \sum_{i=1}^d g_i(Y_{n+\frac{1}{2}}) W_{s,t}^i\m.
\end{align*}
\end{definition}

\begin{remark}\label{rmk:spark}
Taylor expanding the SPaRK method produces the following $O(h)$ terms,
\begin{align*}
&\sum_{i,\m j\m =\m 1}^d g_j^\prime(Y_k)g_i(Y_k)\bigg(\bigg(\frac{1}{2}W_{s,t}^i + \sqrt{3} H_{s,t}^i\bigg)\times\frac{\sqrt{3}}{3}W_{s,t}^j + W_{s,t}^i\times\bigg(\frac{3-\sqrt{3}}{6} W_{s,t}^j - H_{s,t}^j\bigg)\bigg)\\
&\mm = \sum_{i,\m j\m =\m 1}^d g_j^\prime(Y_k)g_i(Y_k)\bigg(\frac{1}{2}W_{s,t}^i W_{s,t}^j + H_{s,t}^i W_{s,t}^j - H_{s,t}^i W_{s,t}^j\bigg)\\[2pt]
&\mm = \sum_{i,\m j\m =\m 1}^d g_j^\prime(Y_k)g_i(Y_k)\bigg(\frac{1}{2}W_{s,t}^i W_{s,t}^j + \E\big[A_{s,t}^{ij}\,|\, W_{s,t}\m, H_{s,t}\big]\bigg),
\end{align*}
where the last line follows by the conditional expectation of L\'{e}vy area given by (\ref{eq:levy_area_cond_exp}). Hence, the SPaRK method will be asymptotically  more accurate than Heun's method. For details about the improved asymptotic accuracy of SPaRK, see \cite[Appendix A]{foster2024adaptive}.
\end{remark}\medbreak

In addition, the experiment of \cite{foster2024adaptive} also investigates different step size strategies. These are constant step sizes (that is, $h = \frac{T}{N}$) and the following two variable step sizes:\medbreak

\begin{definition}[Previsible variable step sizes]
Suppose we approximate the solution of (\ref{eq:SABR_ver2}) at time $t_n\m$ by $(\widetilde{S}_{t_n}\m, \widetilde{\sigma}_{t_n})$. Then we can define the next step size as $h(\widetilde{\sigma}_{t_n})$ with
\begin{align}\label{eq:magical_step_size}
h(\sigma) := \log\big(1 + C\sigma^{-2}\big),
\end{align}
where $C > 0$ is a user-specified constant. We would expect that decreasing the value of $C$ will increase the number of steps (and thus improve the accuracy of the simulation).
\end{definition}\medbreak
\begin{remark}
The variable step size (\ref{eq:magical_step_size}) is specific to the simplified SABR model (\ref{eq:SABR_ver2}) and based on the local error of Heun's method, which is computable by It\^{o}'s isometry as
\begin{align*}
\E\bigg[\bigg(S_{t+h} - \bigg(S_t+\frac{1}{2}\sigma_t\big(1 + e^{-\frac{1}{2}h + (W_{t+h}^2-W_t^2)}\big)\big(W_{t+h}^1 - W_t^1\big)\bigg)\bigg)^2\,\bigg] = \frac{1}{4}\sigma_t^2 h (e^h - 1)\m.
\end{align*} 
\end{remark}\medbreak

\begin{definition}[Non-previsible variable step sizes using an adaptive PI controller {\cite{ilie2015adaptive}}]
Let $\mathcal{W}_n := \{W_{t_n, t_{n+1}}\m, H_{t_n, t_{n+1}}\}$ be Brownian increments and space-time L\'{e}vy areas. Then we define a modified ``Proportional-Integral'' (PI) adaptive step size controller, which does not ``skip'' over times where the Brownian motion was previously sampled:\label{def:pi_controller}
\begin{align*}
h_{n+1} & := \min\big(\{\widetilde{h}_{n+1}\}\hspace{-0.125mm}\cup\hspace{-0.125mm}\{h > 0 : t_{n+1} + h\text{ corresponds to a previously rejected time}\}\big),\\[3pt]
\widetilde{h}_{n+1} & := h_n\bigg(\mathrm{Fac}_{\max}\hspace{-0.5mm}\wedge\hspace{-0.5mm}\bigg(\mathrm{Fac}_{\min}\hspace{-0.5mm}\vee\hspace{-0.5mm}\bigg(\frac{\mathrm{Fac}\cdot C}{e(Y_n, h_n, \mathcal{W}_n)}\bigg)^{K_I}\hspace{-0.5mm}\bigg(\frac{e(Y_{n-1}, h_{n-1}, \mathcal{W}_{n-1})}{e(Y_n, h_n, \mathcal{W}_n)}\bigg)^{K_P}\,\bigg)\bigg),
\end{align*}
where $\{\m\mathrm{Fac}_{\max}\m, \mathrm{Fac}_{\min}\}$ are the maximum and minimum factors $h_k$ can change by, $\mathrm{Fac}$ is a safety factor, $\{K_I, K_P\}$ are the ``integral'' and ``proportional'' coefficients, $C > 0$ is a  user-specified constant and $e(Y_n, h_n, \mathcal{W}_n)$ denotes a local error estimator. If $e(Y_{n+1}\m, h_{n+1}\m, \mathcal{W}_{n+1}) > 1$, then $h_{n+1}$ is rejected and $\frac{1}{2}h_{n+1}$ will be proposed instead. Otherwise, if the estimator satisfies $e(Y_{n+1}\m, h_{n+1}\m, \mathcal{W}_{n+1}) \leq 1$, then $h_{n+1}$ is accepted.\medbreak
In the experiment, we use the values recommended by \cite{ilie2015adaptive}, $(K_I, K_P) = (0.3, 0.1)$,
as well as $\m\mathrm{Fac}_{\max} = 10, \mathrm{Fac}_{\min} = 0.2$ and $\mathrm{Fac} = 0.9$ (the default values in Diffrax \cite{kidger2022nde}).
For the local error estimate $e(Y_n, h_n, \mathcal{W}_n)$, we compare the Heun and SPaRK methods to embedded Runge-Kutta approximations (and refer the reader to \cite{foster2024adaptive} for the details).
\end{definition}\medbreak

The results of the numerical experiment for these methods are given in Figure \ref{fig:SABR_experiment} and the associated Python code is available at \href{https://github.com/andyElking/Adaptive_SABR}{github.com/andyElking/Adaptive{\_\m}SABR}.
In addition, Euler, Heun and SPaRK are also implemented in the Diffrax package \cite{kidger2022nde}.\vspace*{-1mm}
\begin{figure}[ht]
\centering
\includegraphics[width=0.865\textwidth]{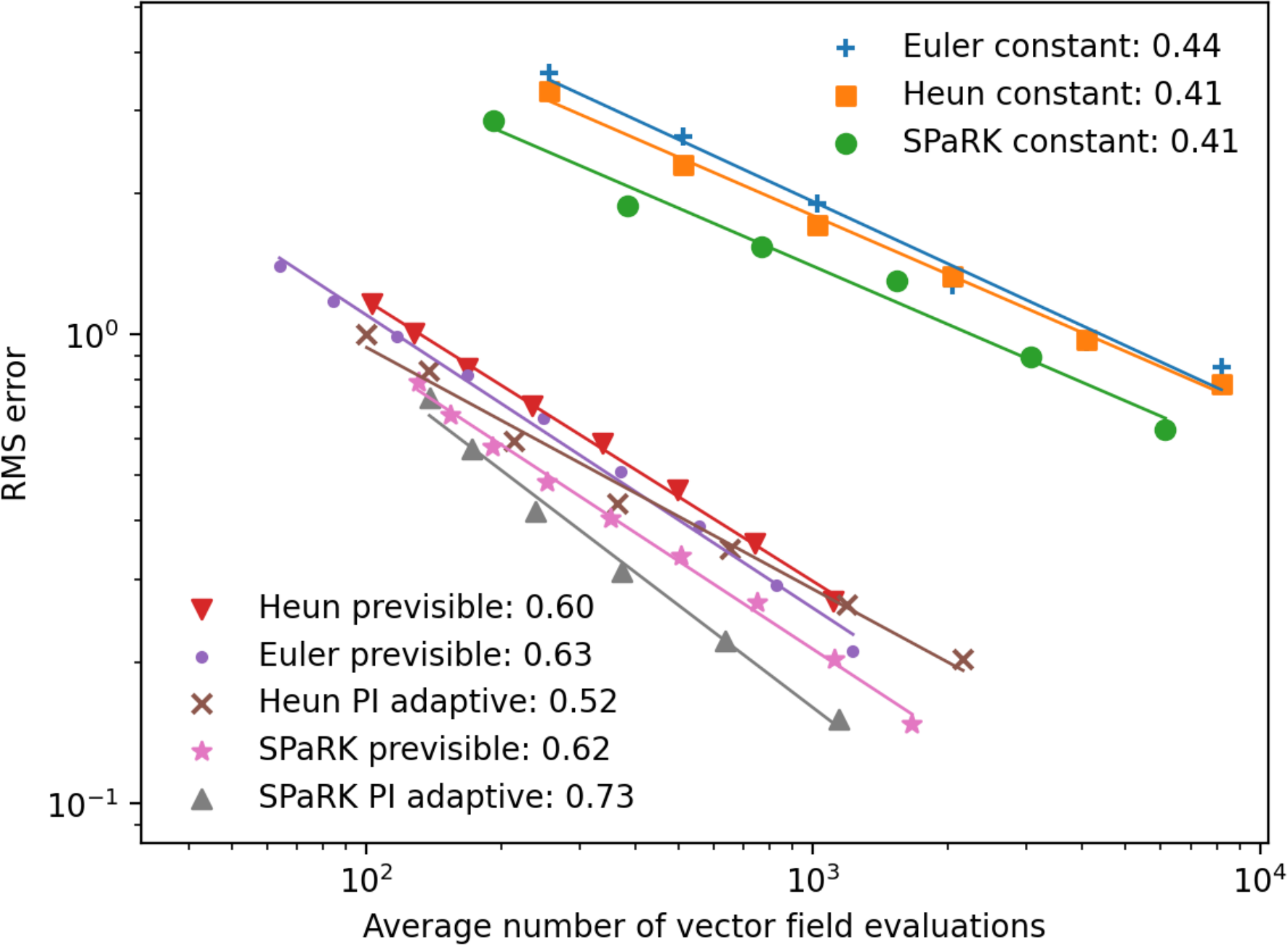}\medbreak
\caption{Strong convergence rates estimated for the different approximations of the SABR model (\ref{eq:SABR_ver2}), which was simulated with $S_0 = 0, \sigma_0 = 1$ and $T = 10$. Due to the counterexample presented in \cite{gaines1997variable}, the authors of \cite{foster2024adaptive} did not use the Euler-Maruyama method with non-previsible adaptive step sizes. The vector field evaluations used per step are: $1$ for Euler-Maruyama, $2$ for Heun and $3$ for SPaRK.}\label{fig:SABR_experiment}
\end{figure}\vspace*{-1mm}

In Figure \ref{fig:SABR_experiment}, we see that SPaRK gives the best accuracy with constant step sizes, which is expected due to the L\'{e}vy area approximation that it uses (see Remark \ref{rmk:spark}).
In fact, we expect a step of SPaRK to be as accurate as three steps of Heun's method.\medbreak
It is also worth noting that the previsible step size is very specific to the simplified SABR model (\ref{eq:SABR_ver2}) and therefore does not immediately extend to more general SDEs. On the other hand, the (modified) adaptive PI step size controller is generic and readily applicable to other problems -- especially since we used the default parameter values.
In any case, this experiment shows that methods can achieve an order of magnitude more accuracy when used with adaptive step sizes (both previsible and non-previsible).\medbreak

In conclusion, the examples in this section (Heston and SABR models) demonstrate that the moments and approximations of Brownian L\'{e}vy area detailed in section \ref{sect:space_space} can be used to improve SDE solvers -- whether they are for strong or weak simulation.

\subsection{Space-space-time L\'{e}vy area examples}

In this section, we will present experiments from \cite{foster2020poly, tubikanec2022igbm, foster2024splitting} to demonstrate the following (optimal) approximations of space-space-time L\'{e}vy area given in section \ref{sect:space_space_time},\vspace{-2mm}
\begin{align}
\E\big[L_{s,t}^{ii}\,\big|\, W_{s,t}\m, H_{s,t}\big] & = \frac{3}{5}h\big(H_{s,t}^i\big)^2 + \frac{1}{30}\m h^2,\label{eq:sst_igbm}\\[3pt]
\E\big[L_{s,t}^{ii}\,\big|\, W_{s,t}\m, H_{s,t}\m, n_{s,t}\big] & = \frac{3}{5}h\big(H_{s,t}^i\big)^2 + \frac{1}{30} h^2  - \frac{1}{8\sqrt{6\pi}}\m W_{s,t}^i n_{s,t}^i h^\frac{3}{2}.\label{eq:sst_splitting}\\[-20pt] \nonumber
\end{align}
Firstly, in \cite{foster2020poly}, equation (\ref{eq:sst_igbm}) was established (with $d=1$) and applied to a scalar SDE.

\subsubsection{Inhomogenous geometric Brownian motion (IGBM)}

In this section, we will consider different numerical methods for the scalar SDE (\ref{eq:IGBM}), which is often refereed to as ``Inhomogenous geometric Brownian motion'' (or IGBM).\vspace{-2mm}
\begin{align}\label{eq:IGBM}
dy_t = a(b-y_t)\m dt + \sigma y_t\, dW_t\m,\\[-20pt] \nonumber
\end{align}
where $a\geq 0$ and $b\in\R$ are the mean reversion parameters and $\sigma \geq 0$ is the volatility.
Since IGBM has multiplicative noise, it can easily be written in Stratonovich form as\vspace{-2mm}
\begin{align}\label{eq:IGBM_strat}
dy_t = \widetilde{a}(\widetilde{b}-y_t)\m dt + \sigma y_t \circ dW_t\m,\\[-20pt] \nonumber
\end{align}
where $\widetilde{a} := a + \frac{1}{2}\sigma^2$ and $\widetilde{b} := \frac{2ab}{2a+\sigma^2}\m$ denote the It\^{o}-Stratonovich adjusted  parameters.\medbreak

In the finance literature, IGBM was proposed as a one-factor short rate model \cite{capriotti2019IGBM}.
Moreover, as the solution of the SDE (\ref{eq:IGBM}) is both mean-reverting and non-negative, IGBM
is suitable for modelling interest rates, stochastic volatilities and hazard rates.
However, for the purpose of demonstrating the proposed approximation (\ref{eq:sst_igbm}), we note that IGBM is one of the simplest SDEs that has no known method of exact simulation.
In \cite{foster2020poly}, the authors investigate the convergence of the following methods for IGBM:\vspace{-2mm}
\begin{table}[ht]
    \centering
    \begin{tabular}{@{}cc@{}}
    \toprule
    Numerical method & Formula for each step \\ \midrule\\[-8pt]
   \multirow{2}{*}{Log-ODE} & \multirow{2}{*}{\begin{tabular}{c}
   $\displaystyle Y_{n+1} := e^{-\widetilde{a}h + \sigma W_{t_n, t_{n+1}}}Y_n + abh\big(1 - \sigma H_{t_n, t_{n+1}}\big)\frac{e^{-\widetilde{a}h + \sigma W_{t_n, t_{n+1}}} - 1}{-a\widetilde{h}+\sigma W_{t_n,t_{n+1}}}$  \\[9pt]
   $\displaystyle\left. +\right. ab\sigma^2 \bigg(\frac{3}{5}h H_{t_n, t_{n+1}}^2 + \frac{1}{30}h^2\bigg)\frac{e^{-\widetilde{a}h + \sigma W_{t_n, t_{n+1}}} - 1}{-a\widetilde{h}+\sigma W_{t_n,t_{n+1}}}\m.$   \end{tabular}}\\ \\ \\[26pt]
Parabola-ODE & \hspace{-11.75mm}$\displaystyle Y_{n+1} := e^{-\widetilde{a}h + \sigma W_{t_n, t_{n+1}}}\bigg(Y_n + ab\int_{t_n}^{t_{n+1}} e^{\widetilde{a}(s-t_n) - \sigma \wideparen{W}_{t_n, s}}\m ds\bigg).$ \\[7pt]
   \multirow{3}{*}{\begin{tabular}{c}\\[-10pt] Linear-ODE  \\[1pt]
   (or Wong-Zakai)\end{tabular}} & \hspace{-17mm}\multirow{3}{*}{$\displaystyle Y_{n+1} := e^{-\widetilde{a}h + \sigma W_{t_n, t_{n+1}}}Y_n + abh\bigg(\frac{e^{-\widetilde{a}h + \sigma W_{t_n, t_{n+1}}} - 1}{-a\widetilde{h}+\sigma W_{t_n,t_{n+1}}}\bigg).$}\\ \\[10pt]
    Milstein &  \hspace{-8.5mm}$\displaystyle Y_{n+1} := \Big(Y_n + \widetilde{a}(\widetilde{b}-Y_n)h + \sigma Y_n W_{t_n,t_{n+1}} + \frac{1}{2}\sigma^2 Y_n W_{t_n\m,t_{n+1}}^2\Big)^{+}.$ \\[8pt] 
    Euler-Maruyama  &  \hspace{-33.75mm}$\displaystyle Y_{n+1} := \Big(Y_n + a(b-Y_n)h + \sigma Y_n W_{t_n,t_{n+1}}\Big)^{+}.$ \\[4pt] 
     \bottomrule
    \end{tabular}\medbreak
    \caption{Numerical methods for inhomogenous geometric Brownian motion (IGBM) considered in \cite{foster2020poly}. Note that the second line of the log-ODE method clearly uses the estimator (\ref{eq:sst_igbm}) for $L_{s,t}\m$.}\label{tab:igbm_methods}
\end{table}

In \cite{foster2020poly}, IGBM is simulated with the above methods and the following parameters:
\begin{align*}
a = 0.1,\mm b = 0.04,\mm \sigma = 0.6,\mm y_0 = 0.06,\mm T = 5.
\end{align*}
For each numerical method, estimators for the strong and weak convergence rates are computed by Monte Carlo simulation (with 100,000 and 500,000 samples respectively). \medbreak

\begin{definition}[Strong and weak error estimators] Given a number of steps $N\geq 1$, let $Y_{N}$
be a numerical solution of (\ref{eq:IGBM}) computed at time $T$ using a fixed step size $h = \frac{T}{N}$.
We can define the following estimators for quantifying strong and weak convergence: \label{def:igbm_error}
\begin{align}
S_N & := \sqrt{\m\E\Big[\big(Y_N - Y_T^{fine}\big)^2\Big]}\m,\label{eq:igbm_strong}\\
E_N & := \Big|\m\E\Big[\big(Y_N - b\big)^{+}\Big] - \E\Big[\big(Y_T^{fine} - b\big)^{+}\Big]\Big|\m,\label{eq:igbm_weak}
\end{align}
where the above expectations are approximated by standard Monte-Carlo simulation
and $Y_{T}^{fine}$ is the numerical solution of (\ref{eq:IGBM_strat}) obtained at time $T$ using the log-ODE method
with a ``{\m}fine'' step size of $\min\left(\frac{h}{10}, \frac{T}{1000}\right)$. The fine step size is chosen so that the
$L^{2}(\mathbb{P})$ error between $Y_{T}^{fine}$ and the true solution $y_T$ is negligible compared to $S_{N}$.
We note that $Y_{N}$ and $Y_{T}^{fine}$ are both computed using the same Brownian sample paths.
\end{definition}\medbreak

The results of the numerical experiment for simulating IGBM are presented below.
Code for the experiment can be found at \href{https://github.com/james-m-foster/igbm-simulation}{github.com/james-m-foster/igbm-simulation}.

\begin{figure}[ht]
\centering
\includegraphics[width=\textwidth]{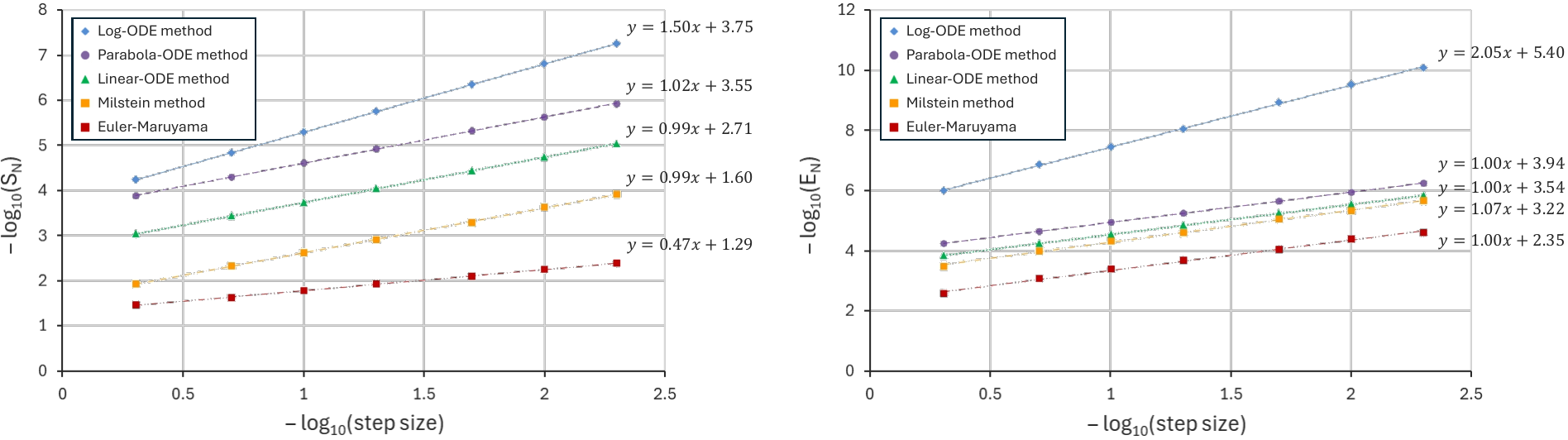}\medbreak
\caption{Strong and weak convergence rates estimated for the different IGBM numerical methods.}\label{fig:igbm_experiment}
\end{figure}

In Figure \ref{fig:igbm_experiment}, we see that the log-ODE method with the proposed space-space-time L\'{e}vy area estimator (\ref{eq:sst_igbm}) is the most accurate -- for both strong and weak convergence.
However, to confirm this, we will also need take computational cost into consideration.

\begin{table}[ht]
    \centering
    \begin{tabular}{@{}ccccccc@{}}
    \toprule
      & Log-ODE  & Parabola-ODE & Linear-ODE  & Milstein & Euler \\ \midrule
   Computational time (s) & 2.44 & 2.95 & 1.48 & 1.18 & \textbf{1.17} \\
     \bottomrule
    \end{tabular}\medbreak
    \caption{Time taken by each numerical method to generate 100,000 sample paths of IGBM with 100 steps per path
(using a single-threaded C++ program ran on a desktop computer).}\label{tab:igbm_times}
\end{table}

From Table \ref{tab:igbm_times}, we see that the log-ODE method takes roughly twice as long to generate sample paths as the linear-ODE, Milstein and Euler-Maruyama methods. This is expected as each log-ODE step requires two Gaussian random variables to be generated, $W_{t_n,t_{n+1}}$ and $H_{t_n,t_{n+1}}$, whereas the standard schemes just require $W_{t_n,t_{n+1}}$.\medbreak

However, after taking computational costs into consideration, it is still clear that the log-ODE method is more efficient than the other numerical methods (see Table \ref{tab:igbm_time_to_acheive}).\vspace{-2mm}
\begin{table}[ht]
    \centering
    \begin{tabular}{@{}ccccccc@{}}
    \toprule
     \hspace*{-2.5mm} & Log-ODE  & Parabola-ODE & Linear-ODE  & Milstein & Euler \\ \midrule
   Estimated time to achieve a\hspace*{-1.75mm} & \textbf{0.179} & 0.405 & 1.47 & 15.4 & 0.437 \\
   strong error of $S_N = 10^{-4}$ & (s) & (s) & (s) & (s) & (days)\\[4pt]
     Estimated time to achieve a\hspace*{-1.75mm} & \textbf{0.827} & 3.90 & 14.9 & 157 & 61.7 \\
   strong error of $S_N = 10^{-5}$  & (s) & (s) & (s) & (s) & (days)\\[2pt]
   	 \midrule\\[-8pt]
   	 Estimated time to achieve a\hspace*{-1.75mm} & $\boldsymbol{<}\mathbf{0.240}$ & 1.69 & 2.15 & 2.78 & 25.5 \\
   weak error of $E_N = 10^{-5}$\,\,\m & (s) & (s) & (s) & (s) & (s)\\[4pt]
     Estimated time to achieve a\hspace*{-1.75mm} & \textbf{0.240} & 16.9 & 21.6 & 24.1 & 252 \\
   weak error of $E_N = 10^{-6}$\,\,\m & (s) & (s) & (s) & (s) & (s)\\[2pt]
     \bottomrule
    \end{tabular}\medbreak
    \caption{Estimated times to generate 100,000 sample paths that achieve a
given accuracy (again using a single-threaded C++ program ran on a desktop computer).}\label{tab:igbm_time_to_acheive}
\end{table}\vspace*{-1mm}

As IGBM is an ergodic SDE with a known stationary distribution (inverse gamma), it is important to investigate properties of the numerical approximation $Y_n$ as $n\rightarrow\infty$.
In \cite{tubikanec2022igbm}, the authors consider the long-term mean and variance of four splitting methods, as well as the log-ODE, linear-ODE, Milstein and Euler-Maruyama methods (Table \ref{tab:igbm_methods}).\\ \noindent
Since \cite{tubikanec2022igbm} conducted very thorough numerical testing, we will only present the results from their final table -- which will require us to define the following weak error metrics.
\medbreak
\begin{definition} Let $\{Y_n\}_{n\m\geq\m 0}$ be an approximation of the SDE (\ref{eq:IGBM}) with step size $h$. Then, following \cite{tubikanec2022igbm}, we define the relative conditional mean and variance biases as
\begin{align*}
\mathrm{rBias}_{h, y_0}\big(\E\big[Y_n\big]\big) & :=  \frac{\E[Y_n\m|\m Y_0 = y_0] - \E\big[\m y(t_n)\m|\m y_0\big]}{\E\big[\m y(t_n)\m|\m y_0\big]},\\[2pt]
\mathrm{rBias}_{h, y_0}\big(\var\big(Y_n\big)\big) & :=  \frac{\var(Y_n\m|\m Y_0 = y_0) - \var\big(\m y(t_n)\m|\m y_0\big)}{\var\big(\m y(t_n)\m|\m y_0\big)},
\end{align*}
and the relative asymptotic mean and variance biases as
\begin{align*}
\mathrm{rBias}_{h}\big(\E\big[Y_n\big]\big) & :=  \frac{\E[Y_n] - \E[\m y_\infty]}{\E[\m y_\infty]},\\[2pt]
\mathrm{rBias}_{h}\big(\var\big(Y_n\big)\big) & :=  \frac{\var(Y_n) - \var(\m y_\infty)}{\var(\m y_\infty)},
\end{align*}
where $\displaystyle \E[\m y_\infty] := \lim_{t\rightarrow\infty} \E[\m y_t]$ and $\displaystyle \var(\m y_\infty) := \lim_{t\rightarrow\infty} \var(\m y_t)$, which are known for IGBM.
\end{definition}\newpage

The experimental results from \cite[Table 3]{tubikanec2022igbm} are presented in Tables \ref{tab:igbm_mean_var1}, \ref{tab:igbm_mean_var2} and \ref{tab:igbm_mean_var3}.
Here, IGBM is numerical simulated using the parameters: $a = \frac{1}{5}, b = 5, \sigma  = 0.2$ and $y_0 = 10$ (for Tables \ref{tab:igbm_mean_var1} and \ref{tab:igbm_mean_var2}) and  $a = \frac{1}{5}, b = 5, \sigma  = 0.55, y_0 = 10$ (for Table \ref{tab:igbm_mean_var3} only).\vspace*{-2mm}
\begin{table}[ht]
\centering
\begin{tabular}{ccccccccccc}
\toprule
\multirow{3}{*}{$h = 0.5$} & \multicolumn{4}{c}{Relative bias at $t_n = 15$ (\%)} & \multicolumn{4}{c}{Relative bias at $t_n = 100$ (\%)} \\[1pt]
& \multicolumn{2}{c}{Conditional} & \multicolumn{2}{c}{Asymptotic} & \multicolumn{2}{c}{Conditional} & \multicolumn{2}{c}{Asymptotic} \\
 & Mean & Var & Mean & Var & Mean & Var & Mean & Var \\
\midrule\\[-9pt]
Euler  &  -\m$0.694$ & $4.425$ & -\m$0.705$ & $4.4$ & $\textbf{-\m 0.001}$ & $5.897$ & \textbf{0} & $5.882$ \\[1pt]
Milstein  &  -\m$0.694$ & $5.602$ & -\m$0.705$ & $5.586$ & $\textbf{-\m 0.001}$ & $7.092$ & \textbf{0} & $7.067$ \\[1pt]
L1  &  -\m$4.44$ & $0.81$ & -\m$4.45$ & $0.786$ & -\m$4.917$ & -\m$0.191$ &  -\m$4.917$ & -\m$0.216$ \\[1pt]
L2  &  $4.611$ & -\m$1.163$ & $4.601$ & -$1.187$ & $5.083$ & -\m$0.191$ & $5.083$ & -\m$0.216$ \\[1pt]
S1  &  $0.085$ & -\m$0.18$ & $0.075$ & -$0.205$ & $0.083$ & -\m$0.191$ & $0.083$ & -\m$0.216$ \\[1pt]
S2  &  -\m$0.039$ & $0.228$ & -\m$0.038$ & $0.204$ & -\m$0.025$ & $0.281$ & -\m$0.042$ & $0.233$ \\[1pt]
Linear-ODE  &  -\m$0.152$ & -\m$0.33$ & -$0.151$ & -\m$0.335$ & -\m$0.0167$ & -\m$0.391$ & -\m$0.166$ & -\m$0.415$ \\[1pt]
Log-ODE  &  $\textbf{-\m 0.008}$ & $\textbf{-\m 0.094}$ & $\textbf{-\m 0.0001}$ & $\textbf{-\m 0.001}$ & $0.01$ & $\textbf{0.07}$ & -\m$0.0001$ & $\textbf{-\m 0.001}$ \\
\bottomrule
\end{tabular}\medbreak
\caption{Relative mean and variance biases (given as a percentage) for different IGBM methods where L1, L2 are Lie{\m}-Trotter splitting schemes and S1, S2 are Strang splitting schemes (see \cite{tubikanec2022igbm}).}
\label{tab:igbm_mean_var1}\vspace{3.5mm}
\begin{tabular}{ccccccccccc}
\toprule
\multirow{3}{*}{$h = 1.0$} & \multicolumn{4}{c}{Relative bias at $t_n = 15$ (\%)} & \multicolumn{4}{c}{Relative bias at $t_n = 100$ (\%)} \\[1pt]
& \multicolumn{2}{c}{Conditional} & \multicolumn{2}{c}{Asymptotic} & \multicolumn{2}{c}{Conditional} & \multicolumn{2}{c}{Asymptotic} \\
 & Mean & Var & Mean & Var & Mean & Var & Mean & Var \\
\midrule\\[-9pt]
Euler  &  -\m$1.384$ & $9.315$ & -\m$1.391$ & $9.296$ & $\textbf{0.006}$ & $12.461$ & \textbf{0} & $12.5$ \\[1pt]
Milstein  &  -\m$1.383$ & $11.796$ & -$1.391$ & $11.807$ & $\textbf{0.003}$ & $14.957$ & \textbf{0} & $15.038$ \\[1pt]
L1  &  -\m$8.742$ & $1.176$ & -\m$8.75$ & $1.169$ & -\m$9.663$ & -\m$0.921$ & -\m$9.667$ & -\m$0.862$ \\[1pt]
L2  &  $9.361$ & -$2.762$ & $9.353$ & -\m$2.762$ & $10.337$ & -\m$0.921$ & $10.333$ & -\m$0.862$ \\[1pt]
S1  &  $0.31$ & -$0.808$ & $0.302$ & -\m$0.815$ & $0.337$ & -\m$0.921$ & $0.333$ & -\m$0.862$ \\[1pt]
S2  &  -\m$0.129$ & $0.91$ & -\m$0.151$ & $0.814$ & -\m$0.164$ & $0.879$ & -\m$0.166$ & $0.928$ \\[1pt]
Linear-ODE  &  -\m$0.286$ & -\m$0.745$ & -\m$0.301$ & -\m$0.802$ & -\m$0.329$ & -\m$1.05$ & -\m$0.332$ & -\m$0.992$ \\[1pt]
Log-ODE  &  $\textbf{0.003}$ & $\textbf{-\m 0.005}$ & $\textbf{-\m 0.0002}$ & $\textbf{-\m 0.003}$ & $0.009$ & $\textbf{0.019}$ & -\m$0.0002$ & $\textbf{-\m 0.004}$ \\
\bottomrule
\end{tabular}\medbreak
\caption{Relative mean and variance biases (given as a percentage) for different IGBM methods where L1, L2 are Lie{\m}-Trotter splitting schemes and S1, S2 are Strang splitting schemes (see \cite{tubikanec2022igbm}).}
\label{tab:igbm_mean_var2}\vspace{3.5mm}
\begin{tabular}{cccccccccc}
\toprule
\multirow{2}{*}{\begin{tabular}{c} $1000\times\text{KL}$ \\ with step size\end{tabular}}  & \multirow{2}{*}{Euler} & \multirow{2}{*}{Milstein} & \multirow{2}{*}{L1} & \multirow{2}{*}{L2} & \multirow{2}{*}{S1} & \multirow{2}{*}{S2} & \multirow{2}{*}{\begin{tabular}{c} Linear \\ ODE\end{tabular}}  & \multirow{2}{*}{Log-ODE}\\ \\
\midrule\\[-10pt]
$h = 0.5$  & 0.925 & 1.124  & 0.194 & 0.445 & 0.005 & 0.009 & 0.045 & \textbf{0.003}  \\[1pt]
$h = 1.0$  & 5.139 & 3.743 & 0.639 & 2.981 & 0.069 & 0.071 & 0.208 & \textbf{0.001}  \\
\bottomrule
\end{tabular}\medbreak
\caption{Kullback-Leibler (KL) divergence between the distributions of $Y_n$ and $y_\infty$ estimated for different numerical methods at $t_n = 100$. For IGBM, the stationary distribution is inverse Gamma. The densities of $Y_n$ are approximated by kernel density estimation, and we refer to \cite{tubikanec2022igbm} for details.}
\label{tab:igbm_mean_var3}
\end{table}\vspace*{-3.5mm}

Therefore, we see that the log-ODE method with the space-space-time L\'{e}vy area approximation (\ref{eq:sst_igbm}) gives the best performance across almost all weak error metrics.
For SDEs similar to IGBM, where computing first and second vector field derivatives is straightforward, we would expect the log-ODE method to be particularly effective.
However, due to the potential difficulties in computing second derivatives, high order splitting methods were developed in \cite{foster2024splitting} -- which can use the optimal estimator (\ref{eq:sst_splitting}).

\subsubsection{Noisy anharmonic oscillator}

In this section, we shall consider Runge-Kutta methods for the stochastic oscillator,
\begin{align}\label{eq:additive_sde}
dy_t = f(y_t)\m dt + \sigma\m dW_t\m,
\end{align}
where $f(x) = \sin(x)$ and $\sigma = y_0 = T = 1$. In particular, we will present an experiment\vspace{0.5mm} from \cite{foster2024splitting}, which tests the following numerical methods for SDEs with additive noise. \medbreak
\begin{definition}[The ``SRA1'' Runge-Kutta method for SDEs with additive noise \cite{rossler2010srk}]  For $N\geq 1$, we define a numerical solution $Y = \{Y_k\}_{0\m\leq\m k\m\leq\m N}$ of (\ref{eq:additive_sde}) by $Y_0 := y_0$ and\vspace*{-1mm}
\begin{align}\label{eq:SRA1}
\widetilde{Y}_{k+\frac{3}{4}} & := Y_k + \frac{3}{4}f(Y_k)h + \frac{3}{4}\sigma\big(W_{t_k, t_{k+1}} + 2H_{t_k, t_{k+1}}\big)\m,\nonumber\\[2pt]
Y_{k+1} & :=  Y_k + \frac{1}{3}f(Y_k)h + \frac{2}{3}f\big(\,\widetilde{Y}_{k+\frac{3}{4}}\big)h + \sigma W_{t_k, t_{k+1}}\m.
\end{align}
\end{definition}\medbreak
\begin{definition}[Shifted Ralston method \cite{foster2024splitting} based on the approximation (\ref{eq:sst_splitting}) of $L_{s,t}$] For $N\geq 1$, we define a numerical solution $Y = \{Y_k\}_{0\m\leq\m k\m\leq\m N}$ of (\ref{eq:additive_sde}) by $Y_0 := y_0$ and\vspace*{-1mm}
\begin{align}\label{eq:shifted_ralston}
\widetilde{Y}_k & := Y_k + \frac{1}{2}\sigma W_{t_k, t_{k+1}} + \sigma H_{t_k, t_{k+1}} - \frac{1}{2}\sigma\m C_{t_k, t_{k+1}}\m,\nonumber\\[2pt]
\widetilde{Y}_{k+\frac{2}{3}} & :=  \widetilde{Y}_k + \frac{2}{3}\Big(f\big(\,\widetilde{Y}_k\big)h + \sigma\m C_{t_k, t_{k+1}}\Big),\nonumber\\[2pt]
Y_{k+1} & := Y_k + \frac{1}{4}f\big(\,\widetilde{Y}_k\big)h + \frac{3}{4}f\big(\,\widetilde{Y}_{k+\frac{2}{3}}\big)h + \sigma W_{t_k, t_{k+1}}\m,
\end{align}
where each random vector $C_{t_k, t_{k+1}}$ is given by\vspace*{-2mm}
\begin{align*}
C_{t_k, t_{k+1}} & := \epsilon_{t_k, t_{k+1}}\Big(W_{t_k, t_{k+1}}^2 + \frac{12}{5}H_{t_k, t_{k+1}}^2 + \frac{4}{5}h - \frac{3}{\sqrt{6\pi}}\m h^\frac{1}{2} n_{t_k, t_{k+1}} W_{t_k, t_{k+1}}\Big)^\frac{1}{2},\\
\epsilon_{t_k, t_{k+1}} & := \sgn\bigg(W_{t_k, t_{k+1}} - \frac{3}{\sqrt{24\pi}}\m h^\frac{1}{2} n_{t_k, t_{k+1}}\bigg),\\[-20pt]
\end{align*}
and $n_{t_k, t_{k+1}}$ denotes the space-time swing of $W$ over $[t_k\m, t_{k+1}]$ given by Definition \ref{def:st_swing}.
\end{definition}\medbreak
\begin{remark}
Taylor expanding the method (\ref{eq:shifted_ralston}) produces the following $O(h^2)$ term:\vspace*{-1mm}
\begin{align*}
&\frac{1}{8}f^{\prime\prime}(Y_k)\bigg(\frac{1}{2}\sigma W_k + \sigma H_k - \frac{1}{2}\sigma C_k\bigg)^2 h + \frac{3}{8}f^{\prime\prime}(Y_k)\bigg(\frac{1}{2}W_k + H_k + \frac{1}{6}\sigma C_k\bigg)^2 h \\
&\mm = \sigma^2 f^{\prime\prime}(Y_k)\bigg(\frac{1}{8} W_k^2 + \frac{1}{2}W_k H_k + \frac{1}{2}H_k^2+ \frac{1}{24}C_k^2\bigg)h\\
&\mm = \sigma^2 f^{\prime\prime}(Y_k)\bigg(\frac{1}{6}h W_k^2 + \frac{1}{2}h W_k H_k + \underbrace{\frac{3}{5}h H_k^2 + \frac{1}{30}h^2 - \frac{1}{8\sqrt{6\pi}}\m h^\frac{3}{2} n_k W_k}_{=\,\, \E[\m L_{t_k,t_{k+1}}\,|\, W_k\m,\m H_k\m,\m n_k]}\bigg).
\end{align*}
Hence, by Corollary \ref{corr:simple_sst_levy}, we see that the proposed method (\ref{eq:shifted_ralston}) is asymptotically optimal.
For additional details regarding the shifted Ralston method, we refer the reader to \cite{foster2024splitting}.
\end{remark}\medbreak

In addition, the authors of \cite{foster2024splitting} also consider the standard Euler-Maruyama method and the following Euler-based method -- which has a smaller asymptotic error of $O(h^2)$.\medbreak

\begin{definition}[Shifted Euler method \cite{foster2024splitting}] We define $\{Y_k\}_{0\m\leq\m k\m\leq\m N}$ by $Y_0 := y_0$ and
\begin{align}\label{eq:shifted_euler}
Y_{k+1} := Y_k + f\bigg(Y_k + \frac{1}{2}\sigma W_{t_k, t_{k+1}} + \sigma H_{t_k, t_{k+1}}\bigg)h + \sigma W_{t_k, t_{k+1}}\m.
\end{align}
\end{definition}
\begin{remark}\label{rmk:shifted_euler}
Taylor expanding the shifted Euler method (\ref{eq:shifted_euler}) produces the following:
\begin{align*}
Y_{k+1} = Y_k + f(Y_k)h + \sigma W_{t_k, t_{k+1}} + \sigma f^\prime(Y_k)\bigg(\underbrace{\frac{1}{2}h W_{t_k, t_{k+1}} + h H_{t_k, t_{k+1}}}_{=\,\int_{t_k}^{t_{k+1}} W_{t_k, t}\, dt}\bigg) + O(h^2).\\[-20pt]
\end{align*}
Thus, the shifted Euler method produces an $O(h^2)$ local error whereas Euler-Maruyama has a local error of $O(h^\frac{3}{2})$, which is due to the above time integral of Brownian motion.
\end{remark}\medbreak

Since we would expect the two-stage Runge-Kutta methods to achieve $O(h^{1.5})$ strong convergence for (\ref{eq:additive_sde}), and the Euler-based methods to achieve $O(h)$ convergence,
the authors of \cite{foster2024splitting} present the ratio of the strong error estimators (\ref{eq:igbm_strong}) for different $h$.
Accompanying code can be found at \href{https://github.com/james-m-foster/high-order-splitting}{github.com/james-m-foster/high-order-splitting}.\vspace{-3mm}

\begin{figure}[ht]
\centering
\includegraphics[width=\textwidth]{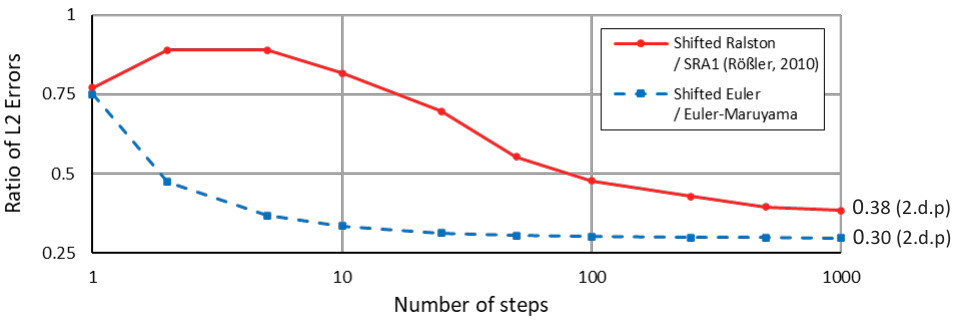}\vspace{-3mm}
\caption{$S_N$ estimated for (\ref{eq:additive_sde}) with 1,000,000 sample paths, where $N$ denotes the number of steps. To illustrate the differences in accuracy between methods, we plot the ratio $S_N^{\text{(method 1)}} / S_N^{\text{(method 2)}}$.}\label{fig:oscillator_convergence}
\end{figure}

From Figure \ref{fig:oscillator_convergence}, we can see that the shifted Ralston method is almost three times as accurate as the SRA1 scheme -- provided that the step size $h$ is sufficiently small.
Moreover, as shown in \cite{foster2024splittingarxiv}, the ratio of $0.38$ is remarkably close to the theoretical ratio\vspace{-2mm}
\begin{align*}
\frac{\E\bigg[\Big(\frac{1}{2}\int_{t_k}^{t_{k+1}} W_{t_k, t}^2\, dt - \E\Big[\frac{1}{2}\int_{t_k}^{t_{k+1}}  W_{t_k,t}^2\, dt\m\big|\m W_{t_k,t_{k+1}}\m, H_{t_k,t_{k+1}}\m, n_{t_k,t_{k+1}}\Big]\Big)^2\,\bigg]^{\frac{1}{2}}}{\E\bigg[\Big(\frac{1}{2}\int_{t_k}^{t_{k+1}}  W_{t_k,t}^2\, dt - \frac{3}{16}\big(W_{t_k,t_{k+1}} + 2\m H_{t_k,t_{k+1}}\big)^2\Big)^2\,\bigg]^{\frac{1}{2}}},\\[-20pt]
\end{align*}
which equals $\big(\frac{7}{30} - \frac{5}{16\pi}\big)^\frac{1}{2}=0.37$ (2.d.p). The terms in the above ratio correspond to\vspace{0.5mm} the space-space-time L\'{e}vy area approximations used by both Runge-Kutta methods.\medbreak

Similarly, we see that the shifted Euler method is roughly three times as accurate as the standard Euler-Maruyama method -- provided that at least 10 steps are used.
However, as the local errors of the shifted Euler and Euler-Maruyama methods scale differently with $h$ (see Remark \ref{rmk:shifted_euler}), it is unclear how to establish a theoretical ratio. Taking computational costs into consideration, we see that the shifted Euler method can clearly outperform Euler-Maruyama for simulating the stochastic oscillator (\ref{eq:additive_sde}).

\vspace{-2mm}
\begin{table}[ht]
    \centering
    \begin{tabular}{@{}ccccccc@{}}
    \toprule
      &\hspace{-5mm} \multirow{2}{*}{\begin{tabular}{c}Shifted Ralston\\ (\ref{eq:shifted_ralston})\end{tabular}}  & \hspace{-4mm}\multirow{2}{*}{\begin{tabular}{c}SRA1\\ \cite{rossler2010srk}\end{tabular}} & \hspace{-4mm}\multirow{2}{*}{\begin{tabular}{c}Shifted Euler\\ (\ref{eq:shifted_euler})\end{tabular}}  &  \hspace{-2.5mm}\multirow{2}{*}{Euler-Maruyama}  \\ \\ \midrule
     \hspace{-4mm}Time to simulate 100,000 paths & \hspace{-4mm}\multirow{2}{*}{3.16} & \hspace{-4mm}\multirow{2}{*}{2.29} & \hspace{-4mm}\multirow{2}{*}{1.91} & \hspace{-2.5mm}\multirow{2}{*}{\textbf{1.09}}  \\
    with 100 steps per sample path (s) &  &  &  & \\[2pt]
   	 \midrule\\[-8pt]
     \hspace{-8mm}Estimated time to achieve a & \hspace{-4mm}\multirow{2}{*}{\textbf{1.79}} &\hspace{-4mm}\multirow{2}{*}{2.00} & \hspace{-4mm}\multirow{2}{*}{23.1} & \hspace{-2.5mm}\multirow{2}{*}{47.0}  \\
   \hspace{-5mm}strong error of $S_N = 10^{-4}$ (s) & & & & \\[2pt]
     \bottomrule
    \end{tabular}\medbreak
    \caption{Computational times for generating 100,000 sample paths of (\ref{eq:additive_sde}) with (a) 100 steps per path and (b) to achieve a given accuracy. All of the numerical methods were implemented in C++.}\label{tab:oscillator_times}
\end{table}\vspace*{-2mm}

Due to the computational cost of computing $C_{t_k,t_{k+1}}$ in the shifted Ralston method, we can see in Table \ref{tab:oscillator_times} that it is only slightly more efficient than the SRA1 scheme.
However, for SDEs where the drift function $f$ is much more expensive to evaluate than $C_{t_k,t_{k+1}}$, we would certainly expect the shifted Ralston method to outperform SRA1.\medbreak

For general multidimensional SDEs with additive noise, the shifted Ralston method is no longer asymptotically efficient due to the integrals $\int_{t_k}^{t_{k+1}} W_{t_k, t}^i W_{t_k, t}^j\m dt$ with $i\neq j$.\vspace{0.5mm}
Whilst we know how to optimally approximate such integrals (Theorems \ref{thm:sst_approx_whk} and \ref{thm:sst_with_swing}),
it is not clear how these estimators translate to derivative-free Runge-Kutta methods.
We will leave the development of such Runge-Kutta methods as a topic of future work.

\subsubsection{FitzHugh-Nagumo model}

In this section, we shall present our final example, which is a stochastic FitzHugh-Nagumo (FHN) model used to describe the spiking activity of single neurons \cite{buckwar2022splitting, leon2018fitzhugh}. This FHN model is given by the following SDE with two-dimensional additive noise:\vspace{-1mm}
\begin{align}\label{eq:fitzhugh}
dv_t & = \frac{1}{\varepsilon}\big(v_t - v_t^3 - u_t\big)\m dt + \sigma_1\m dW_t^1,\\
du_t & = \big(\gamma v_t - u_t + \beta\big)\m dt + \sigma_2\m dW_t^2,\nonumber\\[-18pt]\nonumber
\end{align}
where $\{v_t\}_{t\geq 0}$ denotes the membrane voltage of the neuron, $\{u_t\}_{t\geq 0}$ is an additional recovery variable, $\varepsilon > 0$ represents the time scale between the $v$ and $u$, and the parameters $\beta\geq 0$ and $\gamma > 0$ correspond to the position and duration of an excitation.\medbreak
The stochastic FHN model is quite challenging to simulate due to the $v^3$ term, which is not globally Lipschitz continuous on $\R$. However, the function $f(x) = x - x^3$ does instead satisfy a one-sided Lipschitz condition $(x-y)(f(x) - f(y)) \leq (x-y)^2$.\vspace{0.5mm}
Thus, as the drift vector field of the FHN model also has polynomial growth, there are numerical methods in the literature with strong convergence guarantees for (\ref{eq:fitzhugh}). We will present the FHN experiment in \cite{foster2024splitting}, where two such schemes are considered; the Strang splitting from \cite{buckwar2022splitting} and Tamed Euler-Maruyama method proposed in \cite{hutzenthaler2012}.\medbreak

To build upon the previous work of \cite{buckwar2022splitting} on Strang splittings, the authors of \cite{foster2024splitting} proposed a high order splitting method -- which is based on the optimal estimator (\ref{eq:sst_splitting}).
For the same class of SDEs with locally Lipschitz drift, we expect that the analysis conducted for Strang splitting in \cite{buckwar2022splitting} should be applicable to this high order splitting.
However, the method (\ref{eq:fitzhugh_splitting}) is quite involved, so we expect this line of research to be more suitable for the simpler ``high order Strang splitting'' scheme proposed in \cite{foster2024splitting}.
We would refer the reader to \cite{foster2024splitting} for further details regarding these splitting methods.
\medbreak
\begin{definition}[High order splitting scheme for the FHN model]
For a fixed $N\geq 1$, we define a numerical solution $\{(V_k\m, U_k)\}_{0\m\leq\m k\m\leq\m N}$ of (\ref{eq:fitzhugh}) by $(V_0\m, U_0) := (v_0\m, u_0)$ and\vspace*{-1mm}
\begin{align}
\begin{pmatrix}V_k^{(1)}\\[3pt] U_k^{(1)}\end{pmatrix} & := \begin{pmatrix}V_k\\[3pt] U_k\end{pmatrix} + \begin{pmatrix}\frac{1}{2}\sigma_1 W_{t_k, t_{k+1}}^{1} + \sigma_1 H_{t_k, t_{k+1}}^{1} - \frac{1}{2} \sigma_1\m\overline{C}_{t_k, t_{k+1}}^{1}\\[3pt] \frac{1}{2}\sigma_2 W_{t_k, t_{k+1}}^{2} + \sigma_2 H_{t_k, t_{k+1}}^{2} - \frac{1}{2}\sigma_2\m \overline{C}_{t_k, t_{k+1}}^{2}\end{pmatrix},\nonumber\\[6pt]
\begin{pmatrix}V_k^{(2)}\\[3pt] U_k^{(2)}\end{pmatrix} & := \varphi_{\frac{1}{2}h}^{\text{Strang}}\begin{pmatrix}V_k^{(1)}\\[3pt] U_k^{(1)}\end{pmatrix} + \begin{pmatrix} \sigma_1\overline{C}_{t_k, t_{k+1}}^{1}\\[3pt] \sigma_2\m\overline{C}_{t_k, t_{k+1}}^{2}\end{pmatrix},\nonumber\\[6pt]
\begin{pmatrix}V_{k+1}\\[3pt] U_{k+1}\end{pmatrix} & := \varphi_{\frac{1}{2}h}^{\text{Strang}}\begin{pmatrix}V_k^{(2)}\\[3pt] U_k^{(2)}\end{pmatrix} + \begin{pmatrix}\frac{1}{2}\sigma_1 W_{t_k, t_{k+1}}^{1} - \sigma_1 H_{t_k, t_{k+1}}^{1} - \frac{1}{2} \sigma_1\m\overline{C}_{t_k, t_{k+1}}^{1}\\[3pt] \frac{1}{2}\sigma_2 W_{t_k, t_{k+1}}^{2} - \sigma_2 H_{t_k, t_{k+1}}^{2} - \frac{1}{2}\sigma_2\m \overline{C}_{t_k, t_{k+1}}^{2}\end{pmatrix},\label{eq:fitzhugh_splitting}
\end{align}
where each random vector $\overline{C}_{t_k, t_{k+1}}^{\m i}$ is given by\vspace{-2mm}
\begin{align*}
\overline{C}_{t_k, t_{k+1}}^{\m i} & := \epsilon_{t_k, t_{k+1}}^i\bigg(\frac{1}{3}\big(W_{t_k, t_{k+1}}^i\big)^2 + \frac{4}{5}\big(H_{t_k, t_{k+1}}^i\big)^2 + \frac{4}{15}h - \frac{1}{\sqrt{6\pi}}\m h^\frac{1}{2} n_{t_k, t_{k+1}}^i W_{t_k, t_{k+1}}^i\bigg)^\frac{1}{2},\\
\epsilon_{t_k, t_{k+1}}^{\m i} & := \sgn\bigg(W_{t_k, t_{k+1}}^i - \frac{3}{\sqrt{24\pi}}\m h^\frac{1}{2} n_{t_k, t_{k+1}}^i\bigg),\\[-20pt]
\end{align*}
and $\varphi_t^{\text{Strang}}$ denotes the Strang splitting proposed by \cite{buckwar2022splitting} for (\ref{eq:fitzhugh}) when $\sigma_1 = \sigma_2 = 0$,
\begin{align*}
\varphi_{\frac{1}{2}h}^{\text{Strang}} & := \varphi_{\frac{1}{4}h}^{(\frac{1}{\varepsilon}(v-v^3),\m \beta)}\circ\varphi_{\frac{1}{2}h}^{\text{linear}}\circ\varphi_{\frac{1}{4}h}^{(\frac{1}{\varepsilon}(v-v^3),\m \beta)},
\end{align*}
where, for vectors $u, v\in\R$, we define the ODE solutions $\varphi_t^{(\frac{1}{\varepsilon}(v-v^3),\m \beta)}$ and $\varphi_t^{\text{linear}}$ by
\begin{align*}
\varphi_t^{(\frac{1}{\varepsilon}(v-v^3),\m \beta)}\begin{pmatrix}v\\[3pt] u\end{pmatrix} & := \begin{pmatrix}v \Big(e^{-\frac{2t}{\varepsilon}} + v^2\big(1 - e^{-\frac{2t}{\varepsilon}}\big)\Big)^{-\frac{1}{2}}\\ u + \beta t\end{pmatrix},\\
\varphi_t^{\text{linear}}\begin{pmatrix}v\\[3pt] u\end{pmatrix} & := \exp\Bigg(t\begin{pmatrix}0 & -\frac{1}{\varepsilon}\\[3pt] \gamma & -1\end{pmatrix}\Bigg)\begin{pmatrix}v\\[3pt] u\end{pmatrix},
\end{align*}
and the explicit formula for the above matrix exponential is given in \cite[Section 6.2]{buckwar2022splitting}.
\end{definition}\medbreak
\begin{remark}
Similar to the shifted Ralston method (\ref{eq:shifted_ralston}), $\overline{C}_{t_k,t_{k+1}}$ is designed so that\vspace{-2mm}
\begin{align*}
&\frac{1}{4}h\Big(\frac{1}{2} W_k^i + H_k^i - \frac{1}{2}\overline{C}_k^{\m i}\Big)^2 + \frac{1}{4}h\Big(\frac{1}{2} W_k^i + H_k^i + \frac{1}{2}\m\overline{C}_k^{\m i}\Big)^2\\[2pt]
&\mm = \frac{1}{8}h \big(W_k^i\big)^2 + \frac{1}{2}h W_k^i H_k^i + \frac{1}{2}h \big(H_k^i\big)^2 +  \frac{1}{8}h \big(\m\overline{C}_k^{\m i}\big)^2\\[2pt]
&\mm = \frac{1}{6}h \big(W_k^i\big)^2 + \frac{1}{2}h W_k^i H_k^i + \underbrace{\frac{3}{5}h  \big(H_k^i\big)^2  + \frac{1}{30}h^2 - \frac{1}{8\sqrt{6\pi}}\m h^\frac{3}{2} n_k^i W_k^i}_{=\,\, \E[\m L_{t_k,t_{k+1}}^{ii}\,|\, W_k\m,\m H_k\m,\m n_k]}.\\[-24pt]
\end{align*}
Since the $\frac{\partial^2}{\partial u\partial v}$ derivative of the FHN drift is zero, we do not need to approximate\vspace{0.5mm} $L_{t_k\m, t_{k+1}}^{ij}$ for $i\neq j$ and therefore, the splitting method (\ref{eq:fitzhugh_splitting}) is asymptotically efficient.
\end{remark}\bigbreak

We now present the results of the experiment from \cite{foster2024splitting}, where the FHN model was simulated with the aforementioned numerical methods and the following parameters:\vspace{-2mm}
\begin{align*}
\varepsilon=1,\mm\gamma=1,\mm\beta=1,\hspace{5mm}\sigma_1 = 1, \mm\sigma_2=1,\mm(v_0\m, u_0) = (0,0),\mm T=5.\\[-20pt]
\end{align*}
Accompanying code can be found at \href{https://github.com/james-m-foster/high-order-splitting}{github.com/james-m-foster/high-order-splitting}.\vspace{-6.5mm}
\begin{figure}[ht] \label{fig:fitzhugh_convergence}
\centering
\includegraphics[width=0.9\textwidth]{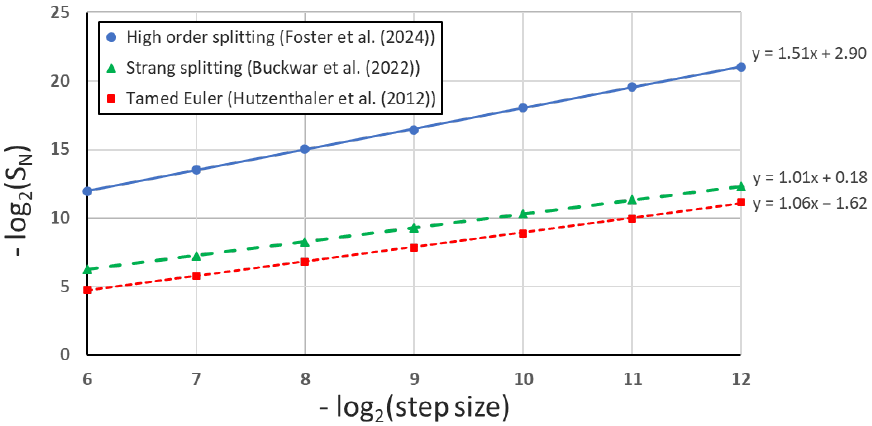}\vspace{0.5mm}
\caption{$S_N$ estimated for the FHN model (\ref{eq:fitzhugh}) using 1,000 sample paths as a function of $h = \frac{T}{N}$.\vspace{0.75mm} The estimated strong errors for the Strang splitting and Tamed Euler schemes were taken from \cite{buckwar2022splitting}.}
\end{figure}\vspace{-5mm}
\begin{table}[ht]
    \centering
    \begin{tabular}{@{}cccccc@{}}
    \toprule
      &\hspace{-2mm} \multirow{2}{*}{\begin{tabular}{c}High order\\ splitting\end{tabular}}  & \hspace{-2mm}\multirow{2}{*}{\begin{tabular}{c}Strang Splitting\\ {\cite{buckwar2022splitting}}\end{tabular}} & \hspace{-4mm}\multirow{2}{*}{\begin{tabular}{c}Tamed Euler-Maruyama\\ {\cite{hutzenthaler2012}}\end{tabular}}\hspace{-2mm}  \\ \\ \midrule
     Time to simulate 1,000 paths & \hspace{-2mm}\multirow{2}{*}{8.15} & \hspace{-2mm}\multirow{2}{*}{2.66} & \hspace{-4mm}\multirow{2}{*}{\textbf{1.71}}  \\
    \hspace{-2mm}with 100 steps per path (s) &  &  &  \\[1pt]
   	 \midrule\\[-9pt]
     \hspace{-1.5mm}Estimated time to achieve a & \hspace{-2mm}\multirow{2}{*}{\textbf{10.4}} &\hspace{-2mm}\multirow{2}{*}{110} & \hspace{-4mm}\multirow{2}{*}{166} \\
   \hspace{1.5mm}strong error of $S_N = 10^{-3}$ (s) & & & \\[1pt]
     \bottomrule
    \end{tabular}\medbreak
    \caption{Computational times for generating 1,000 sample paths of (\ref{eq:fitzhugh}) with (a) 100 steps per path and (b) to achieve a given accuracy. All numerical methods were implemented using Python.}\label{tab:fitzhugh_times}
\end{table}

From Figure \ref{fig:fitzhugh_convergence}, we see that the splitting method (\ref{eq:fitzhugh_splitting}) exhibits an $O(h^\frac{3}{2})$ strong convergence rate and is significantly more accurate than other schemes (for fixed $h$).
For example, the high order splitting method achieves a better accuracy in 320 steps than both the Strang splitting and Tamed Euler approximations do in 10240 steps. 
Finally, after taking computational costs into account, we see in Table \ref{tab:fitzhugh_times} that the high order splitting achieves a strong error of $10^{-3}$ in an order of magnitude less time.\medbreak

Similar to the previous subsection, the proposed splitting method was designed to use the ``diagonal'' space-space-time L\'{e}vy area estimator, $\E[\m L_{s,t}^{ii}\,|\, W_{s,t}\m,\m H_{s,t}\m,\m n_{s,t}]$.
However, a substantially simpler high order Strang splitting was also proposed in \cite{foster2024splitting} and, after Taylor expanding this splitting, we can obtain the following approximation:\vspace{-1mm}
\begin{align}\label{eq:strang_approx_of_sst}
L_{s,t}^{ij} \approx \frac{2-\sqrt{3}}{24} h W_{s,t}^{i} W_{s,t}^{j} + \frac{\sqrt{3}}{2}h H_{s,t}^{i} H_{s,t}^{j}\m,\\[-20pt]\nonumber
\end{align}
which is reasonably close to the estimator $\m\E[\m L_{s,t}^{ij}\,|\, W_{s,t}\m,\m H_{s,t}] = \frac{3}{5}h H_{s,t}^i H_{s,t}^j + \frac{1}{30}\m h^2\delta_{ij}$,\vspace{0.5mm}
but clearly suboptimal. 
Thus, we expect that the development of high order numerical methods for high-dimensional additive-noise SDEs will be a topic for future research.

\section{Conclusion}\label{sect:conclusion}

In this paper, we have considered the signature of multidimensional Brownian motion with time and presented recent approximations for the non-Gaussian iterated integrals\vspace{-1mm}
\begin{align}\label{eq:conclude_integrals1}
\int_s^t W_{s,u}^i\, dW_u^j\m,\hspace{5mm}\int_s^t W_{s,u}^i W_{s,u}^j\m du,\\[-17pt]\nonumber
\end{align}
in terms of the following Gaussian iterated integrals (which can be generated exactly),\vspace{-1mm}
\begin{align}\label{eq:conclude_integrals2}
\int_s^t dW_u\m, &\int_s^t\int_s^u dW_v\, du, \int_s^t\int_s^u dv\, dW_u\m,\\[2pt] \int_s^t\int_s^u\int_s^v dW_r\, dv\, du\m, &\int_s^t\int_s^u\int_s^v dr\, dW_v\, du\m, \int_s^t\int_s^u\int_s^v dr\, dv\, dW_u\m.\nonumber\\[-17pt]\nonumber
\end{align}

The proposed methodologies for approximating the integrals in (\ref{eq:conclude_integrals1}) are heavily based on computing conditional moments with respect to the Gaussian integrals (\ref{eq:conclude_integrals2}).
For the well-known ``L\'{e}vy area'' of Brownian motion (corresponding to $\int_s^t W_{s,u} \otimes\m dW_u$), we propose a random matrix that matches several moments of $A_{s,t}$ (see Theorem \ref{thm:match_moments}),
but built from standard distributions (normal, uniform, etc) and thus fast to generate. 
In \cite{jelincic2025levygan}, this approach was shown to empirically outperform the standard Fourier series approximation and was competitive with a recently developed machine learning model.\medbreak

Since these integrals are all entries in the signature of Brownian motion, they also naturally appear in the Taylor expansions of stochastic differential equations (SDEs).
Therefore, in section \ref{sect:examples}, we presented several experiments from the literature where moment-based integral approximations are used to improve accuracy for SDE solvers.

\noindent
In particular, we observed that approximating the integrals (\ref{eq:conclude_integrals1}) by their conditional expectations resulted in higher order methods with state-of-the-art convergence rates.
As a topic of future work, we would like to consider the ``adaptive'' Langevin dynamics:\vspace{-1mm}
\begin{align}\label{eq:adaptive_uld}
dx_t & = M^{-1}v_t\m dt,\\[2pt]
dv_t & = - \m\xi_t M^{-1}v_t\m dt\m -\nabla f(x_t)\m dt + \sigma\m dW_t\m,\nonumber\\
d\xi_t & = \frac{1}{\nu}\Big(v_t^{\top} M^{-2}\m v_t - \frac{1}{\beta}\mathrm{Trace}\big(M^{-1}\big)\Big) dt,\nonumber\\[-18pt]\nonumber
\end{align}
where $x,v\in\R^d$ represent the position and momentum of a particle, $M\in\R^{n\times n}$ is the $n\times n$ mass matrix, $f:\R^d\rightarrow\R$ denotes a scalar potential, $\sigma > 0$ is a noise parameter, $\beta > 0$ corresponds to the temperature, $\nu > 0$ governs the friction $\xi\in\R$, and $\{W_t\}_{t\geq 0}$ is a $d$-dimensional Brownian motion.
The adaptive Langevin dynamics (\ref{eq:adaptive_uld}) has been studied due to its ergodicity properties and applications to sampling problems \cite{jones2011adaptivelangevin, leimkuhler2020adaptivelangevin1, leimkuhler2020adaptivelangevin2}.
If we Taylor expand (\ref{eq:adaptive_uld}), then we only observe the integrals (\ref{eq:conclude_integrals1}) in the $\xi$ component:\vspace{-1mm}
\begin{align*}
\xi_t = \xi_s + (\,\cdots)\m h + (\,\cdots)\int_s^t W_{s,u}\,du + \frac{1}{\nu}\sigma^2\int_s^t \Big(\m W_{s,u}^{\top}\, M^{-2}\m W_{s,u}\Big)\m du + O(h^\frac{5}{2}).\\[-16pt]
\end{align*}
Since $\frac{1}{\nu}\sigma^2$ and $M^{-2}$ are constant, we can easily apply our integral estimators to (\ref{eq:adaptive_uld}).

\smallbreak

\bibliography{bibliography}

\end{document}